\documentclass[a4paper,11pt]{article}
\usepackage{amsmath,amsthm,amssymb,enumitem,xcolor}
\usepackage{tikz}

\usepackage[hang]{footmisc}
\usepackage[nosort,nocompress,noadjust]{cite}

\usepackage[bookmarks=false,hyperfootnotes=false,colorlinks,
    linkcolor={red!60!black},
    citecolor={blue!50!black},
    urlcolor={blue!80!black}]{hyperref}

\renewcommand{\eqref}[1]{\hyperref[#1]{(\ref{#1})}}

\newlist{enumlist}{enumerate}{1}
\setlist[enumlist]{labelindent=0cm,label=\arabic*.,labelwidth=2.5ex,labelsep=0.5ex,leftmargin=3ex,align=left,topsep=0.5ex,itemsep=1ex,parsep=1ex}

\newlist{itemlist}{itemize}{1}
\setlist[itemlist]{labelindent=0cm,label=$\bullet$,labelwidth=2.5ex,labelsep=0.5ex,leftmargin=3ex,align=left,topsep=0.5ex,itemsep=1ex,parsep=1ex}

\numberwithin{equation}{section}

{\theoremstyle{definition}\newtheorem{definition}{Definition}[section]
\newtheorem*{definition*}{Definition}

\newtheorem{remark}[definition]{Remark}

\newtheorem*{example*}{Example}
\newtheorem*{examples*}{Examples}}

\newtheorem{proposition}[definition]{Proposition}
\newtheorem{lemma}[definition]{Lemma}
\newtheorem{theorem}[definition]{Theorem}
\newtheorem{corollary}[definition]{Corollary}

\newtheorem{letterthm}{Theorem}
\newtheorem{lettercor}[letterthm]{Corollary}

{\theoremstyle{definition}\newtheorem{letterexs}[letterthm]{Examples}}

\newcommand{\bim}[3]{\mathord{\raisebox{-0.4ex}[0ex][0ex]{\scriptsize $#1$}{#2}\hspace{-0.25ex}\raisebox{-0.4ex}[0ex][0ex]{\scriptsize $#3$}}}

\renewcommand{\Re}{\operatorname{Re}}

\newcommand{\inter}{\operatorname{int}}
\newcommand{\Atil}{\widetilde{A}}
\newcommand{\Btil}{\widetilde{B}}
\newcommand{\Vtil}{\widetilde{V}}
\newcommand{\dom}{\operatorname{dom}}
\newcommand{\mutil}{\widetilde{\mu}}
\newcommand{\SL}{\operatorname{SL}}
\newcommand{\Sp}{\operatorname{Sp}}
\newcommand{\Isom}{\operatorname{Isom}}
\newcommand{\pitil}{\widetilde{\pi}}
\newcommand{\thetatil}{\widetilde{\theta}}
\newcommand{\gammatil}{\widetilde{\gamma}}
\newcommand{\omtil}{\widetilde{\omega}}
\newcommand{\rhotil}{\widetilde{\rho}}
\newcommand{\covol}{\operatorname{covol}}
\newcommand{\Htil}{\widetilde{H}}
\newcommand{\BS}{\operatorname{BS}}
\newcommand{\core}{\operatorname{c}}

\newcommand{\rtimesalg}{\rtimes_{\text{\rm alg}}}
\newcommand{\otmin}{\otimes_{\text{\rm min}}}
\newcommand{\Xtil}{\widetilde{X}}
\newcommand{\adap}{_{\text{\rm adap}}}
\newcommand{\C}{\mathbb{C}}

\newcommand{\eps}{\varepsilon}
\newcommand{\al}{\alpha}

\newcommand{\ot}{\otimes}
\newcommand{\recht}{\rightarrow}
\newcommand{\cb}{_\text{\rm cb}}
\newcommand{\Z}{\mathbb{Z}}
\newcommand{\vphi}{\varphi}
\newcommand{\op}{^\text{\rm op}}

\newcommand{\id}{\mathord{\text{\rm id}}}
\newcommand{\om}{\omega}

\newcommand{\N}{\mathbb{N}}
\newcommand{\cL}{\mathcal{L}}
\newcommand{\ovt}{\mathbin{\overline{\otimes}}}

\newcommand{\Tr}{\operatorname{Tr}}

\newcommand{\Om}{\Omega}

\newcommand{\si}{\sigma}
\newcommand{\R}{\mathbb{R}}

\newcommand{\F}{\mathbb{F}}
\newcommand{\cH}{\mathcal{H}}
\newcommand{\otalg}{\otimes_{\text{\rm alg}}}

\newcommand{\cZ}{\mathcal{Z}}
\newcommand{\Ad}{\operatorname{Ad}}

\newcommand{\cG}{\mathcal{G}}
\newcommand{\cK}{\mathcal{K}}

\newcommand{\T}{\mathbb{T}}
\newcommand{\actson}{\curvearrowright}

\newcommand{\cW}{\mathcal{W}}
\newcommand{\cU}{\mathcal{U}}

\newcommand{\Ker}{\operatorname{Ker}}
\newcommand{\cM}{\mathcal{M}}
\newcommand{\lspan}{\operatorname{span}}
\newcommand{\cN}{\mathcal{N}}
\newcommand{\cR}{\mathcal{R}}

\newcommand{\dpr}{^{\prime\prime}}

\newcommand{\cV}{\mathcal{V}}

\newcommand{\Aut}{\operatorname{Aut}}

\newcommand{\cP}{\mathcal{P}}

\newcommand{\cS}{\mathcal{S}}
\newcommand{\Prob}{\operatorname{Prob}}
\newcommand{\sN}{\mathcal{N}^{\text{\rm s}}}

\begin{document}

\begin{center}
{\boldmath\Large\bf Rigidity for von Neumann algebras given by\vspace{0.5ex}\\ locally compact groups and their crossed products}

\bigskip

{\sc by Arnaud Brothier\footnote{\noindent University Roma Tor Vergata, Department of Mathematics, Roma (Italy). E-mail: arnaud.brothier@gmail.com. AB is supported by European Research Council Advanced Grant 669240 QUEST.}, Tobe Deprez$^2$ and Stefaan Vaes\footnote{\noindent KU~Leuven, Department of Mathematics, Leuven (Belgium).\\ E-mails: tobe.deprez@kuleuven.be and stefaan.vaes@kuleuven.be. TD is supported by a PhD fellowship of the Research Foundation Flanders (FWO). SV is supported by European Research Council Consolidator Grant 614195 RIGIDITY, and by long term structural funding~-- Methusalem grant of the Flemish Government.}}

\bigskip

To appear in {\it Communications in Mathematical Physics}.
\end{center}

\begin{abstract}\noindent
We prove the first rigidity and classification theorems for crossed product von Neumann algebras given by actions of non-discrete, locally compact groups. We prove that for arbitrary free probability measure preserving actions of connected simple Lie groups of real rank one, the crossed product has a unique Cartan subalgebra up to unitary conjugacy. We then deduce a W$^*$ strong rigidity theorem for irreducible actions of products of such groups. More generally, our results hold for products of locally compact groups that are nonamenable, weakly amenable and that belong to Ozawa's class $\cS$.
\end{abstract}

\section{Introduction and statement of the main results}

Popa's deformation/rigidity theory has lead to a wealth of classification, rigidity and structural theorems for von Neumann algebras, and especially for II$_1$ factors arising from countable groups and their actions on probability spaces, through the group von Neumann algebra and the group measure space construction of Murray and von Neumann. We refer to \cite{Po06,Va10,Io12,Va16} for an introduction to deformation/rigidity theory. The main goal of this article is to prove rigidity and classification theorems for crossed products by actions of non-discrete, locally compact groups.

The classification problem for II$_1$ factors $M$ given as crossed products $M = L^\infty(X) \rtimes \Gamma$ for free ergodic probability measure preserving (pmp) actions of countable groups splits into two separate problems: the uniqueness problem for the Cartan subalgebra $L^\infty(X)$ and the classification problem for $\Gamma \actson (X,\mu)$ up to orbit equivalence. Striking progress has been made on both problems. In \cite{OP07}, it is proved that for \emph{profinite} free ergodic pmp actions of the free groups $\F_n$, the crossed product $M$ has a unique Cartan subalgebra up to unitary conjugacy. In \cite{CS11}, it was shown that the same holds for profinite actions of nonelementary hyperbolic groups and actually for profinite actions of nonamenable, weakly amenable groups in Ozawa's class $\cS$ introduced in \cite{Oz03,Oz04}. For \emph{arbitrary} free ergodic pmp actions of the same groups, the uniqueness of the Cartan subalgebra was established in \cite{PV11,PV12}.

The first goal of this paper is to prove that also for \emph{locally compact} groups that are nonamenable, weakly amenable and in class $\cS$, crossed products $M = L^\infty(X) \rtimes G$ by arbitrary free ergodic pmp actions have a unique Cartan subalgebra up to unitary conjugacy. This class of groups includes all rank one simple Lie groups, as well as all locally compact groups that admit a continuous and metrically proper action on a tree, or on a hyperbolic graph (see Proposition \ref{prop.groups-prop-S}). The precise definition of \emph{property (S)} goes as follows.

%

\begin{definition*}
Let $G$ be a locally compact group and denote $\cS(G) = \{ F \in L^1(G) \mid F(g) \geq 0$ $\text{for a.e.\ } g \in G \text{ and } \|F\|_1 = 1\}$. Equip $\cS(G)$ with the topology induced by the $L^1$-norm. We say that $G$ has \emph{property~(S)} if there exists a continuous map $\eta : G \recht \cS(G)$ satisfying
\begin{equation}\label{eq.QH-G}
\lim_{k \recht \infty} \| \eta(gkh) - g \cdot \eta(k) \|_1 = 0 \quad\text{uniformly on compact sets of $g,h \in G$.}
\end{equation}
\end{definition*}

By \cite[Proposition 15.2.3]{BO08}, Ozawa's class $\cS$ (see \cite{Oz04}) consists of all countable groups $\Gamma$ that are \emph{exact} and that have property~(S).

Our uniqueness of Cartan theorem can then be stated as follows. In Section \ref{sec.proof-unique-cartan}, we actually prove a more general result, also valid for nonsingular actions (see Theorem \ref{thm.unique-Cartan-general}) and thus generalizing the results in \cite{HV12} to the locally compact setting.

\begin{letterthm}\label{thm.unique-Cartan}
Let $G = G_1 \times \cdots \times G_n$ be a direct product of nonamenable locally compact second countable (lcsc) weakly amenable groups with property~(S). Let $G \actson (X,\mu)$ be an essentially free pmp action.

Then $L^\infty(X) \rtimes G$ has a unique Cartan subalgebra up to unitary conjugacy.
\end{letterthm}

To understand Theorem \ref{thm.unique-Cartan}, note that if $G$ is non-discrete, then $L^\infty(X)$ is not a Cartan subalgebra of $M$, but there is a canonical Cartan subalgebra given by choosing a cross section for $G \actson (X,\mu)$ (see Section \ref{sec.proof-unique-cartan}).

We then turn to orbit equivalence rigidity. In Section \ref{sec.cocycle-OE-rigidity}, we prove a cocycle superrigidity theorem for arbitrary cocycles of irreducible pmp actions $G_1 \times G_2 \actson (X,\mu)$ taking values in a locally compact group with property~(S). This result is similar to the cocycle superrigidity theorem of \cite{MS04}, where the target group is assumed to be a closed subgroup of the isometry group of a negatively curved space. We then deduce that Sako's orbit equivalence rigidity theorem \cite{Sa09} for irreducible pmp actions $G_1 \times G_2 \actson (X,\mu)$ of nonamenable groups in class $\cS$ stays valid in the locally compact setting. Recall here that a nonsingular action $G_1 \times G_2 \actson (X,\mu)$ of a direct product group is called irreducible if both $G_1$ and $G_2$ act ergodically.

In combination with Theorem \ref{thm.unique-Cartan}, we deduce the following W$^*$ strong rigidity theorem. This is the first W$^*$ strong rigidity theorem for actions of locally compact groups.

\begin{letterthm}\label{thm.Wstar-strong-rigidity}
Let $G = G_1 \times G_2$ and $H = H_1 \times H_2$ be unimodular lcsc groups without nontrivial compact normal subgroups. Let $G \actson (X,\mu)$ and $H \actson (Y,\eta)$ be essentially free, irreducible pmp actions. Assume that $G_1,G_2,H_1,H_2$ are nonamenable and that $H_1,H_2$ are weakly amenable and have property~(S).

If $p(L^\infty(X) \rtimes G)p \cong q(L^\infty(Y) \rtimes H)q$ for nonzero projections $p$ and $q$, then the actions are conjugate: there exists a continuous group isomorphism $\delta : G \recht H$ and a pmp isomorphism $\Delta : X \recht Y$ such that $\Delta(g \cdot x) = \delta(g) \cdot \Delta(x)$ for all $g \in G$ and a.e.\ $x \in X$.

Fix Haar measures on $G$ and $H$ and denote by $\Tr$ the associated normal semifinite trace on the crossed products $L^\infty(X) \rtimes G$ and $L^\infty(Y) \rtimes H$. If the Haar measures are normalized such that $\delta$ is measure preserving, then $\Tr(p) = \Tr(q)$. Also, the isomorphism $p(L^\infty(X) \rtimes G)p \cong q(L^\infty(Y) \rtimes H)q$ has the explicit form given in Remark \ref{rem.precise-form-iso}.
\end{letterthm}

We deduce Theorem \ref{thm.unique-Cartan} from a very general structural result on the normalizer $\cN_M(A) = \{u \in \cU(M) \mid u A u^* = A\}$ of a von Neumann subalgebra $A \subset M$ when $M$ is equipped with an arbitrary \emph{coaction} $\Phi : M \recht M \ovt L(G)$ of a locally compact weakly amenable group with property~(S), see Theorem \ref{thm.main-tech} below. The main novelty is to show that the main ideas of \cite{PV11} can be made to work in this very general and much more abstract setting, by using several results from the harmonic analysis of coactions.

We prove uniqueness of Cartan subalgebras by applying this general result to the canonical coaction $\Phi : L(\cR) \recht L(\cR) \ovt L(G)$ associated with a countable pmp equivalence relation $\cR$ and a cocycle $\om : \cR \recht G$ with values in the locally compact group $G$. Applying the same general result to the comultiplication $\Delta : L(G) \recht L(G) \ovt L(G)$ itself, we obtain the following \emph{strong solidity} results for locally compact group von Neumann algebras.

Recall that a diffuse von Neumann algebra $M$ is called \emph{strongly solid} if for every diffuse amenable von Neumann subalgebra $A \subset M$ that is the range of a normal conditional expectation, the normalizer $\cN_M(A)\dpr$ remains amenable. When also the amplification $B(\ell^2(\N)) \ovt M$ is strongly solid, we say that $M$ is \emph{stably strongly solid}, see \cite{BHV15}.

\begin{letterthm}\label{thm.strongly-solid}
Let $G$ be a locally compact group with property~(S) and assume that $L(G)$ is diffuse.
\begin{enumlist}
\item If $G$ is unimodular and weakly amenable, then for every finite trace projection $p \in L(G)$, we have that $p L(G) p$ is strongly solid.

\item If $G$ is second countable, if $G$ has the complete metric approximation property (CMAP) and if the kernel of the modular function $G_0 = \{g \in G \mid \delta(g) = 1 \}$ is an open subgroup of $G$, then $L(G)$ is stably strongly solid.
\end{enumlist}
\end{letterthm}

Note that the von Neumann algebras $L(G)$ appearing in the second part of Theorem \ref{thm.strongly-solid} can be of type~III. The assumption on $G_0$ being open in the second part of Theorem \ref{thm.strongly-solid} is not essential, but it makes the proof much less technical. In all our examples of locally compact groups $G$ with property~(S) and with $L(G)$ being nonamenable, the assumption is satisfied.

\begin{letterexs}\label{ex.our-examples}
Every finite center connected simple Lie group $G$ of real rank one is weakly amenable and has property~(S). Every locally compact group $G$ that acts metrically properly on a tree (not necessarily locally finite) has CMAP and property~(S). Every locally compact hyperbolic group is weakly amenable and has property~(S). References and proofs for these statements are discussed in Section \ref{sec.groups-prop-S}.

For locally compact groups $G$ acting properly on a tree, \cite[Theorems C and D]{HR16} and \cite[Theorems E and F]{Ra15} provide criteria ensuring that $L(G)$ is a nonamenable factor. Applying Theorem \ref{thm.strongly-solid}, we thus obtain the first examples of nonamenable strongly solid \emph{locally compact} group von Neumann algebras. In particular, when $n,m \in \Z$ with $2 \leq |m| < n$ and $G$ denotes the Schlichting completion of the Baumslag-Solitar group $\BS(m,n)$, then $L(G)$ is strongly solid, nonamenable and of type III$_{|m/n|}$ by combining Theorem \ref{thm.strongly-solid} and \cite[Theorem G]{Ra15}.
\end{letterexs}

Combining Theorem \ref{thm.unique-Cartan} with \cite[Proposition 7.1]{PV08}, we also obtain the following first examples of II$_1$ factors having a unique Cartan subalgebra up to unitary conjugacy, but not having a group measure space Cartan subalgebra, in the sense that the countable equivalence relation generated by the unique Cartan subalgebra cannot be written as the orbit equivalence relation of an essentially free group action.

\begin{lettercor}
Let $G = \Sp(n,1)$ with $n \geq 2$ and let $G \actson (X,\mu)$ be any weakly mixing Gaussian action. Put $M = L^\infty(X) \rtimes G$. Then, $M$ is a II$_\infty$ factor that has a unique Cartan subalgebra up to unitary conjugacy, but that has no group measure space Cartan subalgebra. In particular, its finite corners $pMp$ are II$_1$ factors with unique Cartan subalgebra, but without group measure space Cartan subalgebra.
\end{lettercor}

As explained above, Theorems \ref{thm.unique-Cartan} and \ref{thm.strongly-solid} follow from a general result on normalizers inside tracial von Neumann algebras $M$ that are equipped with a so-called coaction of a locally compact group. Recall that a \emph{coaction} of a locally compact group $G$ on a von Neumann algebra $M$ is a faithful normal $*$-homomorphism $\Phi : M \recht M \ovt L(G)$ satisfying $(\Phi \ot \id)\Phi = (\id \ot \Delta)\Phi$, where $\Delta : L(G) \recht L(G) \ovt L(G)$ is the \emph{comultiplication} given by $\Delta(\lambda_g) = \lambda_g \ot \lambda_g$ for all $g \in G$.

Assume that $\Phi : M \recht M \ovt L(G)$ is a coaction, $\Tr$ is a faithful normal semifinite trace on $M$ and $p \in M$ is a projection with $\Tr(p) < \infty$. Let $A \subset pMp$ be a von Neumann subalgebra. We say that
\begin{itemlist}
\item $A$ can be $\Phi$-embedded if the $pMp$-bimodule $\Phi(p)(L^2(Mp) \ot L^2(G))$ given by $x \cdot \xi \cdot y = \Phi(x) \xi (y \ot 1)$ admits a nonzero $A$-central vector;
\item $A$ is $\Phi$-amenable if there exists a nonzero positive functional $\Om$ on $\Phi(p) (M \ovt B(L^2(G))) \Phi(p)$ that is $\Phi(A)$-central and satisfies $\Om(\Phi(x)) = \Tr(x)$ for all $x \in pMp$.
\end{itemlist}

Note that the $\Phi$-amenability of $A \subset pMp$ is equivalent with the left $A$-amenability of the $pMp$-$M$-bimodule $\bim{\Phi(pMp)}{\Phi(p)(L^2(M) \ot L^2(G))}{M}$ in the sense of \cite[Definition 2.3]{PV11} and this amenability notion for bimodules is a generalization of relative amenability for pairs of von Neumann subalgebras introduced in \cite[Section 2.2]{OP07}. The following dichotomy type theorem is a locally compact version of \cite[Theorem 3.1]{PV12}.

\begin{letterthm}\label{thm.main-tech}
Let $G$ be a locally compact group that is weakly amenable and has property~(S). Let $(M,\Tr)$ be a von Neumann algebra with a faithful normal semifinite trace and $\Phi : M \recht M \ovt L(G)$ a coaction. Let $p \in M$ be a projection with $\Tr(p) < \infty$ and $A \subset pMp$ a von Neumann subalgebra.

If $A$ is $\Phi$-amenable then at least one of the following statements holds: $A$ can be $\Phi$-embedded or $\cN_{pMp}(A)\dpr$ stays $\Phi$-amenable.
\end{letterthm}

Finally, in order to obtain stable strong solidity, we have to replace the normalizer $\cN_{M}(A)$ by the \emph{stable normalizer} $\sN_M(A) = \{x \in M \mid x A x^* \subset A \;\;\text{and}\;\; x^* A x \subset A \}$.
Adapting the methods of \cite{BHV15} to the abstract setting of Theorem \ref{thm.main-tech}, we obtain the following result.

\begin{letterthm}\label{thm.main-tech-stable-normalizer}
If in Theorem \ref{thm.main-tech}, we add the hypothesis that $G$ has the complete metric approximation property, then in the conclusion, we may replace the normalizer $\cN_{pMp}(A)\dpr$ by the stable normalizer $\sN_{pMp}(A)\dpr$.
\end{letterthm}

%

\section{Proof of Theorem \ref{thm.main-tech}}\label{sec.proof-main-tech}

The proof of Theorem \ref{thm.main-tech} follows closely the proofs of \cite[Theorem 5.1]{PV11} and \cite[Theorem 3.1]{PV12}. The main novelty is to develop, in the context of coactions of locally compact groups, a framework in which the main ideas of \cite{PV11,PV12} are applicable. To do this, we need several results from the harmonic analysis of coactions and their crossed products, which were proven for arbitrary locally compact quantum groups in \cite{Va00,BSV02,BS92}.

Fix a weakly amenable, locally compact group $G$ with property~(S). Denote by $\Lambda(G)$ the Cowling-Haagerup constant of $G$, see \cite{CH88}. Also fix a von Neumann algebra $M$ with a faithful normal semifinite trace $\Tr$ and a coaction $\Phi : M \recht M \ovt L(G)$. Let $p \in M$ be a projection with $\Tr(p) < \infty$ and $A \subset pMp$ a von Neumann subalgebra that is $\Phi$-amenable. Denote by $\Delta : L(G) \recht L(G) \ovt L(G)$ the comultiplication, given by $\Delta(\lambda_g) = \lambda_g \ot \lambda_g$.

{\bf Weak amenability.} Denote by $A(G)$ the Fourier algebra of $G$, defined as the predual of $L(G)$ and identified with a subalgebra of the algebra $C_b(G)$ of bounded continuous functions on $G$, by identifying $\om \in L(G)_*$ with the function $g \mapsto \om(\lambda_g)$. We denote by $A_c(G) \subset A(G)$ the subalgebra of compactly supported functions in $A(G)$. By weak amenability of $G$, using \cite[Proposition 1.1]{CH88} and a convexity argument, we can fix a net $\eta_n \in A_c(G)$ such that the associated normal completely bounded maps $m_n : L(G) \recht L(G) : m_n(x) = (\id \ot \eta_n)\Delta(x)$ satisfy
\begin{equation}\label{eq.cond-wa}
\begin{split}
& \|m_n\|\cb \leq \Lambda(G) < \infty \quad\text{for all $n$, and}\\
& (\id \ot m_n)(X) \recht X \quad\text{strongly, for all Hilbert spaces $\cH$ and $X \in B(\cH) \ovt L(G)$.}
\end{split}
\end{equation}
Define the normal completely bounded maps $\vphi_n : M \recht M : \vphi_n(x) = (\id \ot \eta_n)\Phi(x)$. Using that $\Phi$ is a coaction, we get that $\Phi \circ \vphi_n = (\id \ot m_n) \circ \Phi$. Since $\Phi$ is faithful, $\Phi$ is completely isometric and thus, $\|\vphi_n\|\cb \leq \|m_n\|\cb$. Since $\Phi$ is a homeomorphism for the strong topology on norm bounded subsets, we get that $\vphi_n(x) \recht x$ strongly for every $x \in M$.

{\bf Notations and terminology.} Denote $\cK = L^2(Mp) \ot L^2(G)$ and view $\Phi$ as a normal $*$-homomorphism $\Phi : M \recht B(\cK)$. Also define the normal $*$-antihomomorphism $\rho : A \recht B(\cK)$ given by $\rho(a) \xi = \xi(a \ot 1)$. Define $\cN = \Phi(M) \vee \rho(A)$ as the von Neumann subalgebra of $B(\cK)$ generated by $\Phi(M)$ and $\rho(A)$. Note that $\cN \subset B(L^2(Mp)) \ovt L(G)$. We also denote by $\rho : A \recht B(L^2(Mp))$ the $*$-antihomomorphism given by right multiplication.

Whenever $\cV$ is a set of operators on a Hilbert space, we denote by $[\cV]$ the operator norm closed linear span of $\cV$. Denote by $\cN_0 \subset \cN$ the dense C$^*$-subalgebra defined as $\cN_0 := [\Phi(M)\rho(A)]$. Write $q = \Phi(p)$.

We say that a normal completely bounded map $\psi : pMp \recht pMp$ is \emph{adapted} if the following two conditions hold.
\begin{enumlist}
\item There exists a normal completely bounded map $\theta : q \cN q \recht B(L^2(pMp))$ with $\theta(\Phi(x)\rho(a)) = \psi(x) \rho(a)$ for all $x \in pMp$ and $a \in A$.
\item There exist a Hilbert space $\cL$, a unital $*$-homomorphism $\pi_0 : q \cN_0 q \recht B(\cL)$ and maps $\cV , \cW : \cN_{pMp}(A) \recht \cL$ such that
\begin{align*}
& \Tr(w^* \psi(x) v a) = \langle \pi_0(\Phi(x)\rho(a)) \cV(v),\cW(w) \rangle \quad\text{for all}\;\; x \in pMp, a \in A, v,w \in \cN_{pMp}(A) \; , \\
& \text{and, defining $\|\cV\|_\infty = \sup \bigl\{ \|\cV(v)\|  \bigm| v \in \cN_{pMp}(A) \bigr\}$, we have}\;\; \|\cV\|_\infty \, \|\cW\|_\infty < \infty \; .
\end{align*}
\end{enumlist}
We denote by $\|\psi\|\adap$ the infimum of all possible values of $\Tr(p)^{-1} \, \|\cV\|_\infty \, \|\cW\|_\infty$.

{\bf Step 1.} Let $\om \in A(G)$ and define $m : L(G) \recht L(G)$ by $m = (\id \ot \om)\circ \Delta$. Also define $\vphi : M \recht M$ by $\vphi = (\id \ot \om) \circ \Phi$ and, as before, note that $(\id \ot m)\circ \Phi = \Phi \circ \vphi$. Put $\psi : pMp \recht pMp : \psi(x) = p \vphi(x) p$. We claim that $\psi$ is adapted and that $\|\psi\|\adap \leq \|m\|\cb$.

To prove step~1, we first prove the following statement: the $pMp$-$A$-bimodule $\bim{pMp}{L^2(pMp)}{A}$ is weakly contained in the $pMp$-$A$-bimodule $\bim{\Phi(pMp)}{q(L^2(Mp) \ot L^2(G))}{A \ot 1}$.

Using the leg numbering notation for multiple tensor products, we view
$$\cK' := q_{12}(L^2(M) \ot L^2(G) \ot \overline{L^2(G)})q_{13}$$
as the standard Hilbert space for $q (M \ovt B(L^2(G)))q$. The left representation of $q (M \ovt B(L^2(G)))q$ on $\cK'$ is given by left multiplication in tensor positions 1 and 2, while the right representation is given by right multiplication in tensor positions 1 and 3. The $\Phi$-amenability of $A$ then provides a net of vectors $\xi_i \in \cK'$ satisfying
$$\lim_i \langle \Phi(x)_{12} \, \xi_i,\xi_i \rangle = \Tr(x) \quad\text{and}\quad \lim_i \| \Phi(a)_{12} \, \xi_i - \xi_i \, \Phi(a)_{13}\| = 0$$
for all $x \in pMp$ and $a \in A$. This implies that $\bim{pMp}{L^2(pMp)}{A}$ is weakly contained in the $pMp$-$A$-bimodule $\bim{\Phi(pMp)_{12}}{\cK'}{\Phi(A)_{13}}$. Since the $pMp$-bimodule
$\bim{\Phi(pMp)_{12}}{\cK'}{\Phi(pMp)_{13}}$ is unitarily conjugate to a multiple of the $pMp$-bimodule $\bim{\Phi(pMp)}{q(L^2(Mp) \ot L^2(G))}{pMp \ot 1}$, the above weak containment statement is proven. So, we get a unital $*$-homomorphism $\theta' : q \cN_0 q \recht B(L^2(pMp))$ satisfying $\theta'(\Phi(x) \rho(a)) = x \rho(a)$ for all $x \in pMp, a \in A$.

Define the normal completely bounded map $\theta : q \cN q \recht B(L^2(pMp))$ by $\theta(x) = p (\id \ot \om)(x) p$. By construction, $\theta(\Phi(x)\rho(a)) = \psi(x) \rho(a)$ for all $x \in pMp$ and $a \in A$. Since $\theta(x) = \theta'(q (\id \ot m)(x) q)$ for all $x \in q \cN_0 q$, we get that $\|\theta\|\cb \leq \|m\|\cb$. So, the Stinespring like factorization theorem (see e.g.\ \cite[Theorem B.7]{BO08}) provides a Hilbert space $\cL$, a unital $*$-homomorphism $\pi_0 : q \cN_0 q \recht B(\cL)$ and bounded operators $\cV_0, \cW_0 : L^2(pMp) \recht \cL$ satisfying $\theta(x) = \cW_0^* \pi_0(x) \cV_0$ for all $x \in q \cN_0 q$ and $\|\cV_0\| \, \|\cW_0\| = \|\theta\|\cb \leq \|m\|\cb$. It now suffices to define $\cV$ and $\cW$ by restricting $\cV_0$ and $\cW_0$ to $\cN_{pMp}(A) \subset L^2(pMp)$. So we have proved that $\psi : pMp \recht pMp$ is adapted and that $\|\psi\|\adap \leq \|m\|\cb$. This concludes the proof of step~1.

{\bf Notations and terminology.} We start with a net $\eta_n \in A_c(G)$ such that the associated normal completely bounded maps $m_n : L(G) \recht L(G)$ given by $m_n = (\id \ot \eta_n) \circ \Delta$ satisfy \eqref{eq.cond-wa}. Defining
$$\psi_n : pMp \recht pMp: \psi_n(x) = p (\id \ot \eta_n)\Phi(x) p \; ,$$
we obtain a net of adapted completely bounded maps $\psi_n : pMp \recht pMp$ such that $\psi_n(x) \recht x$ strongly for all $x \in pMp$ and $\limsup_n \|\psi_n\|\adap \leq \Lambda(G)$. We call such a net an \emph{adapted approximate identity}. We then define $\kappa \geq 1$ as the smallest positive number for which there exists an adapted approximate identity $\psi_n : pMp \recht pMp$ with $\limsup_n \|\psi_n\|\adap \leq \kappa$. We fix such a $\psi_n$ realizing $\kappa$.

Since each $\psi_n$ is adapted, we have normal completely bounded maps $\theta_n : q \cN q \recht B(L^2(pMp))$ satisfying $\theta_n(\Phi(x)\rho(a)) = \psi_n(x) \rho(a)$ for all $x \in pMp$ and $a \in A$. We can thus define $\mu_n \in (q \cN q)_*$ given by $\mu_n(T) = \langle \theta_n(T) p, p \rangle$ and satisfying $\mu_n(\Phi(x)\rho(a)) = \Tr(\psi_n(x) a)$ for all $x \in pMp$, $a \in A$.

For every $v \in \cN_{pMp}(A)$, denote by $\beta_v$ the automorphism of $\cN$ implemented by right multiplication with $v^* \ot 1$ on $L^2(Mp) \ot L^2(G)$. Note that $\beta_v(\Phi(x)\rho(a)) = \Phi(x) \rho(vav^*)$ for all $x \in M$ and $a \in A$. In particular, $\beta_v(q) = q$ and we also view $\beta_v$ as an automorphism of $q \cN q$.

{\bf Step 2.} The functionals $\mu_n$ satisfy the following properties.
\begin{enumlist}
\item $\limsup_n \|\mu_n\| < \infty$,
\item $\lim_n \mu_n(\Phi(x)\rho(a)) = \Tr(xa)$ for all $x \in pMp$, $a \in A$,
\item $\lim_n \|\mu_n \circ (\beta_v \circ \Ad \Phi(v)) - \mu_n \|=0$ for all $v \in \cN_{pMp}(A)$,
\item $\lim_n \| (\Phi(a)\rho(a^*))\cdot \mu_n - \mu_n\| = 0$ for all $a \in \cU(A)$.
\end{enumlist}
To prove step~2, one can literally repeat the argument in \cite[Proof of Proposition 7]{Oz10} and \cite[Proof of Proposition 5.4]{PV11}, because for every $v \in \cN_{pMp}(A)$ and every adapted approximate identity $\psi_n : pMp \recht pMp$, the maps $x \mapsto \psi_n(xv^*)v$ and $x \mapsto v^*\psi_n(v x)$ form again adapted approximate identities.

{\bf Step 3.} There exist \emph{positive} normal functionals $\om_n \in (q \cN q)_*$ satisfying
\begin{enumlist}
\item $\lim_n \om_n(\Phi(x)) = \Tr(x)$ for all $x \in pMp$,
\item $\lim_n \|\om_n \circ (\beta_v \circ \Ad \Phi(v)) - \om_n \|=0$ for all $v \in \cN_{pMp}(A)$,
\item $\lim_n \om_n(\Phi(a)\rho(a^*)) = \Tr(p)$ for all $a \in \cU(A)$.
\end{enumlist}
To prove step~3, choose a weak$^*$ limit point $\Xi \in (p\cN p)^*$ of the net $\mu_n$. We find that $\Xi(\Phi(x)) = \Tr(x)$ for all $x \in pMp$, that $\Xi$ is invariant under the automorphisms $\beta_v \circ \Ad \Phi(v)$ for all $v \in \cN_{pMp}(A)$ and that $(\Phi(a)\rho(a^*))\cdot \Xi = \Xi$ for all $a \in \cU(A)$. Define $\Om_1 = |\Xi|$. So $\Om_1$ is a positive element of $(q\cN q)^*$ satisfying
\begin{equation}\label{eq.set1}
\Om_1 \circ (\beta_v \circ \Ad \Phi(v)) = \Om_1 \quad\text{and}\quad (\Phi(a) \rho(a^*))\cdot \Om_1  = \Om_1 \quad\text{for all}\;\; v \in \cN_{pMp}(A) , a \in \cU(A) \; .
\end{equation}
Furthermore, we have that
$$|\Tr(x)|^2 = |\Xi(\Phi(x))|^2 \leq \|\Om_1\| \, \Om_1(\Phi(x^* x)) \quad\text{for all}\;\; x \in pMp \; .$$
In order to conclude the proof of step~3, we need to modify $\Omega_1$ so that its restriction to $\Phi(pMp)$ is given by the trace. We first modify $\Omega_1$ so that this restriction is normal and faithful.

The bidual of the embedding $\Phi : pMp \recht q\cN q$ is an embedding $\Phi^{**} : (pMp)^{**} \recht (q\cN q)^{**}$. Denote by $z \in \cZ((pMp)^{**})$ the support projection of the natural normal $*$-homomorphism $(pM p)^{**} \recht pMp$ given by dualizing the embedding $(pMp)_* \hookrightarrow (pMp)^*$. By construction, for every $\Om \in (pMp)^*$, the functional $\Om(\cdot z)$ belongs to $(pMp)_*$. Write $z_1 = \Phi^{**}(z)$. Whenever $\al \in \Aut(q \cN q)$ satisfies $\al(\Phi(pMp)) = \Phi(pMp)$, the bidual automorphism $\al^{**} \in \Aut((q\cN q)^{**})$ satisfies $\al^{**}(z_1) = z_1$. In particular, $z_1$ commutes with every unitary in $q \cN q$ that normalizes $\Phi(pMp)$. Since $\Phi(pMp) \subset q \cN q$ is regular, it follows that $z_1$ belongs to the center of $(q \cN q)^{**}$. Applying the statement above to the automorphism $\al = \beta_v \circ \Ad \Phi(v)$ and the unitary $\Phi(a) \rho(a^*)$, it follows that the positive functional $\Om_2(\cdot) = \Om_1(\, \cdot \, z_1)$ still satisfies the properties in \eqref{eq.set1}.

By density, we have $|\Tr(x)|^2 \leq \|\Om_1\| \, \Om_1(\Phi^{**}(x^* x))$ for all $x \in (pMp)^{**}$. Since $\Tr(x) = \Tr(xz)$ for all $x \in pMp$, we conclude that $|\Tr(x)|^2 \leq \|\Om_1\| \, \Om_2(\Phi(x^* x))$ for all $x \in pMp$. In particular, $\Om_2 \circ \Phi$ is faithful. Since $\Om_2(\Phi(x)) = \Om_2(\Phi^{**}(xz))$ for all $x \in pMp$, we get that $\Om_2 \circ \Phi$ is normal. So, we find a nonsingular $T \in L^1(pMp)^+$ such that $\Om_2(\Phi(x)) = \Tr(x T)$ for all $x \in pMp$. Since \eqref{eq.set1} holds, we have that $T$ commutes with $\cN_{pMp}(A)$. For every $n \geq 3$, we then define the positive functional $\Om_n$ on $q \cN q$ given by
$$\Om_n(\cdot) = \Om_2(\Phi((T+1/n)^{-1/2}) \, \cdot \, \Phi((T+1/n)^{-1/2}))\; .$$
Each $\Om_n$ satisfies the properties in \eqref{eq.set1}. Choosing $\Om$ to be a weak$^*$-limit point of the sequence $\Om_n$, we have found a positive functional $\Om$ on $q \cN q$ that satisfies the properties in \eqref{eq.set1} and that moreover satisfies $\Om(\Phi(x)) = \Tr(x)$ for all $x \in pMp$. Approximating $\Om$ in the weak$^*$ topology and taking convex combinations, we find a net of positive $\om_n \in (q \cN q)_*$ satisfying the conditions in step~3.

{\bf Notations and terminology.} Choose a standard Hilbert space $\cH$ for the von Neumann algebra $\cN$, which comes with the normal $*$-homomorphism $\pi_l : \cN \recht B(\cH)$, the normal $*$-antihomomorphism $\pi_r : \cN \recht B(\cH)$ and the positive cone $\cH^+ \subset \cH$. For every $v \in \cN_{pMp}(A)$, denote by $W_v \in \cU(\cH)$ the canonical implementation of $\beta_v \in \Aut(\cN)$.

{\bf Step 4.} There exist vectors $\xi_n \in \cH^+$ satisfying $\pi_l(q) \xi_n = \xi_n = \pi_r(q) \xi_n$ for all $n$ and
\begin{enumlist}
\item $\lim_n \langle \pi_l(\Phi(x)) \xi_n , \xi_n \rangle = \Tr(pxp) = \lim_n \langle \pi_r(\Phi(x)) \xi_n , \xi_n \rangle$ for all $x \in M$,
\item $\lim_n \| \pi_l(\Phi(v)) \pi_r(\Phi(v^*)) W_v \xi_n - \xi_n \| = 0$ for all $v \in \cN_{pMp}(A)$,
\item $\lim_n \| \pi_l(\Phi(a)) \xi_n - \pi_l(\rho(a)) \xi_n \| = 0$ for all $a \in \cU(A)$.
\end{enumlist}
Note that $\pi_l(q) \pi_r(q) \cH$ serves as the standard Hilbert space of $q \cN q$. Define $\xi_n \in \pi_l(q)) \pi_r(q) \cH^+$ as the canonical implementation of the normal positive functional $\om_n \in (q \cN q)_*$. The properties of $\om_n$ in step~3 translate into the above properties for $\xi_n$ by the Powers-St{\o}rmer inequality.

{\bf Notations and terminology.} Define the coaction $\Psi : \cN \recht \cN \ovt L(G)$ given by $\Psi = \id \ot \Delta$. By \cite[Definition 3.6 and Theorem 4.4]{Va00}, the coaction $\Psi$ has a canonical implementation on $\cH$, given by a nondegenerate $*$-homomorphism $\pi : C_0(G) \recht B(\cH)$ satisfying the following natural covariance properties w.r.t.\ $\pi_l$, $\pi_r$ and $\Psi$.

Denote by $C^*_\lambda(G) \subset L(G)$ and $C^*_\rho(G) \subset R(G)$ the canonical dense C$^*$-subalgebras. We denote by $\otmin$ the spatial C$^*$-tensor product and define $V \in M(C_0(G) \otmin C^*_\lambda(G))$ and $W \in M(C_0(G) \otmin C^*_\rho(G))$ given by the functions $V(g) = \lambda_g$ and $W(g) = \rho_g$. Define the unitary operators $X,Y \in B(\cH \ot L^2(G))$ given by $X = (\pi \ot \id)(V)$ and $Y = (\pi \ot \id)(W)$.
Denoting by $\chi : L(G) \recht R(G) : \chi(\lambda_g) = \rho_g^*$ the canonical anti-isomorphism, the covariance properties are then given by
$$(\pi_l \ot \id)\Psi(x) = X (\pi_l(x) \ot 1) X^* \quad\text{and}\quad (\pi_r \ot \chi)\Psi(x) = Y(\pi_r(x) \ot 1)Y^*$$
for all $x \in \cN$.

{\bf Formulation of the dichotomy.} We are in precisely one of the following cases.
\begin{itemlist}
\item {\bf Case 1.} For every $F \in C_0(G)$, we have that $\limsup_n \|\pi(F) \xi_n\| = 0$.
\item {\bf Case 2.} There exists an $F \in C_0(G)$ with $\limsup_n \|\pi(F) \xi_n\| > 0$.
\end{itemlist}

We prove that in case~1, the von Neumann subalgebra $\cN_{pMp}(A)\dpr \subset pMp$ is $\Phi$-amenable and that in case~2, the von Neumann subalgebra $A \subset pMp$ can be $\Phi$-embedded.

{\bf Case 1 -- Notations and terminology.} Since $G$ has property~(S), we have a continuous map $Z_0 : G \recht L^2(G)$ satisfying $\|Z_0(g)\| = 1$ for all $g \in G$ and
\begin{equation}\label{eq.what-we-have}
\lim_{k \recht \infty} \|Z_0(gkh) - \lambda_g(Z_0(k))\| = 0 \quad\text{uniformly on compact sets of $g,h \in G$.}
\end{equation}
For each $F \in C_0(G)$, we view $Z_0 F \in C_0(G) \otmin L^2(G)$ and in this way, $Z_0$ is an adjointable operator from the C$^*$-algebra $C_0(G)$ to the Hilbert C$^*$-module $C_0(G) \otmin L^2(G)$. Define $\Vtil \in M(C^*_\lambda(G) \otmin C_0(G))$ given by the function $g \mapsto \lambda_g$. So, $\Vtil$ is just the flip of the unitary $V$ defined above. As operators on $L^2(G) \ot L^2(G)$, we have $\Delta(a) = \Vtil (1 \ot a) \Vtil^*$ for all $a \in L(G)$.

By \cite[Section 5]{BSV02}, the closed linear span
\begin{equation}\label{eq.Cstar-M0}
M_0 = [(\id \ot \om)\Phi(x) \mid x \in M , \om \in L(G)_*]
\end{equation}
is a unital C$^*$-subalgebra of $M$. Also, $\Phi(M_0) \subset M(M_0 \otmin C^*_\lambda(G))$ and the restriction of $\Phi$ to $M_0$ defines a continuous coaction. In particular, the closed linear span
\begin{equation}\label{eq.Cstar-Sl}
S_l = [\Phi(M_0)(1 \ot C_0(G))] \subset M \ovt B(L^2(G))
\end{equation}
is a C$^*$-algebra (i.e.\ the crossed product of $M_0$ and the coaction $\Phi$ of $C^*_\lambda(G)$, as first defined in \cite[D\'{e}finition 7.1]{BS92}) and
\begin{equation}\label{eq.density}
M_0 = [(\id \ot \om)\Phi(x) \mid x \in M_0 , \om \in L(G)_*] \; .
\end{equation}

{\bf Case 1 -- Step~1.} We claim that
\begin{equation}\label{eq.crucial-limit}
(1 \ot Z_0)\Phi(x) - (\Phi \ot \id)\Phi(x) (1 \ot Z_0) \in S_l \otmin L^2(G)
\end{equation}
for all $x \in M_0$.

Note that \eqref{eq.what-we-have}, with $h=e$, can be rephrased as follows: $(1 \ot \lambda_g^*) Z_0 - (\lambda_g \ot 1)Z_0 \lambda_g^*$ belongs to $C_0(G) \otmin L^2(G)$, uniformly on compact sets of $g \in G$. This means that for every $F \in C_0(G)$, we have
$$\Vtil_{23}^* (Z_0 \ot F) - \Vtil_{13} (Z_0 \ot F) \Vtil^* \in [C_0(G) \ot L^2(G) \ot C_0(G)]$$
and thus, because $\Vtil$ normalizes $C_0(G \times G)$,
\begin{equation}\label{eq.est-a}
\Vtil_{13}^* \Vtil_{23}^* (Z_0 \ot F) - (Z_0 \ot F)\Vtil^* \in [(C_0(G) \ot L^2(G) \ot C_0(G))\Vtil^*] \; .
\end{equation}
Using that $\Vtil_{23}$ and $\Vtil_{13}$ commute, we similarly find that
\begin{equation}\label{eq.est-b}
\Vtil_{23} \Vtil_{13} (Z_0 \ot F) - (Z_0 \ot F) \Vtil \in [(C_0(G) \ot L^2(G) \ot C_0(G))\Vtil] \; .
\end{equation}

By \eqref{eq.density}, it suffices to prove \eqref{eq.crucial-limit} for $x = (1 \ot \eta^*)\Phi(y)(1 \ot \mu)$ where $y \in M_0$ and where $\eta,\mu \in C_c(G)$ are viewed as vectors in the Hilbert space $L^2(G)$. Fix $F \in C_0(G)$ such that $\eta^* F = \eta^*$ and $F \mu = \mu$. Using that $\Phi$ is a coaction and that $\Delta(a) = \Vtil (1 \ot a)\Vtil^*$ for all $a \in L(G)$, we find that
$$(\Phi \ot \id)\Phi(x) = (1 \ot 1 \ot 1 \ot \eta^*) \, \Vtil_{34} \, \Vtil_{24} \, \Phi(y)_{14} \, \Vtil_{24}^* \, \Vtil_{34}^* \, (1 \ot 1 \ot 1 \ot \mu) \; .$$
Twice using that $\mu = F \mu$, it then follows from \eqref{eq.est-a} that
\begin{align}
(\Phi \ot \id)\Phi(x)&\, (1 \ot Z_0) \notag\\
&= (1 \ot 1 \ot 1 \ot \eta^*) \, \Vtil_{34} \, \Vtil_{24} \, \Phi(y)_{14} \, (1 \ot Z_0 \ot F) \, \Vtil^*_{23} \, (1 \ot 1 \ot \mu) + T \notag\\
&= (1 \ot 1 \ot 1 \ot \eta^*) \, \Vtil_{34} \, \Vtil_{24} \, \Phi(y)_{14} \, (1 \ot Z_0 \ot 1) \, \Vtil^*_{23} \, (1 \ot 1 \ot \mu) + T \notag\\
&= (1 \ot 1 \ot 1 \ot \eta^*) \, \Vtil_{34} \, \Vtil_{24} \, (1 \ot Z_0 \ot 1)  \, \Phi(y)_{13} \, \Vtil^*_{23} \, (1 \ot 1 \ot \mu) + T \label{eq.intermediate}
\end{align}
where the error term $T$ belongs to
$$[(1 \ot 1 \ot 1 \ot L^2(G)^*) \, \Vtil_{34} \, \Vtil_{24} \, \Phi(M_0)_{14} \, (1 \ot C_0(G) \ot L^2(G) \ot C_0(G)) \, \Vtil^*_{23} \, (1 \ot 1 \ot L^2(G))] \; .$$
Using that $[\Phi(M_0)(1 \ot C_0(G))] = S_l = [(1 \ot C_0(G)) \Phi(M_0)]$, that $\Vtil_{34}$ and $\Vtil_{24}$ commute, that $[\Vtil (L^2(G) \ot C_0(G))] = [L^2(G) \ot C_0(G)]$ and that $\Vtil$ normalizes $C_0(G) \otmin C_0(G) = C_0(G \times G)$, we get that $T$ belongs to
\begin{align*}
& [(1 \ot 1 \ot 1 \ot L^2(G)^*) \, \Vtil_{34} \, \Vtil_{24} \, (1 \ot C_0(G) \ot L^2(G) \ot C_0(G)) \, \Phi(M_0)_{13} \, \Vtil^*_{23} \, (1 \ot 1 \ot L^2(G))] \\
& = [(1 \ot 1 \ot 1 \ot L^2(G)^*) \, \Vtil_{24} \, (1 \ot C_0(G) \ot L^2(G) \ot C_0(G)) \, \Phi(M_0)_{13} \, \Vtil^*_{23} \, (1 \ot 1 \ot L^2(G))] \\
& = [(1 \ot C_0(G) \ot L^2(G)) \, (1 \ot 1 \ot L^2(G)^*) \, \Vtil_{23} \, \Phi(M_0)_{13} \, \Vtil^*_{23} \, (1 \ot 1 \ot L^2(G))] \\
& = [(1 \ot C_0(G) \ot L^2(G)) \, (1 \ot 1 \ot L^2(G)^*) \, (\Phi \ot \id)\Phi(M_0) \, (1 \ot 1 \ot L^2(G))] \\
& = [(1 \ot C_0(G))\Phi(M_0) \ot L^2(G)] = S_l \otmin L^2(G) \; .
\end{align*}

Using that $\eta^* = \eta^* F$ and using \eqref{eq.est-b}, we can continue the computation in \eqref{eq.intermediate} and find that
\begin{align*}
(\Phi \ot \id)\Phi(x)&\, (1 \ot Z_0) \\
&= (1 \ot 1 \ot 1 \ot \eta^*) \, \Vtil_{34} \, \Vtil_{24} \, (1 \ot Z_0 \ot F)  \, \Phi(y)_{13} \, \Vtil^*_{23} \, (1 \ot 1 \ot \mu) + T \\
&= (1 \ot 1 \ot 1 \ot \eta^*) \, (1 \ot Z_0 \ot F) \, \Vtil_{23}  \, \Phi(y)_{13} \, \Vtil^*_{23} \, (1 \ot 1 \ot \mu) + T' + T \\
&= (1 \ot Z_0) \, (1 \ot 1 \ot \eta^*) \, (\Phi \ot \id)\Phi(y) \, (1 \ot 1 \ot \mu) + T' + T \\
&= (1 \ot Z_0) \, \Phi(x) + T' + T
\end{align*}
where the error term $T'$ belongs to
\begin{align*}
& [(1 \ot 1 \ot 1 \ot L^2(G)^*) \, (1 \ot C_0(G) \ot L^2(G) \ot C_0(G)) \, \Vtil_{23} \, \Phi(M_0)_{13} \, \Vtil^*_{23} \, (1 \ot 1 \ot L^2(G))] \\
& = [(1 \ot C_0(G) \ot L^2(G)) \, (1 \ot 1 \ot L^2(G)^*) \, (\Phi \ot \id)\Phi(M_0) \, (1 \ot 1 \ot L^2(G))] \\
& = S_l \otmin L^2(G) \; .
\end{align*}
So \eqref{eq.crucial-limit} and step~1 are proven.

{\bf Case 1 -- Step 2.} Define the $*$-homomorphism $\zeta_l : M \recht B(\cH) : \zeta_l = \pi_l \circ \Phi$ and the $*$-antihomomorphism $\zeta_r : M \recht B(\cH) : \zeta_r = \pi_r \circ \Phi$. Define
\begin{equation}\label{eq.def-cstar-S}
S = [\zeta_l(M_0) \, \zeta_r(M_0) \, \pi(C_0(G)) \, W_v \mid v \in \cN_{pMp}(A) ] \; .
\end{equation}
Finally, define the isometry $Z \in B(\cH, \cH \ot L^2(G))$ given by $Z = (\pi \ot \id)(Z_0)$. We claim that $S \subset B(\cH)$ is a C$^*$-algebra and that
\begin{equation}\label{eq.conclusion}
Z \zeta_l(x) - (\zeta_l \ot \id)\Phi(x) Z \quad\text{and}\quad Z \zeta_r(x) - (\zeta_r(x) \ot 1) Z \quad\text{belong to}\;\; S \otmin L^2(G)
\end{equation}
for all $x \in M_0$.

Since $\zeta_l : M_0 \recht B(\cH)$ and $\pi : C_0(G) \recht B(\cH)$ are covariant w.r.t.\ the continuous coaction $\Phi : M_0 \recht M(M_0 \otmin C^*_r(G))$, they induce a nondegenerate representation of the full crossed product. Since $G$ is co-amenable, the canonical homomorphism of the full crossed product onto the reduced crossed product is an isomorphism. The reduced crossed product is given by the C$^*$-algebra $S_l$ defined in \eqref{eq.Cstar-Sl}. So, we find a nondegenerate $*$-homomorphism
$$\theta_l : S_l \recht B(\cH) : \theta_l(\Phi(x) (1 \ot F)) = \zeta_l(x) \pi(F) \; ,$$
for all $x \in M_0, F \in C_0(G)$.

Associated with the coaction $\Phi : M \recht M \ovt L(G)$, we have the canonical coaction $\Phi\op : M\op \recht M\op \ovt R(G)$ defined as follows. Denote by $\gamma : M \recht M\op : \gamma(x) = x\op$ the canonical $*$-anti-isomorphism. As before, define the $*$-anti-isomorphism $\chi : L(G) \recht R(G) = \eta(\lambda_g) = \rho_g^*$. Then, $\Phi\op \circ \gamma = (\gamma \ot \chi) \circ \Phi$. The corresponding crossed product C$^*$-algebra is
$$S_r = [\Phi\op(M_0\op)(1 \ot C_0(G))] \subset M\op \ovt B(L^2(G)) \; .$$
Since also $\zeta_r$ and $\pi$ are covariant, we similarly find a nondegenerate $*$-homomorphism
$$\theta_r : S_r \recht B(\cH) : \theta_r(\Phi\op(x\op) (1 \ot F)) = \zeta_r(x) \pi(F) \; ,$$
for all $x \in M_0, F \in C_0(G)$.

So $\theta_l(S_l) = [\zeta_l(M_0) \pi(C_0(G))]$ and $\theta_r(S_r) = [\zeta_r(M_0) \pi(C_0(G))]$ and these are C$^*$-algebras. Moreover, the unitaries $W_v$, $v \in \cN_{pMp}(A)$, commute with $\zeta_l(M)$, $\zeta_r(M)$ and $\pi(C_0(G))$. So, the space $S$ defined in \eqref{eq.def-cstar-S} is a C$^*$-algebra and
$$S = [\theta_l(S_l) \, \theta_r(S_r) \, W_v \mid v \in \cN_{pMp}(A) ] \; .$$
Also, $\theta_l(S_l) \subset S$ and $\theta_r(S_r) \subset S$.

Applying to \eqref{eq.crucial-limit} the canonical extension of $\theta_l \ot \id$ to the multiplier algebra, we find the first half of \eqref{eq.conclusion}. In the same way as we proved \eqref{eq.crucial-limit}, one proves that
\begin{equation}\label{eq.crucial-limit-bis}
(1 \ot Z_0) \Phi\op(x\op) - (\Phi\op(x\op) \ot 1)(1 \ot Z_0) \in S_r \otmin L^2(G)
\end{equation}
for all $x \in M_0$. Applying $\theta_r \ot \id$ to \eqref{eq.crucial-limit-bis}, also the second half of \eqref{eq.conclusion} follows and step~2 is proven.

{\bf Case 1 -- Notations.} Write $\cG = \cN_{pMp}(A)$ and consider the $*$-algebras $\C \cG$ and $D = M \otalg M\op \otalg \C \cG$. Define the $*$-homomorphisms
\begin{align*}
& \Theta : D \recht B(\cH) : \Theta(x \ot y\op \ot v) = \zeta_l(x) \, \zeta_r(y) \, W_v \;\; ,\\
& \Theta_1 : D \recht B(\cH \ot L^2(G)) : \Theta_1(x \ot y\op \ot v) = (\zeta_l\ot \id)\Phi(x) \, (\zeta_r(y) \, W_v \ot 1) \;\; .
\end{align*}
Choose a positive functional $\Om$ on $B(\cH)$ as a weak$^*$ limit point of the net of vector functionals $T \mapsto \langle T \xi_n , \xi_n \rangle$. The properties of the net $\xi_n$ established in step~4 above then imply that:
\begin{equation}\label{eq.translated-properties}
\begin{split}
& \Om(1) = \Om(\Theta(p \ot p\op \ot p)) = \Tr(p) \; , \;\; \Om(\Theta(x \ot 1 \ot p)) = \Tr(pxp) \;\text{for all $x \in M$},\\
& \Om(\Theta(v \ot (v^*)\op \ot v)) = \Tr(p) \;\text{for all $v \in \cG$.}
\end{split}
\end{equation}

{\bf Case 1 -- Step 3.} Writing $C = \|\Om\| \, \Lambda(G)^2$, we claim that
\begin{equation}\label{eq.even-better-est}
|\Om(\Theta(x))| \leq C \, \|\Theta_1(x)\| \quad\text{for all}\;\; x \in D \; .
\end{equation}

Since $W_v$ commutes with $\pi(C_0(G))$ for all $v \in \cG$, we have $Z W_v = (W_v \ot 1)Z$ for all $v \in \cG$. Denoting $D_0 = M_0 \otalg M_0\op \otalg \C \cG$, \eqref{eq.conclusion} implies that
\begin{equation}\label{eq.final-conclusion}
Z^* \Theta_1(x) Z - \Theta(x) \in S \quad\text{for all}\;\; x \in D_0 \; .
\end{equation}
Since we are in case~1, we have that $\Om(\pi(F)) = 0$ for all $F \in C_0(G)$. So, $\Om(T) = 0$ for all $T \in S$. It then follows from \eqref{eq.final-conclusion} that
\begin{equation}\label{eq.great-est}
|\Om(\Theta(x))| = |\Om(Z^* \Theta_1(x) Z)| \leq \|\Om\| \, \|\Theta_1(x)\| \quad\text{for all}\;\; x \in D_0 \; .
\end{equation}

To conclude step~3, we now have to approximate as follows an arbitrary $x \in D$ by elements in $D_0$.

Take a net $\eta_n \in A(G)$ such that the net $m_n = (\id \ot \eta_n)\circ \Delta$ satisfies \eqref{eq.cond-wa}. Define $\vphi_n : M \recht M$ by $\vphi_n = (\id \ot \eta_n) \circ \Phi$. Note that the image of $\Theta_1$ lies in $B(\cH) \ovt L(G)$ and that $(\id \ot m_n) \circ \Theta_1 = \Theta_1 \circ (\vphi_n \ot \id \ot \id)$. It follows that
$$\|\Theta_1((\vphi_n \ot \id \ot \id)(x))\| \leq \Lambda(G) \, \|\Theta_1(x)\| \quad\text{for all}\;\; x \in D \;\;\text{and all}\;\; n \; .$$
Denoting by $\chi_1 : L(G) \recht L(G)$ the period 2 anti-automorphism given by $\chi_1(\lambda_g) = \lambda_{g^{-1}}$, the representation $\Theta_1$ is unitarily conjugate to the representation
$$\Theta_2 : D \recht B(\cH \ot L^2(G)) : \Theta_2(x \ot y\op \ot v) = (\zeta_l(x) \ot 1) \, (\zeta_r\ot \chi_1)\Phi(y) \, (W_v \ot 1) \;\; .$$
So, writing $\vphi_m\op(y\op) = (\vphi_m(y))\op$, we also find that
$$\|\Theta_1((\id \ot \vphi_m\op \ot \id)(x))\| \leq \Lambda(G) \, \|\Theta_1(x)\| \quad\text{for all}\;\; x \in D \;\;\text{and all}\;\; m \; .$$
Altogether, we have proved that
$$\|\Theta_1((\vphi_n \ot \vphi_m\op \ot \id)(x))\| \leq \Lambda(G)^2 \, \|\Theta_1(x)\| \quad\text{for all}\;\; x \in D \;\;\text{and all}\;\; n,m \; .$$
For every $T \in B(\cH)$, write $\|T\|_\Om = \sqrt{\Om(T^* T)}$. Since
$$\|\zeta_l(\vphi_n(x)) - \zeta_l(x)\|_\Om^2 = \Tr(p(\vphi_n(x)-x)^*(\vphi_n(x)-x)p)$$
for every $x \in M$,
it follows from the Cauchy-Schwarz inequality that for every $x \in D$ and every $m$,
$$\Om(\Theta((\id \ot \vphi_m\op \ot \id)(x))) = \lim_n \Om(\Theta((\vphi_n \ot \vphi_m\op \ot \id)(x))) \; .$$
Similarly, we have
$$\Om(\Theta(x)) = \lim_m \Om(\Theta((\id \ot \vphi_m\op \ot \id)(x)))$$
for all $x \in D$. Since $(\vphi_n \ot \vphi_m\op \ot \id)(x) \in D_0$ for all $n,m$, it follows from \eqref{eq.great-est} that
$$|\Om((\vphi_n \ot \vphi_m\op \ot \id)(x))| \leq \|\Om\| \, \|\Theta_1((\vphi_n \ot \vphi_m\op \ot \id)(x))\| \leq C \, \|\Theta_1(x)\|$$
for all $n,m$. Taking first the limit over $n$ and then over $m$, we find that \eqref{eq.even-better-est} holds and step~3 is proven.

{\bf Case 1 -- End of the proof.} Because of \eqref{eq.even-better-est}, we can define a continuous functional $\Om_1$ on the C$^*$-algebra $[\Theta_1(D)]$ satisfying $\Om_1(\Theta_1(x)) = \Om(\Theta(x))$ for all $x \in D$. Since
$$\Om_1(\Theta_1(x)^* \Theta_1(x)) = \Om(\Theta(x^* x)) \geq 0$$
for all $x \in D$, it follows by density that $\Om_1$ is positive.

Extend $\Om_1$ to a functional on $B(\cH \ot L^2(G))$ without increasing its norm. So $\Om_1$ remains positive. Write
$$q_1 = \Theta_1(p \ot p\op \ot p) = (\zeta_l \ot \id)\Phi(p) \, (\zeta_r(p) \ot 1) \; .$$
For every $v \in \cG$, define
$$U_v := \Theta_1(v \ot (v^*)\op \ot v) = (\zeta_l \ot \id)(\Phi(v)) (\zeta_r(v^*) W_v \ot 1)$$
and note that $U_v$ is a unitary in $B(q_1(\cH \ot L^2(G)))$. By \eqref{eq.translated-properties}, these unitaries $U_v$ satisfy
\begin{align*}
\Om_1(U_v) & = \Om_1(\Theta_1(v \ot (v^*)\op \ot v)) = \Om(\Theta(v \ot (v^*)\op \ot v)) = \Tr(p) \\ &= \Om(\Theta(p \ot p\op \ot p)) = \Om_1(\Theta_1(p \ot p\op \ot p)) = \Om_1(q_1) \; .
\end{align*}
By \eqref{eq.translated-properties}, we also have that
\begin{align*}
& \Om_1(1-q_1) = \Om(1 - \Theta(p \ot p\op \ot p)) = 0 \quad\text{and}\\
& \Om_1((\zeta_l \ot \id)(\Phi(x))) = \Om(\Theta(x \ot 1 \ot p)) = \Tr(pxp) \quad\text{for all $x \in M$.}
\end{align*}
Altogether, we have in particular that $\Om_1$ is $U_v$-central for every $v \in \cG$.
%

Define the positive functional $\Om_2$ on $q(M \ovt B(L^2(G)))q$ given by $\Om_2(T) = \Om_1((\zeta_l \ot \id)(T))$. Then, $\Om_2(\Phi(x)) = \Tr(x)$ for all $x \in pMp$. Since for every $v \in \cG$, the functional $\Om_1$ is $U_v$-central, while $\zeta_r(v^*) W_v \ot 1$ commutes with $(\zeta_l \ot \id)(q(M \ovt B(L^2(G)))q)$, we get that $\Om_2$ is $\Phi(\cG)$-central. Writing $P = \cN_{pMp}(A)\dpr$, the Cauchy-Schwarz inequality implies that $\Om_2$ is $\Phi(P)$-central. So we have proved that $P$ is $\Phi$-amenable.

{\bf Proof in Case 2.} After passing to a subnet, we may assume that there is an $F \in C_0(G)$ such that the net $\|\pi(F)\xi_n\|$ is convergent to a strictly positive number. Choose a positive functional $\Om$ on $B(\cH)$ as a weak$^*$ limit point of the net of vector functionals $T \mapsto \langle T \xi_n , \xi_n \rangle$. Define the C$^*$-algebra $S_1 := \theta_l(S_l) = [\zeta_l(M_0) \pi(C_0(G))]$. Denote by $\Om_1$ the restriction of $\Om$ to $S_1\dpr$. By the properties of the net $\xi_n$ established in step~4 above, $\Om_1(\zeta_l(x)) = \Tr(pxp)$ for all $x \in M$ and $\Om_1$ is $\zeta_l(A)$-central. Also, the restriction of $\Om_1$ to $S_1$ is nonzero.

Define $\delta = \|\Om_1|_{S_1}\|$ and put $\eps = \delta (4\Lambda(G)^3 + 2\Lambda(G)^2 + 2)^{-1}$. Since the elements $\pi(F)$, with $F \in C_c(G)$ and $0 \leq F \leq 1$, form an approximate identity for $S_1$, we can fix $F \in C_c(G)$ with $0 \leq F \leq 1$ and
$$\Om_1(\pi(F)) \geq \delta-\eps \quad\text{and}\quad |\Om_1(T) - \Om_1(T \pi(F))| < \eps \, \|T\| \;\;\text{for all}\;\; T \in S_1 \; .$$
As above, take a net of completely bounded maps $\vphi_n : M \recht M$ such that $\|\vphi_n\|\cb \leq \Lambda(G)$ and $\vphi_n(M) \subset M_0$ for all $n$ and $\vphi_n(x) \recht x$ strongly for all $x \in M$. Because $\Om_1(\zeta_l(x)) = \Tr(pxp)$ for all $x \in M$,
\begin{equation}\label{eq.useful}
\Om_1(\zeta_l(x) T \zeta_l(y)) = \lim_n \Om_1(\zeta_l(\vphi_n(x)) T \zeta_l(\vphi_n(y)))
\end{equation}
for all $x,y \in M$ and $T \in S_1\dpr$.

Using the $\zeta_l(A)$-centrality of $\Om_1$, we then find, for all $a \in \cU(A)$,
\begin{align*}
\delta &\leq \Om_1(\pi(F)) + \eps = \Om_1(\zeta_l(a^*) \pi(F) \zeta_l(a)) + \eps \\
&= \lim_n \Re \Om_1(\zeta_l(\vphi_n(a^*)) \pi(F) \zeta_l(\vphi_n(a))) + \eps \; .
\end{align*}
Since $\zeta_l(\vphi_m(a^*)) \pi(F) \zeta_l(\vphi_n(a))$ belongs to $S_1$ and has norm at most $\Lambda(G)^2$, we get that
\begin{equation}\label{eq.stapje}
\delta \leq \limsup_n \Re \Om_1(\zeta_l(\vphi_n(a^*)) \pi(F) \zeta_l(\vphi_n(a)) \pi(F)) + \eps(\Lambda(G)^2 + 1) \; .
\end{equation}

We claim that there exists an $\om_0 \in A(G)$ such that the corresponding completely bounded map $\vphi_0 : M \recht M : \vphi_0 = (\id \ot \om_0)\circ \Phi$ satisfies $\|\vphi_0\|\cb \leq 2 \Lambda(G)$ and
\begin{equation}\label{eq.squeeze}
\pi(F) \zeta_l(x) \pi(F) = \pi(F) \zeta_l(\vphi_0(x)) \pi(F) \quad\text{for all}\;\; x \in M_0 \; .
\end{equation}
Using $\theta_l : S_l \recht S_1$, it suffices to construct $\om_0 \in A(G)$ such that $\|\vphi_0\|\cb \leq 2 \Lambda(G)$ and
\begin{equation}\label{eq.to-do}
(1 \ot F) \Phi(x) (1 \ot F) = (1 \ot F) \Phi(\vphi_0(x)) (1 \ot F) \quad\text{for all}\;\; x \in M_0 \; .
\end{equation}
Denote by $K \subset G$ the (compact) support of $F$. By \cite[Proposition 1.1]{CH88}, we can choose $\om_0 \in A(G)$ such that $\om_0(g) = 1$ for all $g \in K K^{-1}$ and such that the map $m_0 = (\id \ot \om_0)\circ \Delta$ satisfies $\|m_0\|\cb \leq 2 \Lambda(G)$. As operators on $L^2(G)$, we have that $F \lambda_g F = 0$ for all $g \in G \setminus K K^{-1}$. It follows that $F x F = F m_0(x) F$ for all $x \in L(G)$. Writing $\vphi_0 = (\id \ot \om_0) \circ \Phi$, we then also have
$$(1 \ot F) \Phi(x) (1 \ot F) = (1 \ot F) \, (\id \ot m_0)(\Phi(x)) \, (1 \ot F) = (1 \ot F) \Phi(\vphi_0(x)) (1 \ot F) \; .$$
So \eqref{eq.to-do} holds and \eqref{eq.squeeze} is proved.

Combining \eqref{eq.squeeze} and \eqref{eq.stapje} and using that $\zeta_l(\vphi_n(a^*)) \pi(F) \zeta_l(\vphi_0(\vphi_n(a)))$ is an element of $S_1$ with norm at most $2\Lambda(G)^3$, we get that
\begin{align*}
\delta & \leq \limsup_n \Re \Om_1(\zeta_l(\vphi_n(a^*)) \pi(F) \zeta_l(\vphi_0(\vphi_n(a))) \pi(F)) + \eps(\Lambda(G)^2 + 1) \\
&\leq \limsup_n \Re \Om_1(\zeta_l(\vphi_n(a^*)) \pi(F) \zeta_l(\vphi_0(\vphi_n(a)))) + \delta/2 \; .
\end{align*}
As in \eqref{eq.useful} and using the $\zeta_l(A)$-centrality of $\Om_1$, we conclude that
$$
\delta/2 \leq \Re \Om_1(\zeta_l(a^*) \pi(F) \zeta_l(\vphi_0(a))) = \Re \Om_1(\pi(F) \zeta_l(\vphi_0(a)a^*))$$
for all $a \in \cU(A)$. For every $T \in S_1\dpr$ and $x \in M$, we have
$$|\Om_1(T \zeta_l(x))|^2 \leq \Om_1(T T^*) \, \Tr(px^* xp) \; .$$
So we find a unique $\eta \in L^2(Mp)$ such that
$$\Om_1(\pi(F) \zeta_l(x)) = \langle xp , \eta \rangle \quad\text{for all}\;\; x \in M \; .$$
It then follows that
$$\delta/2 \leq \Re \langle \vphi_0(a)a^*, \eta \rangle \quad\text{for all}\;\; a \in \cU(A) \; .$$
Since $\om_0 \in A(G)$, we can take $\xi_1,\xi_2 \in L^2(G)$ such that $\om_0(g) = \langle \lambda_g \xi_1,\xi_2 \rangle$ for all $g \in G$. It follows that
$$\delta/2 \leq \Re \langle \Phi(a) (p \ot \xi_1) a^* , \eta \ot \xi_2 \rangle \quad\text{for all}\;\; a \in \cU(A) \; .$$
Defining $\xi_3 \in \Phi(p) (L^2(Mp) \ot L^2(G))$ as the element of minimal norm in the closed convex hull of $\{\Phi(a) (p \ot \xi_1) a^* \mid a \in \cU(A) \}$, we conclude that $\xi_3$ satisfies $\Phi(a) \xi_3 = \xi_3 a$ for all $a \in A$ and that $\Re \langle \xi_3 , \eta \ot \xi_2 \rangle \geq \delta/2$. So, $\xi_3 \neq 0$ and we have proven that $A$ can be $\Phi$-embedded.

\section{Uniqueness of Cartan subalgebras; proof of Theorem \ref{thm.unique-Cartan}}\label{sec.proof-unique-cartan}

Theorem \ref{thm.unique-Cartan} is a special case of the following general result. To formulate this result, recall that a nonsingular action $G \actson (X,\mu)$ of a lcsc group $G$ on a standard probability space is called \emph{amenable in the sense of Zimmer} if there exists a $G$-equivariant conditional expectation $E : L^\infty(X \times G) \recht L^\infty(X)$ w.r.t.\ the action $G \actson X \times G$ given by $g \cdot (x,h) = (g \cdot x, gh)$.

\begin{theorem}\label{thm.unique-Cartan-general}
Let $G = G_1 \times \cdots \times G_n$ be a direct product of lcsc weakly amenable groups with property~(S). Let $G \actson (X,\mu)$ be an essentially free nonsingular action. Denote by $G_i^\circ$ the direct product of all $G_j$, $j \neq i$.

If for every $i \in \{1,\ldots,n\}$ and every non-null $G$-invariant Borel set $X_0 \subset X$, the action $G_i \actson L^\infty(X_0)^{G_i^\circ}$ is nonamenable in the sense of Zimmer, then $L^\infty(X) \rtimes G$ has a unique Cartan subalgebra up to unitary conjugacy.

In particular, $L^\infty(X) \rtimes G$ has a unique Cartan subalgebra up to unitary conjugacy when the groups $G_i$ are nonamenable and the action $G \actson (X,\mu)$ is either probability measure preserving or irreducible.
\end{theorem}

{\bf Cross section equivalence relations.}\label{discussion.cross-section} Theorem \ref{thm.unique-Cartan-general} is proven by using \emph{cross section equivalence relations}. These were introduced in \cite{Fo74,Co79} and a rather self-contained approach can be found in \cite[Section 4.1 and Appendix B]{KPV13}.

Let $G \actson (X,\mu)$ be a nonsingular action of a lcsc group $G$ on the standard probability space $(X,\mu)$. This means that $X$ is a standard Borel space and that $G \actson X$ is a Borel action that leaves the measure $\mu$ quasi-invariant in the sense that $\mu(g \cdot \cU) = 0$ if and only if $\mu(\cU) = 0$, whenever $\cU \subset X$ is Borel and $g \in G$. Assume that this action is \emph{essentially free}, meaning that almost every point $x \in X$ has a trivial stabilizer. Since the set of points $x \in X$ having a trivial stabilizer is a Borel subset of $X$, we may equally well assume that the action is really free.

A \emph{cross section} for $G \actson (X,\mu)$ is a Borel subset $X_1 \subset X$ with the following two properties.
\begin{itemlist}
\item There exists a neighborhood $\cU$ of $e$ in $G$ such that the map $\cU \times X_1 \recht X : (g,x) \mapsto g \cdot x$ is injective.
\item The subset $G \cdot X_1 \subset X$ is conull.
\end{itemlist}
Note that the first condition implies that the map $G \times X_1 \recht X : (g,x) \mapsto g \cdot x$ is countable-to-one and thus, maps Borel sets to Borel sets. So, the set $G \cdot X_1$ appearing in the second condition is Borel.

A \emph{partial cross section} for $G \actson (X,\mu)$ is a Borel subset $X_1 \subset X$ satisfying the first condition and satisfying the property that $G \cdot X_1$ is non-null.

Given any partial cross section $X_1$, the equivalence relation $\cR$ on $X_1$ defined by
$$\cR = \{(y,y') \in X_1 \times X_1 \mid y \in G \cdot y'\}$$
is Borel and has countable equivalence classes. Also, $X_1$ has a canonical measure class, given by a probability measure $\mu_1$, and this measure $\mu_1$ is quasi-invariant under the equivalence relation $\cR$. The \emph{cross section equivalence relation} $\cR$ on $(X_1,\mu_1)$ is thus a countable nonsingular Borel equivalence relation.

By construction, the von Neumann algebra $L(\cR)$ is canonically isomorphic with $q (L^\infty(X) \rtimes G) q$, for some projection $q$, see e.g.\ \cite[Lemma 4.5]{KPV13}. In this way, cross sections define the canonical Cartan subalgebra $L^\infty(X) \rtimes G$.

As will become clear in the proof of Theorem \ref{thm.unique-Cartan-general}, it is useful to allow $\mu$ to be a $\sigma$-finite measure and to consider the special case where $\mu$ is scaled by the inverse of the modular function of $G$, meaning that $\mu(g \cdot \cU) = \delta(g)^{-1} \mu(\cU)$ for all Borel sets $\cU \subset X$ and all $g \in G$. This covers in particular the case where $\mu$ is a $G$-invariant probability measure and $G$ is unimodular. Fix a \emph{right invariant} Haar measure $\lambda$ on $G$ and recall that $\lambda(g \cU) = \delta(g)^{-1} \lambda(\cU)$ for all Borel sets $\cU \subset G$ and $g \in G$.

Let $\mu$ be a $\si$-finite measure on $X$ that is scaled by the inverse of the modular function.
Let $X_1 \subset X$ be any partial cross section (and note that the definitions above only depend on the measure class of $\mu$ so that taking $\mu$ to be $\sigma$-finite makes no difference). Then there is a \emph{unique} $\sigma$-finite measure $\mu_1$ on $X_1$ such that the following holds: whenever $\cU$ is a neighborhood of $e$ in $G$ such that the map $\Psi : \cU \times X_1 \recht X : (g,x) \mapsto g \cdot x$ is injective, we have $\Psi_*(\lambda|_{\cU} \times \mu_1) = \mu|_{\cU \cdot X_1}$. This measure $\mu_1$ is \emph{invariant} under the cross section equivalence relation.

In the case where $G$ is unimodular and $\mu$ is a $G$-invariant probability measure, we get that $\mu_1$ is a finite $\cR$-invariant measure. It is then more customary to normalize $\mu_1$, so that $\mu_1$ becomes an $\cR$-invariant probability measure on $X_1$ and
$$\Psi_*(\lambda|_{\cU} \times \mu_1) = \covol(X_1) \, \mu|_{\cU \cdot X_1} \; .$$
The scaling factor $\covol(X_1)$ is called the \emph{covolume} of $X_1$. Note that this covolume is proportional to the choice of the Haar measure on $G$.

{\bf\boldmath The coaction $\Phi_\om$ associated with a cocycle $\om$ and $\om$-compactness.} Let $\cR$ be a countable pmp equivalence relation on the standard probability space $(X_1,\mu_1)$. Denote by $[\cR]$ its \emph{full group}, i.e.\ the group of all pmp isomorphisms $\vphi : X_1 \recht X_1$ with the property that $(\vphi(x),x) \in \cR$ for all $x \in X_1$. Denote by $[[\cR]]$ the \emph{full pseudogroup} of $\cR$, consisting of all partial measure preserving transformations with the property that $(\vphi(x),x) \in \cR$ for all $x \in \dom(\vphi)$. The tracial von Neumann algebra $M = L(\cR)$ is generated by the Cartan subalgebra $L^\infty(X_1)$ and the unitary elements $u_\vphi$, $\vphi \in [\cR]$, normalizing $L^\infty(X_1)$. Similarly, every $\vphi \in [[\cR]]$ defines a partial isometry $u_\vphi \in M$. Finally, $\cR$ is equipped with a natural $\sigma$-finite measure and $L^2(M)$ is naturally identified with $L^2(\cR)$.

Let $G$ be a lcsc group and $\om : \cR \recht G$ a cocycle, i.e.\ a Borel map satisfying
$$\om(x,y) \, \om(y,z) = \om(x,z) \quad\text{for a.e.}\;\; (x,y,z) \in \cR^{(2)} \; ,$$
where $\cR^{(2)} = \{(x,y,z) \in X_1 \times X_1 \times X_1 \mid (x,y) \in \cR \;\;\text{and}\;\; (y,z) \in \cR \}$ is equipped with its natural $\sigma$-finite measure. Note that $\om(x,x) = e$ for a.e.\ $x \in X$.

We say that a von Neumann subalgebra $B \subset pMp$ is \emph{$\om$-compact} if for every $\eps > 0$, there exists a compact subset $K \subset G$ such that
$$\|b - P^\om_K(b) \|_2 \leq \eps \, \|b\| \quad\text{for all}\;\; b \in B \; ,$$
where $P^\om_K$ is the orthogonal projection of $L^2(\cR)$ onto $L^2(\om^{-1}(K))$.

Given a von Neumann subalgebra $B \subset M$, the same argument as in the proof of \cite[Proposition 2.6]{Va10b} implies that the set of projections
$$\{p \in B' \cap M \mid \; \text{$p$ is a projection and $B p$ is $\om$-compact}\;\}$$
attains its maximum in a unique projection $p$ and that this projection $p$ belongs to $\cN_M(B)' \cap M$.

We associate to $\om$ the coaction
$$\Phi_\om : M \recht M \ovt L(G) : \Phi_\om(F) = F \ot 1 \;\; , \;\; \Phi_\om(u_\vphi) = (u_\vphi \ot 1) V_\vphi \;\; ,$$
for all $F \in L^\infty(X_1)$ and all $\vphi \in [\cR]$, where $V_\vphi \in L^\infty(X_1) \ovt L(G)$ is given by $V_\vphi(x) = \lambda_{\om(\vphi(x),x)}$.

We deduce from Theorem \ref{thm.main-tech} the following result.

\begin{theorem}\label{thm.reinterpret}
Let $\cR$ be a countable pmp equivalence relation on the standard probability space $(X_1,\mu_1)$. Let $G$ be a weakly amenable locally compact group with property~(S) and $\om : \cR \recht G$ a cocycle. Write $M = L(\cR)$ and assume that $A \subset M$ is a $\Phi_\om$-amenable von Neumann subalgebra with normalizer $P = \cN_M(A)\dpr$. Denote by $p \in P' \cap M$ the unique maximal projection such that $A p$ is $\om$-compact. Then, $P(1-p)$ is $\Phi_\om$-amenable.
\end{theorem}

\begin{proof}
By Theorem \ref{thm.main-tech}, we only have to prove the following statement: if $p \in A' \cap M$ is a nonzero projection such that $Ap$ can be $\Phi_\om$-embedded, then there exists a nonzero projection $q \in A' \cap M$ such that $q \leq p$ and $A q$ is $\om$-compact. Since $Ap$ can be $\Phi_\om$-embedded, there exists a nonzero vector $\xi \in L^2(M) \ot L^2(G)$ such that $\Phi_\om(p) \xi = \xi = \xi (p \ot 1)$ and such that $\Phi_\om(a) \xi = \xi(a \ot 1)$ for all $a \in Ap$.

Denote by $q$ the smallest projection in $M$ that satisfies $\xi = \xi(q \ot 1)$. Then, $q \in A' \cap M$, $q \leq p$ and $q \neq 0$. Viewing $\xi$ as affiliated with the W$^*$-module $M \ovt L^2(G)$, we can take the polar decomposition of $\xi$ and find $V \in M \ovt L^2(G)$ satisfying $V^* V = q$ and $V a = \Phi_\om(a) V$ for all $a \in A p$.

Define $\cG \subset [[\cR]]$ consisting of all $\vphi \in [[\cR]]$ for which the set $\{\om(\vphi(x),x) \mid x \in \dom(\vphi) \}$ has compact closure in $G$. For every $\vphi \in [\cR]$ and every $\eps > 0$, we can choose a Borel set $\cU \subset X_1$ with $\mu_1(X_1 \setminus \cU)<\eps$ such that the restriction of $\vphi$ to $\cU$ belongs to $\cG$. Therefore, the linear span of all $F u_\vphi$, $F \in L^\infty(X_1)$, $\vphi \in \cG$, defines a dense $*$-subalgebra $M_0$ of $M$. By construction, for every $x \in M_0$, there exists a compact subset $K \subset G$ such that $x = P^\om_K(x)$.

Choose $\eps > 0$. Consider on the W$^*$-module $M \ovt L^2(G)$ the norm $\|\cdot\|_2$ given by the embedding $M \ovt L^2(G) \subset L^2(M) \ot L^2(G)$, as well as the operator norm $\|\cdot\|_\infty$. By the Kaplansky density theorem, we can take $W \in M_0 \otalg C_c(G) \subset M \ovt L^2(G)$ such that $\|W\|_\infty \leq 1$ and $\|V-W\|_2 < \eps/3$ and $\|W^* W - q\|_2 < \eps/3$. For every $a \in M$, we find that
$$\|V a - W a\|_2 \leq \|V - W\|_2 \, \|a\| \leq \frac{\eps}{3} \, \|a\| \quad\text{and}\quad \|\Phi_\om(a) V - \Phi_\om(a) W\|_2 \leq \|a\| \, \|V - W\|_2  \leq \frac{\eps}{3} \, \|a\| \; .$$
Therefore, $\|\Phi_\om(a) W - W a\|_2 \leq \frac{2\eps}{3} \, \|a\|$ for all $a \in A p$. Since $\|W\|_\infty \leq 1$ and $\|W^* W - q\|_2 < \eps/3$, we find that
\begin{equation}\label{eq.poma}
\|W^* \Phi_\om(a) W - a q\|_2 \leq \eps \, \|a\|
\end{equation}
for all $a \in Ap$.

When $\xi_i \in C_c(G)$ have (compact) supports $K_i \subset G$, then $(1 \ot \xi_2^*) \Phi_\om(x) (1 \ot \xi_1)$ belongs to $L^2(\om^{-1}(K_2 K_1^{-1}))$ for all $x \in M$. So because $W \in M_0 \otalg C_c(G)$, we can take a compact subset $K \subset G$ such that $W^* \Phi_\om(x) W$ belongs to the range of $P^\om_K$ for every $x \in M$. It then follows from \eqref{eq.poma} that $\|P^\om_K(aq) - aq\|_2 \leq \eps \, \|a\|$ for all $a \in Ap$. For every element $a \in Aq$, we can choose $a_1 \in Ap$ with $\|a_1\| = \|a\|$ and $a = a_1 q$. So  we have proved that $\|P^\om_K(a) - a\|_2 \leq \eps \, \|a\|$ for all $a \in Aq$. Since $\eps > 0$ was arbitrary, this means that $Aq$ is $\om$-compact.
\end{proof}

In the formulation of Corollary \ref{cor.reinterpret} below, we make use of the following notion of an amenable pair of group actions, as introduced in \cite{AD81}. Let $G$ be a lcsc group and let $G \actson (Y,\eta)$ and $G \actson (X,\mu)$ be nonsingular actions. Assume that $p : Y \recht X$ is a $G$-equivariant Borel map such that the measures $p_*(\eta)$ and $\mu$ are equivalent. Following \cite[D\'{e}finition 2.2]{AD81}, the pair $(Y,X)$ is called amenable if there exists a $G$-equivariant conditional expectation $L^\infty(Y,\eta) \recht L^\infty(X,\mu)$. In particular, the action $G \actson (X,\mu)$ is amenable in the sense of Zimmer if and only if the pair $(X \times G , X)$ with $g \cdot (x,h) = (g \cdot x,gh)$ is amenable.

\begin{corollary}\label{cor.reinterpret}
Let $G$ be a lcsc group and $G \actson (X,\mu)$ an essentially free nonsingular action on the standard $\sigma$-finite measure space $(X,\mu)$. Assume that the action scales the measure $\mu$ by the inverse of the modular function of $G$. Let $(X_1,\mu_1)$ be a partial cross section with $\mu_1(X_1) < \infty$ and denote by $\cR$ the cross section equivalence relation on $(X_1,\mu_1)$, which is a countable equivalence relation with invariant probability measure $\mu_1(X_1)^{-1} \, \mu_1$.

Let $H$ be a weakly amenable locally compact group with property~(S) and $\pi : G \recht H$ a continuous group homomorphism. Denote by $\om : \cR \recht H$ the cocycle given by the composition of $\pi$ and the canonical cocycle $\om_0 : \cR \recht G$ determined by $\om_0(x',x) \cdot x = x'$ for all $(x',x) \in \cR$.

Let $A \subset L(\cR)$ be a Cartan subalgebra. If $A$ is not $\om$-compact, then there exists a non-null $G$-invariant Borel set $X_0 \subset X$ and a $G$-equivariant conditional expectation $L^\infty(X_0 \times H) \recht L^\infty(X_0)$ w.r.t.\ the action $g \cdot (x,h) = (g \cdot x,\pi(g) h)$.
\end{corollary}

\begin{proof}
Write $M = L(\cR)$. Assume that the Cartan subalgebra $A \subset M$ is not $\om$-compact. By Theorem \ref{thm.reinterpret}, we can take a nonzero central projection $p \in \cZ(M)$ such that $M p$ is $\Phi_\om$-amenable. Write $p = 1_{X_2}$, where $X_2 \subset X_1$ is an $\cR$-invariant Borel set. Put $X_0 = G \cdot X_2$. Then $X_0$ is a non-null $G$-invariant Borel set. We prove that there exists a $G$-equivariant conditional expectation $L^\infty(X_0 \times H) \recht L^\infty(X_0)$.

Since $Mp$ is $\Phi_\om$-amenable and since $p \in L^\infty(X_1)$ so that $\Phi_\om(p) = p \ot 1$, there exists a conditional expectation $M p \ovt B(L^2(H)) \recht \Phi_\om(Mp)$. Since $(X_1,\mu_1)$ is a partial cross section, we can choose a compact neighborhood $K$ of $e$ in $G$ such that $\Psi : K \times X_1 \recht X  : \Psi(k,x) = k \cdot x$ is injective.

Write $N = L^\infty(X) \rtimes G$ and define the coaction $\Phi_\pi : N \recht N \ovt L(H)$ given by
$$\Phi_\pi(F u_g) = F u_g \ot u_{\pi(g)} \quad\text{for all}\;\; F \in L^\infty(X) , g \in G \; .$$
Define the projection $q_1 \in L^\infty(X)$ given by $q_1 = 1_{K \cdot X_1}$. In \cite[Lemma 4.5]{KPV13}, an explicit isomorphism
\begin{equation}\label{eq.iso-with-corner}
q_1 N q_1 \cong B(L^2(K)) \ovt M
\end{equation}
is constructed. Under this isomorphism, the restriction of $\Phi_\pi$ to $q_1 N q_1$ is unitarily conjugate with $\id \ot \Phi_\om$ and the projection $q_2 = 1_{K \cdot X_2}$ corresponds to $1 \ot p$. We thus conclude that there exists a conditional expectation $q_2 N q_2 \ovt B(L^2(H)) \recht \Phi_\pi(q_2 N q_2)$.

Since $q_0 = 1_{X_0}$ is the central support of $q_2$ inside $N$, there also exists a conditional expectation $E : N q_0 \ovt B(L^2(H)) \recht \Phi_\pi(N q_0)$.

We now restrict $E$ to $L^\infty(X_0 \times H) \subset N q_0 \ovt B(L^2(H))$. For all $F \in L^\infty(X_0 \times H)$ and $F' \in L^\infty(X_0)$, we have
$$E(F) \, \Phi_\pi(F') = E(F \, \Phi_\pi(F')) = E(F \, (F' \ot 1)) = E((F' \ot 1) \, F) = E(\Phi_\pi(F') F) = \Phi_\pi(F') \, E(F) \; .$$
Since $L^\infty(X_0) \subset N q_0$ is maximal abelian, it follows that $E(F) \in \Phi_\pi(L^\infty(X_0)) = L^\infty(X_0) \ot 1$. Define the conditional expectation $E_0 : L^\infty(X_0 \times H) \recht L^\infty(X_0)$ such that $E(F) = E_0(F) \ot 1$. Since the action $G \actson L^\infty(X_0 \times H)$ is implemented by the unitary operators $\Phi_\pi(u_g q_0)$, $g \in G$, it follows that $E_0$ is $G$-equivariant. This concludes the proof of the corollary.
\end{proof}

\begin{lemma}\label{lem.reduce-to-core}
Let $N$ be a $\sigma$-finite von Neumann algebra. Assume that the Connes-Takesaki continuous core $\core(N) = N \rtimes_{\sigma^\vphi} \R$ has at most one Cartan subalgebra up to unitary conjugacy. Then the same holds for $N$ itself.
\end{lemma}
\begin{proof}
Let $A$ and $B$ be Cartan subalgebras of $N$. Denote by $E_A : N \recht A$ and $E_B : N \recht B$ the unique faithful normal conditional expectations. Let $z \in \cZ(N)$ be a nonzero central projection. Note that $z \in A \cap B$. The main part of the proof consists in showing that $Az \prec_{Nz} Bz$, where we use the type III variant of Popa's intertwining relation \cite[Section 2]{Po03} as defined in \cite[Definition 2.4]{HV12}. Assuming that $Az \not\prec_{Nz} Bz$, we deduce a contradiction.

By \cite[Theorem 2.3]{HV12}, there exists a net of unitaries $a_n \in \cU(Az)$ such that
\begin{equation}\label{eq.convergence-we-have}
E_B(x^* a_n y) \recht 0 \quad\text{$*$-strongly for all $x,y \in N$.}
\end{equation}

Choose faithful normal states $\vphi$ on $A$ and $\psi$ on $B$. Still denote by $\vphi$ and $\psi$ the faithful normal states on $N$ given by $\vphi \circ E_A$, resp.\ $\psi \circ E_B$.
The continuous core of $N$ can then be realized as $\core_\vphi(N) = N \rtimes_{\sigma^\vphi} \R$ and as $\core_\psi(N) = N \rtimes_{\sigma^\psi} \R$. Denote by $\Theta : \core_\vphi(N) \recht \core_\psi(N)$ the canonical $*$-isomorphism given by Connes' Radon-Nikodym theorem.

Write $\Atil = \Theta(A \rtimes_{\sigma^\vphi} \R)$ and $\Btil = B \rtimes_{\sigma^\psi} \R$. Write $M = \core_\psi(N)$ and note that $\Atil \subset M$ and $\Btil \subset M$ are Cartan subalgebras. Denote by $\Tr$ the canonical faithful normal semifinite trace on $M$ and let $E_{\Btil} : M \recht \Btil$ be the trace preserving conditional expectations. Note that $E_{\Btil}(x) = E_B(x)$ for all $x \in N$. We prove that
\begin{equation}\label{eq.desired-convergence}
E_{\Btil}(x^* a_n y) \recht 0 \quad\text{$*$-strongly for all $x,y \in M$.}
\end{equation}
Since $a_n$ is a net of unitaries in $\Atil z$, once \eqref{eq.desired-convergence} is proved, it follows that the Cartan subalgebras $\Atil$ and $\Btil$ cannot be unitarily conjugate, contradicting the assumptions of the theorem. So once \eqref{eq.desired-convergence} is proved, it follows that $A z \prec_{Nz} Bz$.

Since $\Btil$ is abelian, to prove \eqref{eq.desired-convergence}, it suffices to prove that
\begin{equation}\label{eq.easier-convergence}
\lim_n \| E_{\Btil}(x^* a_n y) \|_{2,\Tr} = 0 \quad\text{for all $x,y \in M$ with $\Tr(x^* x) < \infty$ and $\Tr(y^* y) < \infty$.}
\end{equation}
Approximating $x,y$ in $\| \cdot \|_{2,\Tr}$, it suffices to prove \eqref{eq.easier-convergence} for all $x,y$ of the form $x = x_1 x_0$ and $y = y_1 y_0$ with $x_1,y_1 \in N$ and $x_0,y_0 \in \Btil$ with $\Tr(x_0^* x_0) < \infty$ and $\Tr(y_0^* y_0) < \infty$. But then,
$$E_{\Btil}(x^* a_n y) = x_0^* \, E_{\Btil}(x_1^* a_n y_1) \, y_0 = x_0^* \, E_B(x_1^* a_n y_1) \, y_0 \; ,$$
so that \eqref{eq.easier-convergence} follows from \eqref{eq.convergence-we-have}.

Thus, \eqref{eq.desired-convergence} is proved. As we already explained, it follows that $A z \prec_{Nz} Bz$ for every nonzero central projection $z \in \cZ(N)$.

Let $z_0 \in \cZ(N)$ be the maximal central projection such that $A z_0$ and $B z_0$ are unitarily conjugate inside $N z_0$. Assume that $z_0 \neq 1$ and put $z = 1- z_0$. By the above, $Az \prec_{Nz} Bz$. By the type III version of Popa's \cite[Theorem A.1]{Po01} proved in \cite[Theorem 2.5]{HV12}, there exists a nonzero central projection $z_1 \in \cZ(N) z$ such that $A z_1$ and $B z_1$ are unitarily conjugate. This contradicts the maximality of $z_0$, so that $z_0 = 1$.
\end{proof}

We are now ready to prove Theorem \ref{thm.unique-Cartan-general}.

\begin{proof}[{Proof of Theorem \ref{thm.unique-Cartan-general}}]
Take $G = G_1 \times \cdots \times G_n$ as in the formulation of the theorem. Let $G \actson (X,\mu)$ be an essentially free nonsingular action and assume that the hypotheses of the theorem hold. We have to prove that $N = L^\infty(X) \rtimes G$ has a unique Cartan subalgebra up to unitary conjugacy. By Lemma \ref{lem.reduce-to-core}, it is enough to prove that the continuous core $\core(N)$ has a unique Cartan subalgebra up to unitary conjugacy.

The continuous core $\core(N)$ can be realized as a crossed product $\core(N) = L^\infty(\Xtil) \rtimes G$ where $G \actson (\Xtil,\mutil)$ is the \emph{Maharam extension} given by
$$\Xtil = X \times \R \;\; , \;\; g \cdot (x,t) = \bigl( \, g \cdot x \, , \, t + \log(\delta(g)) + \log(D(g,x))\, \bigr) \; \; , \;\; d\mutil(x,t) = d\mu(x) \times \exp(-t) dt \;  ,$$
where $\delta : G \recht \R_*^+$ is the modular function of $G$ and $D$ is the Radon-Nikodym cocycle for $G \actson (X,\mu)$ determined by
$$\int_X F(g^{-1} \cdot x) \, d\mu(x) = \int_X F(x) \, D(g,x) \, d\mu(x) \; .$$
Note that the action $G \actson \Xtil$ scales the measure $\mutil$ with $\delta^{-1}$.

Let $(X_1,\mu_1)$ be a cross section for $G \actson (\Xtil,\mutil)$. Denote by $\cR_1$ the cross section equivalence relation on $(X_1,\mu_1)$. To prove that $\core(N)$ has a unique Cartan subalgebra, it suffices to prove that $L^\infty(X_1)$ is the unique Cartan subalgebra of $L(\cR_1)$, up to unitary conjugacy. So it suffices to prove that for every non-null Borel set $X_2 \subset X_1$ with $\mu_1(X_2) < \infty$, the restricted equivalence relation $\cR = (\cR_1)|_{X_2}$ has the property that $L^\infty(X_2)$ is the unique Cartan subalgebra of $L(\cR)$ up to unitary conjugacy. Denote by $\mu_2$ the restriction of $\mu_1$ to $X_2$. Then $(X_2,\mu_2)$ is a partial cross section for $G \actson (\Xtil,\mutil)$ and $\mu_2(X_2) < \infty$. Let $A \subset L(\cR)$ be another Cartan subalgebra.

Denote by $\om : \cR \recht G$ the canonical cocycle determined by $\om(x',x) \cdot x = x'$ for all $(x',x) \in \cR$. We claim that $A$ is $\om$-compact. Denote by $\pi_i : G \recht G_i$ the quotient maps and put $\om_i = \pi_i \circ \om$. To prove the claim that $A$ is $\om$-compact, it suffices to prove that $A$ is $\om_i$-compact for every $i \in \{1,\ldots,n\}$. Fix such an $i$ and assume that $A$ is not $\om_i$-compact.

By Corollary \ref{cor.reinterpret}, we find a non-null $G$-invariant Borel set $\Xtil_0 \subset \Xtil$ and a $G$-equivariant conditional expectation $E_0 : L^\infty(\Xtil_0 \times G_i) \recht L^\infty(\Xtil_0)$ w.r.t.\ the action $g \cdot ((x,t), g') = (g \cdot (x,t) , \pi_i(g) g')$.

Denote by $(\beta_s)_{s \in \R}$ the action of $\R$ on $L^\infty(\Xtil)$ given by $s \cdot (g,t) = (g, t+ s)$. Note that this action of $\R$ commutes with the above $G$-action. Write $p=1_{\Xtil_0}$ and denote by $q$ the smallest $(\beta_s)_{s \in \R}$-invariant projection in $L^\infty(\Xtil)$ with $p \leq q$. Note that $q = 1_{X_0 \times \R}$, where $X_0 \subset X$ is a $G$-invariant Borel set. Choose $s_k \in \R$, with $s_0 = 0$, such that $q = \bigvee_{k=0}^\infty \beta_{s_k}(p)$. Inductively define the $G$-invariant projections $p_k \in L^\infty(\Xtil_0)$ given by $p_0 = p$ and
$$p_k = p \, \Bigl(1 - \sum_{i=0}^{k-1} \beta_{s_i-s_k}(p_i) \Bigr) \quad\text{for all $k \geq 1$.}$$
By construction, $q = \sum_k \beta_{s_k}(p_k)$. Choosing a point-weak$^*$ limit point of the sequence
$$E_n : L^\infty(X_0 \times \R \times G_i) \recht L^\infty(X_0 \times \R) : E_n(F) = \sum_{k=0}^n \beta_{s_k} \bigl( E_0\bigl((p_k \ot 1) \, (\beta_{-s_k} \ot \id)(F) \bigr)\bigr) \; ,$$
we obtain a $G$-equivariant conditional expectation $E : L^\infty(X_0 \times \R \times G_i) \recht L^\infty(X_0 \times \R)$. Since $\R$ is amenable, we can take a mean over $\R$ of $\beta_s \circ E \circ (\beta_{-s} \ot \id)$, so that we may assume that $E$ is $G \times \R$-equivariant.

The restriction of $E$ to $L^\infty(X_0) \ovt 1 \ovt L^\infty(G_i)$ then has its image in $L^\infty(X_0 \times \R)^\R = L^\infty(X_0) \ot 1$. So, we find a $G$-equivariant conditional expectation $L^\infty(X_0 \times G_i) \recht L^\infty(X_0)$. Restricting to $L^\infty(X_0)^{G_i^\circ} \ovt L^\infty(G_i)$, we find a $G_i$-equivariant conditional expectation
$$L^\infty(X_0)^{G_i^\circ} \ovt L^\infty(G_i) \recht L^\infty(X_0)^{G_i^\circ} \; .$$
This precisely means that the action $G_i \actson L^\infty(X_0)^{G_i^\circ}$ is amenable in the sense of Zimmer, contrary to our assumptions.

So the claim that $A$ is $\om$-compact is proved. Take a compact subset $K \subset G$ such that $\|P^\om_K(a)\|_2^2 \geq 1/2$ for all $a \in \cU(A)$. Since $K$ is compact and $\om : \cR \recht G$ is the canonical cocycle, the subset $\om^{-1}(K) \subset \cR$ is bounded, meaning that $\om^{-1}(K)$ is the disjoint union of the graphs of finitely many elements $\vphi_i \in [[\cR]]$, $i=1,\ldots,n$, in the full pseudogroup of $\cR$. But then, writing $B = L^\infty(X_2)$,
$$\|P^\om_K(a)\|_2^2 = \sum_{i=1}^n \|E_B(a u_{\vphi_i}^*)\|_2^2$$
for all $a \in L(\cR)$. Since $\|P^\om_K(a)\|_2^2 \geq 1/2$ for all $a \in \cU(A)$, it follows that $A \prec_{L(\cR)} B$, so that $A$ and $B$ are unitarily conjugate by \cite[Theorem A.1]{Po01}.
\end{proof}

\section{Cocycle and orbit equivalence rigidity; proof of Theorem \ref{thm.Wstar-strong-rigidity}}\label{sec.cocycle-OE-rigidity}

Given an irreducible pmp action of $G = G_1 \times G_2$ on a standard probability space $(X,\mu)$, Monod and Shalom proved in \cite[Theorem 1.2]{MS04} a cocycle superrigidity theorem for non-elementary cocycles $G \times X \recht H$ with values in a closed subgroup $H < \Isom(X)$ of the isometry group of a ``negatively curved'' space. It is therefore not surprising that one can also prove a cocycle superrigidity theorem for cocycles with values in a group $H$ satisfying property~(S). We do this in Theorem \ref{thm.cocycle-superrigidity}.

Applying cocycle superrigidity to the cocycles given by a stable orbit equivalence between essentially free, irreducible pmp actions $G_1 \times G_2 \actson (X,\mu)$ and $H_1 \times H_2 \actson (Y,\eta)$, we obtain the following orbit equivalence strong rigidity theorem (see Theorem \ref{thm.OE-strong-rigidity}): if $G_1$ and $G_2$ are nonamenable, while $H_1$ and $H_2$ have property~(S), the actions must be conjugate.

Again, such an orbit equivalence strong rigidity theorem should not come as a surprise: in \cite[Theorem 40]{Sa09}, Sako proved exactly this result when $G_1,G_2$ and $H_1,H_2$ are \emph{countable} groups in class $\cS$. However, he does not use or prove a cocycle superrigidity theorem.

The main novelty of this section is that our approach is surprisingly simple and short.

Given lcsc groups $G$ and $H$ and a nonsingular action $G \actson (X,\mu)$, a Borel cocycle $\om : G \times X \recht H$ is a Borel map satisfying
$$\om(gh,x) = \om(g,h \cdot x) \, \om(h,x) \quad\text{for all}\;\; g,h \in G, x \in X \; .$$
In a measurable context, the slightly more appropriate notion of cocycle is however the following. Denote by $\cM(X,H)$ the Polish group of Borel functions from $X$ to $H$, modulo functions equal almost everywhere. The group $G$ acts continuously on $\cM(X,H)$ by $(\al_g(F))(x) = F(g^{-1} \cdot x)$. Then a cocycle is a continuous map
$$\om : G \recht \cM(X,H) : g \mapsto \om_g \quad\text{satisfying}\;\; \om_{gh} = \al_{h^{-1}}(\om_g) \, \om_h \;\;\text{for all}\;\; g,h \in G \; .$$
Every Borel cocycle $\om$ gives rise to the cocycle $\om_g = \om(g,\cdot)$. Conversely, every cocycle can be realized by a Borel cocycle after removing from $X$ a $G$-invariant Borel set of measure zero, see e.g.\ \cite[Theorem B.9]{Zi84}.

The (measurable) cocycles $\om$ and $\om'$ are called cohomologous if there exists an element $\vphi \in \cM(X,H)$ such that
$$\om'_g = \al_{g^{-1}}(\vphi) \, \om_g \, \vphi^{-1} \quad\text{for all}\;\; g \in G \; .$$
Borel cocycles $\om, \om' :  G \times X \recht H$ are called cohomologous if there exists a Borel map $\vphi : X \recht H$ such that
$$\om'(g,x) = \vphi(g \cdot x) \, \om(g,x) \, \vphi(x)^{-1} \quad\text{for all}\;\; g \in G , x \in X \; .$$
Again, if two Borel cocycles are measurably cohomologous, then they also are Borel cohomologous on a conull $G$-invariant Borel set.

As in \cite[Theorem 1.2]{MS04}, the following cocycle superrigidity theorem says that every ``non-elementary'' cocycle for an irreducible action $G_1 \times G_2 \actson (X,\mu)$ with values in a group with property~(S) is cohomologous to a group homomorphism. In our context, being ``non-elementary'' is expressed by a non relative amenability property introduced in \cite{AD81} (see the discussion preceding Corollary \ref{cor.reinterpret}).

\begin{theorem}\label{thm.cocycle-superrigidity}
Let $G_1,G_2$ and $H$ be lcsc groups and $G_1 \times G_2 \actson (X,\mu)$ a pmp action with $G_2$ acting ergodically. Assume that $H$ has property~(S).
Let $\om : G_1 \times G_2 \times X \recht H$ be a cocycle. Then at least one of the following statements holds.
\begin{enumlist}
\item There exist closed subgroups $K < \Htil < H$ such that $K$ is compact and $K < \Htil$ is normal, and there exists a continuous group homomorphism $\delta : G_1 \recht \Htil/K$ with dense image such that $\om$ is cohomologous to a cocycle $\om_0$ satisfying $\om_0(g_1 g_2,x) \in \delta(g_1) K$ for all $g_i \in G_i$ and a.e.\ $x \in X$.
\item With respect to the action $G_1 \actson X \times H$ given by $g_1 \cdot (x,h) = (g_1 \cdot x , \om(g_1,x) \, h)$ and the factor map $(x,h) \mapsto x$, there exists a $G_1$-equivariant conditional expectation $L^\infty(X \times H) \recht L^\infty(X)$.
\end{enumlist}
\end{theorem}

\begin{proof}
Throughout the proof, we write $G = G_1 \times G_2$ and we view $G_1$ and $G_2$ as closed subgroups of $G$. We fix a left invariant Haar measure $\lambda$ on $H$. We denote by $h \cdot \xi$ the left translation action of $H$ on $L^2(H)$.

{\bf Formulation of the dichotomy.} We are in precisely one of the following situations.
\begin{enumlist}
\item There exists no sequence $g_n \in G_2$ such that $\om(g_n,\cdot) \recht \infty$ in measure. More precisely, there exists a compact subset $L \subset H$ and an $\eps > 0$ such that for all $g \in G_2$ the set $\{x \in X \mid \om(g,x) \in L\}$ has measure at least $\eps$.

\item There exists a sequence $g_n \in G_2$ such that $\om(g_n,\cdot) \recht \infty$ in measure.
\end{enumlist}

{\bf Case 1.} Fix such a compact set $L \subset H$ and $\eps > 0$. Define the unitary representation
$$\pi : G \recht \cU(L^2(X \times H)) : (\pi(g)^* \xi)(x,h) = \xi(g \cdot x, \om(g,x) h)$$
for all $g \in G, x \in X, h \in H, \xi \in L^2(X \times H)$. Fix a compact subset $L_0 \subset H$ with $\lambda(L_0) > 0$. Given a Borel set $A \subset H$ of finite measure, denote by $1_A \in L^2(H)$ the function equal to $1$ on $A$ and equal to $0$ elsewhere. By our choice of $L$ and $\eps$, we find that
$$\langle \pi(g)^* (1 \ot 1_{LL_0}) , 1 \ot 1_{L_0} \rangle \geq \eps \, \lambda(L_0) \quad\text{for all}\;\; g \in G_2 \; .$$
Taking the unique vector of minimal norm in the closed convex hull of $\{\pi(g) (1 \ot 1_{LL_0}) \mid g \in G_2\}$, it follows that $\pi$ admits a nonzero $G_2$-invariant vector. We thus find a Borel map
$$\xi : X \recht L^2(H) \quad\text{such that}\quad \xi(g_2 \cdot x) = \om(g_2,x) \cdot \xi(x) \;\;\text{for all $g_2 \in G_2$ and a.e.\ $x \in X$,}$$
and such that $\xi$ is not zero a.e. Since $x \mapsto \|\xi(x)\|_2$ is essentially $G_2$-invariant and the action $G_2 \actson (X,\mu)$ is ergodic, we may assume that $\|\xi(x)\|_2 = 1$ for a.e.\ $x \in X$.

Denote by $T \subset L^2(H)$ the unit sphere, defined as $T = \{\xi_0 \in L^2(H) \mid \|\xi_0\|_2 = 1\}$. The left translation action $H \actson T$ has closed orbits and thus $H \backslash T$ is a well defined Polish space. Since the map $x \mapsto H \cdot \xi(x)$ from $X$ to $H \backslash T$ is $G_2$-invariant, it is constant a.e. So we find a unit vector $\xi_0 \in L^2(H)$ and a Borel map $\vphi : X \recht H$ such that $\xi(x) = \vphi(x) \cdot \xi_0$ for a.e.\ $x \in X$. Replacing $\om$ by the cohomologous cocycle given by
$$(g,x) \mapsto \vphi(g \cdot x)^{-1} \, \om(g,x) \, \vphi(x) \; ,$$
we find that $\om(g_2,x) \cdot \xi_0 = \xi_0$ for all $g_2 \in G_2$ and a.e.\ $x \in X$.

Define the closed subgroup $K < H$ given by
$$K = \{s \in H \mid s \cdot \xi_0 = \xi_0 \} \; .$$
Then, $K$ is compact and $\om(g_2,x) \in K$ for all $g_2 \in G_2$ and a.e.\ $x \in X$. By Zimmer's theory for compact group valued cocycles (see \cite[Section 3]{Zi75}), we may further assume that the restricted cocycle
$$\om_2 : G_2 \times X \recht K : \om_2 = \om|_{G_2 \times X}$$
is minimal, in the sense that the associated action $G_2 \actson X \times K$ given by
$$g_2 \cdot (x,k) = (g_2 \cdot x, \om_2(g_2,x) k)$$
is ergodic.

Whenever $g_1 \in G_1$ and $g_2 \in G_2$, we have for a.e.\ $x \in X$
\begin{equation}\label{eq.commutation}
\om(g_1,g_2 \cdot x) \, \om(g_2, x) = \om(g_1 g_2, x) = \om(g_2 g_1,x) = \om(g_2,g_1 \cdot x) \, \om(g_1,x) \; .
\end{equation}
Fix $g_1 \in G_1$. It follows from \eqref{eq.commutation} that for all $g_2 \in G_2$ and a.e.\ $x \in X$, $\om(g_1,g_2 \cdot x) \in K \cdot \om(g_1,x) \cdot K$. Therefore, the map
$$X \recht K \backslash H / K : x \mapsto K \cdot \om(g_1,x) \cdot K$$
is $G_2$-invariant and thus constant a.e. We then find $s \in H$ and Borel maps $\vphi,\psi : X \recht K$ such that ($g_1$ still being fixed) we have $\om(g_1,x) = \vphi(x) \, s \, \psi(x)$ for a.e.\ $x \in X$.

Then \eqref{eq.commutation} becomes
$$\vphi(g_2 \cdot x) \, s \, \psi(g_2 \cdot x) \, \om(g_2,x) = \om(g_2 , g_1 \cdot x) \, \vphi(x) \, s \, \psi(x)$$
for all $g_2 \in G_2$ and a.e.\ $x \in X$. So, the cocycle
$$\om_2' : G_2 \times X \recht K : \om_2'(g_2,x) = \psi(g_2 \cdot x) \, \om_2(g_2,x) \, \psi(x)^{-1}$$
is cohomologous to $\om_2$ (as cocycles for $G_2 \actson X$ with values in the compact group $K$) and takes values in $K \cap s^{-1} K s$. The minimality of $\om_2$ then implies that $K \cap s^{-1} K s = K$. Making a similar reasoning for the cocycle $(g_2,x) \mapsto \om(g_2,g_1 \cdot x)$, which by construction is isomorphic with $\om_2$ and thus minimal as well, we also find that $K \cap s K s^{-1} = K$. Defining the closed subgroup $H' < H$ given by
$$H' := \{ s \in H \mid s K s^{-1}  = K\} \; ,$$
we find that $s \in H'$. By construction, $K < H'$ is normal. We have proved that for every $g_1 \in G_1$, there exists an $s \in H'$ such that $\om(g_1,x) \in s K$ for a.e.\ $x \in X$. We already had $\om(g_2,x) \in K$ for all $g_2 \in G_2$ and a.e.\ $x \in X$. So we find a Borel and thus continuous homomorphism $\delta : G_1 \recht H'/K$ such that $\om(g_1 g_2,x) \in \delta(g_1) K$ for all $g_1 \in G_1$, $g_2 \in G_2$ and a.e.\ $x \in X$. Defining $\Htil < H'$ as the inverse image of the closure of $\delta(G_1)$, the first statement in the theorem holds.

{\bf Case 2.} Fix a sequence $g_n \in G_2$ such that $\om(g_n,\cdot) \recht \infty$ in measure. Fix a map $\eta : H \recht \cS(H)$ as given by property~(S). Define the sequence of Borel maps
$$\eta_n : X \recht \cS(H) : \eta_n(x) = \eta(\om(g_n,x)^{-1}) \; .$$
By \eqref{eq.commutation}, we have for all $g \in G_1$ that
\begin{equation}\label{eq.my-eq-with-om}
\om(g_n,g \cdot x)^{-1} = \om(g,x) \, \om(g_n,x)^{-1} \, \om(g,g_n \cdot x)^{-1} \; .
\end{equation}

Fix $g \in G_1$ and fix $\eps > 0$. Take a compact subset $L \subset H$ such that $\om(g,x) \in L$ for all $x$ in a set of measure at least $1-\eps$. Then take a compact subset $L_1 \subset H$ such that
$$\|\eta(h_1 h h_2^{-1}) - h_1 \cdot \eta(h)\|_1 < \eps$$
for all $h_1,h_2 \in L$ and all $h \in H \setminus L_1$. Finally take $n_0$ such that for all $n \geq n_0$, we have that $\om(g_n,x)^{-1} \in H \setminus L_1$ for all $x$ in a set of measure at least $1-\eps$.

So, for our fixed $g \in G_1$ and for all $n \geq n_0$, there exists a Borel set $X_n \subset X$ of measure at least $1-3 \eps$ such that
$$\om(g,x) \in L \;\; , \;\; \om(g_n,x)^{-1} \in H \setminus L_1 \;\; , \;\; \om(g,g_n \cdot x) \in L$$
for all $x \in X_n$. Applying $\eta$ to \eqref{eq.my-eq-with-om}, we conclude that for our fixed $g \in G_1$ and all $n \geq n_0$, we have
$$\|\eta_n(g \cdot x) - \om(g,x) \cdot \eta_n(x)\|_1 < \eps$$
for all $x \in X_n$. Since $\mu(X_n) \geq 1-3\eps$, we have proved that for every $g \in G_1$, the sequence of functions
\begin{equation}\label{eq.conv-to-zero-in-meas}
x \mapsto \|\eta_n(x) - \om(g,x)^{-1} \cdot \eta_n(g \cdot x)\|_1
\end{equation}
converges to zero in measure.

%

View $\cS(H) \subset L^1(H)$ and define the normal conditional expectations
$$P_n : L^\infty(X \times H) \recht L^\infty(X) : (P_n(F))(x) = \int_H F(x,y) \; (\eta_n(x))(y) \; dy \; .$$
Choose a point-weak$^*$ limit point $P : L^\infty(X \times H) \recht L^\infty(X)$. Since the sequence in \eqref{eq.conv-to-zero-in-meas} converges to zero in measure, $P$ is a $G_1$-equivariant conditional expectation. So the second statement in the theorem holds.
\end{proof}

The cocycle superrigidity theorem \ref{thm.cocycle-superrigidity} implies the following orbit equivalence strong rigidity theorem. As mentioned above, for countable groups, the same result was obtained in \cite[Theorem 40]{Sa09}. Combining Theorem \ref{thm.unique-Cartan} and Theorem \ref{thm.OE-strong-rigidity}, it follows that Theorem \ref{thm.Wstar-strong-rigidity} holds.

Let $G \actson (X,\mu)$ and $H \actson (Y,\eta)$ be essentially free, nonsingular actions of the lcsc groups $G,H$. We say that these actions are \emph{stably orbit equivalent} if they admit cross sections such that the associated cross section equivalence relations are isomorphic.

\begin{theorem}\label{thm.OE-strong-rigidity}
Let $G = G_1 \times G_2$ and $H = H_1 \times H_2$ be unimodular lcsc groups without nontrivial compact normal subgroups. Assume that $G \actson (X,\mu)$ and $H \actson (Y,\eta)$ are essentially free, irreducible pmp actions. Assume that $G_1,G_2$ are nonamenable and that $H_1,H_2$ have property~(S).

If the actions are stable orbit equivalent, they must be conjugate.

More precisely, if $(X_1,\mu_1)$ and $(Y_1,\eta_1)$ are cross sections, with cross section equivalence relations $\cR$ and $\cS$, and if $\pi : X_1 \recht Y_1$ is a nonsingular isomorphism between the equivalence relations $\cR$ and $\cS$, there exist conull $\cR$-invariant (resp.\ $\cS$-invariant) Borel sets $X_2 \subset X_1$ and $Y_2 \subset Y_1$, a Borel bijection $\Delta : G \cdot X_2 \recht H \cdot Y_2$ and a continuous group isomorphism $\delta : G \recht H$ such that
\begin{itemlist}
\item $X_0 = G \cdot X_2$ and $Y_0 = H \cdot Y_2$ are conull Borel sets and $\Delta_*(\mu) = \eta$,
\item $\Delta(g \cdot x) = \delta(g) \cdot \Delta(x)$ for all $g \in G$ and all $x \in X_0$,
\item $\Delta(x) \in H \cdot \pi(x)$ for all $x \in X_2$,
\item $\delta$ is either of the form $\delta_1 \times \delta_2$ where $\delta_i : G_i \recht H_i$ are isomorphisms, or of the form $(g_1,g_2) \mapsto (\delta_2(g_2),\delta_1(g_1))$ where $\delta_1 : G_1 \recht H_2$ and $\delta_2 : G_2 \recht H_1$ are isomorphisms,
\item normalizing the Haar measures $\lambda_G$ and $\lambda_H$ such that $\delta_*(\lambda_G) = \lambda_H$, we have $\covol(X_1) = \covol(Y_1)$.
\end{itemlist}
\end{theorem}
\begin{proof}
Replacing $X$ and $Y$ by a conull $G$-invariant, resp.\ $H$-invariant, Borel set, we may assume that the Borel actions $G \actson X$ and $H \actson Y$ are free.

Recall that $\mu_1$ is the natural $\cR$-invariant probability measure on $X_1$ and that $\eta_1$ is the natural $\cS$-invariant probability measure on $Y_1$. Normalize the Haar measures $\lambda_G$ and $\lambda_H$ such that $\covol(X_1) =1=\covol(Y_1)$. Take compact neighborhoods $\cU$ of $e$ in $G$ and $\cV$ of $e$ in $H$ such that the maps
\begin{equation}\label{eq.maps-Psi-Phi}
\Psi : \cU \times X_1 \recht X : (k,x) \mapsto k \cdot x \quad\text{and}\quad \Phi : \cV \times Y_1 \recht Y : (l,y) \mapsto l \cdot y
\end{equation}
are injective. By the definition of a cross section and its covolume (see page \pageref{discussion.cross-section}), these maps satisfy
$$\Psi_*\bigl((\lambda_G)|_\cU \times \mu_1\bigr) = \mu|_{\cU \cdot X_1} \quad\text{and}\quad \Phi_*\bigl( (\lambda_H)|_\cV \times \eta_1 \bigr) = \eta|_{\cV \cdot Y_1} \; .$$

Replacing $X_1$ and $Y_1$ by conull Borel subsets that are invariant under the cross section equivalence relations, we may assume that $\pi : X_1 \recht Y_1$ is a Borel isomorphism between $\cR$ and $\cS$. Since $\pi_*(\mu_1)$ is an $\cS$-invariant probability measure on $Y_1$ in the same measure class as $\eta_1$, we have $\pi_*(\mu_1) = \eta_1$. Since $G \cdot X_1$ and $H \cdot Y_1$ are conull and Borel, we may assume that $X = G \cdot X_1$ and $Y = H \cdot Y_1$.

We start by translating the stable orbit equivalence $\pi$ into a measure equivalence between $G$ and $H$. This is quite standard: the discrete group case can be found in \cite[Section 3]{Fu98}, but the locally compact case needs a little bit of care.

Choose a Borel map $p : X \recht X_1$ such that $p(x) \in G \cdot x$ for all $x \in X$ and $p(k \cdot x) = x$ for all $k \in \cU$, $x \in X_1$. Extend $\pi$ to the Borel map $\rho : X \recht Y$ defined by $\rho = \pi \circ p$. Similarly choose a Borel map $q : Y \recht Y_1$ and define $\rhotil : Y \recht X : \rhotil = \pi^{-1} \circ q$. By construction, $\rho(G \cdot x) \in H \cdot \rho(x)$ and $\rhotil(H \cdot y) \in G \cdot \rhotil(y)$ for all $x \in X$ and all $y \in Y$. Since the actions $G \actson X$ and $H \actson Y$ are free, we have unique Borel cocycles
\begin{alignat*}{2}
& \om : G \times X \recht H : \rho(g \cdot x) = \om(g,x) \cdot \rho(x) & \quad\text{for all $g \in G, x \in X$,} \\
& \zeta : H \times Y \recht G : \rhotil(h \cdot y) = \zeta(h,y) \cdot \rhotil(y) & \quad\text{for all $h \in H, y \in Y$.}
\end{alignat*}
Since $\rhotil(\rho(x)) \in G \cdot x$ and $\rho(\rhotil(y)) \in H \cdot y$ for all $x \in X$ and all $y \in Y$, we also have unique Borel maps
$$\vphi : X \recht G : \rhotil(\rho(x)) = \vphi(x) \cdot x \quad\text{and}\quad \psi : Y \recht H : \rho(\rhotil(y)) = \psi(y) \cdot y$$
for all $x \in X$, $y \in Y$.

Define the measure preserving Borel actions $G \times H \actson X \times H$ and $G \times H \actson Y \times G$ given by
\begin{equation}\label{eq.my-actions}
\begin{split}
& (g,h) \cdot (x,h') = (g \cdot x, \om(g,x) \, h' \, h^{-1} ) \; , \\
& (g,h) \cdot (y,g') = (h \cdot y, \zeta(h,y) \, g' \, g^{-1}) \; .
\end{split}
\end{equation}
It is straightforward to check that
$$\theta : X \times H \recht Y \times G : \theta(x,h) = \bigl(h^{-1} \cdot \rho(x)  \, , \,  \zeta(h^{-1},\rho(x)\bigr) \, \vphi(x))$$
is a $G \times H$-equivariant Borel map and that $\theta$ is a bijection with inverse
$$\theta^{-1} : Y \times G \recht X \times H : \theta^{-1}(y,g) = \bigl(g^{-1} \cdot \rhotil(y) \, , \, \om(g^{-1},\rhotil(y)) \, \psi(y)\bigr) \; .$$
Using the maps $\Psi$ and $\Phi$ given by \eqref{eq.maps-Psi-Phi}, one checks that
$$\theta(\Psi(k,x),l^{-1}) = (\Phi(l,\pi(x)),k^{-1}) \quad\text{for all}\;\; k \in \cU, l \in \cV, x \in X_1 \; .$$
Since $\Psi$, $\Phi$ and $\pi$ are measure preserving, it follows that the restriction of $\theta$ to $\cU \cdot X_1 \times \cV^{-1}$ is measure preserving. Since the actions of $G \times H$ on $X \times H$ and $Y \times G$ are measure preserving and since the map $\theta$ is $G \times H$-equivariant, it follows that the entire map $\theta$ is measure preserving.

The main part of the proof consists in using Theorem \ref{thm.cocycle-superrigidity} to show that the cocycle $\om$ is cohomologous to an isomorphism of groups $\delta : G \recht H$. Write $\om(g,x) = (\om_1(g,x),\om_2(g,x)) \in H_1 \times H_2$.

{\bf Claim.} W.r.t.\ the actions $G_1 \actson X \times H_i : g_1 \cdot (x,h_i) = (g_1 \cdot x, \om_i(g_1,x) \, h_i)$, there is at most one $i \in \{1,2\}$ for which there exists a $G_1$-equivariant conditional expectation $L^\infty(X \times H_i) \recht L^\infty(X)$.

To prove this claim, assume that such a conditional expectation exists for both $i=1,2$. Then there also exists a $G_1$-equivariant conditional expectation $L^\infty(X \times H) \recht L^\infty(X)$ w.r.t.\ the action $g_1 \cdot (x,h) = (g_1\cdot x, \om(g_1,x) \, h)$. Composing with the $G$-invariant probability measure $\mu$ on $X$, we find a $G_1$-invariant state on $L^\infty(X \times H)$. Using $\theta$, it follows that $L^\infty(Y \times G_1 \times G_2)$ admits a $G_1$-invariant state, w.r.t.\ the action $g_1 \cdot (y,g_1',g_2) = (y,g_1' g_1^{-1},g_2)$. This implies that $G_1$ is amenable, contrary to our assumptions. So, the claim is proved.

Assume that there is no $G_1$-equivariant conditional expectation $L^\infty(X \times H_1) \recht L^\infty(X)$. We prove that the conclusions of the theorem hold with the group isomorphism $\delta : G \recht H$ being of the form $\delta_1 \times \delta_2$. In the case where there is no $G_1$-equivariant conditional expectation $L^\infty(X \times H_2) \recht L^\infty(X)$, we exchange the roles of $H_1$ and $H_2$ and obtain again that the conclusions of the theorem hold with $\delta$ being of the form $(g_1,g_2) \mapsto (\delta_2(g_2),\delta_1(g_1))$.

Applying Theorem \ref{thm.cocycle-superrigidity} to the cocycle $\om_1$, we find a compact subgroup $K_1 < H_1$, a closed subgroup $\Htil_1 < H_1$ with $K_1$ being a normal subgroup of $\Htil_1$, and a continuous group homomorphism $\delta : G_1 \recht \Htil_1/K_1$ with dense image such that $\om_1$ is cohomologous (as a measurable cocycle) with a cocycle $\omtil_1 : G \times X \recht \Htil_1$ satisfying $\omtil_1(g_1 g_2,x) \in \delta_1(g_1) K_1$ for all $g_i \in G_i$ and $x \in X$.

In particular, there is an isomorphism between the actions of $G$ on $L^\infty(X \times H_1)$ induced by $\om_1$ and $\omtil_1$. Using a $K_1$-invariant state on $L^\infty(H_1)$, it follows that there is a $G_2$-equivariant conditional expectation $L^\infty(X \times H_1) \recht L^\infty(X)$. Applying the claim above to $G_2$ instead of $G_1$, it follows that there is no $G_2$-equivariant conditional expectation $L^\infty(X \times H_2) \recht L^\infty(X)$ w.r.t.\ the action induced by $\om_2$.

We can again apply Theorem \ref{thm.cocycle-superrigidity} and altogether, we find compact subgroups $K_i < H_i$, closed subgroups $\Htil_i < H_i$ with $K_i$ being a normal subgroup of $\Htil_i$, and continuous group homomorphisms $\delta_i : G_i \recht \Htil_i/K_i$ with dense image such that, writing $\delta = \delta_1 \times \delta_2$, $K = K_1 \times K_2$, $\Htil = \Htil_1 \times \Htil_2$, the cocycle $\om$ is cohomologous (as a measurable cocycle) with a cocycle $\omtil : G \times X \recht \Htil$ satisfying $\omtil(g,x) \in \delta(g) K$ for all $g \in G$, $x \in X$.

Since the actions of $G \times H$ on $L^\infty(X \times H)$ induced by $\om$ and $\omtil$ are isomorphic, we find an $H$-equivariant embedding of $L^\infty(H / \Htil)$ into $L^\infty(X \times H)^G$. Using $\theta$, we also find such an embedding into $L^\infty(Y \times G)^G = L^\infty(Y)$. Since the elements of $L^\infty(H_1/\Htil_1) \ot 1 \subset L^\infty(H / \Htil)$ are $H_2$-invariant, the irreducibility of the action $H \actson (Y,\eta)$ implies that $\Htil_1 = H_1$. We similarly find that $\Htil_2 = H_2$. Since we assumed that the groups $H_i$ have no nontrivial compact normal subgroups, we also conclude that $K$ is trivial.

We have proved that $\om$ is cohomologous, as a measurable cocycle, with the continuous group homomorphism $\delta : G \recht H$ having dense image. We now prove that $\delta$ is bijective.
Consider the unitary representation $\Pi$ of $G$ on $L^2(X \times H)$ given by $(\Pi(g) \xi)(x,h) = \xi(g^{-1} \cdot x, \delta(g)^{-1}h)$. Combining the map $\theta$ and the fact that the cocycle $\om$ is cohomologous with $\delta$, the representation $\Pi$ is unitarily conjugate to the representation of $G$ on $L^2(Y \times G)$ given by $(g \cdot \xi)(y,g') = \xi(y,g'g)$. Therefore, $\Pi$ is a multiple of the regular representation. In particular, $\Pi$ has $C_0$-coefficients. So,
$$g \mapsto \langle \Pi(g) (1 \ot 1_D), 1 \ot 1_E \rangle = \lambda_H(\delta(g) D \cap E)$$
is a $C_0$-function on $G$ for all compact subsets $D,E \subset H$. It follows that $\Ker \delta$ is a compact subgroup of $G$ and that $\delta$ is proper in the sense that $\delta(g_n) \recht \infty$ whenever $g_n \recht \infty$. By assumption, $G$ has no nontrivial compact normal subgroups. So, $\Ker \delta = \{e\}$ and $\delta$ is injective. Since $\delta$ is proper, the image of $\delta$ is closed. Since $\delta$ has dense image, we conclude that $\delta$ is surjective. So we have proved that $\delta$ is bijective.

Since the Borel cocycles $\om$ and $\delta$ are cohomologous as measurable cocycles, they are also cohomologous as Borel cocycles on a conull $G$-invariant Borel set $X_0 \subset X$. Since $\theta(X_0 \times H)$ is a conull $G \times H$-invariant Borel subset of $Y \times G$, it must be of the form $Y_0 \times H$. So we can restrict everything to $X_0$ and $Y_0$, and assume that $X_0 = X$ and $Y_0 = Y$. Choose a Borel map $\gamma : X \recht H$ such that $\om(g,x) = \gamma(g \cdot x) \, \delta(g) \, \gamma(x)^{-1}$ for all $g \in G$, $x \in X$. Define the measure preserving Borel bijections
$\theta_1,\theta_2 : X \times H \recht X \times H$ given by
$$\theta_1 (x,h) = (x,\gamma(x)h) \quad\text{and}\quad \theta_2(x,h) = (\delta^{-1}(h^{-1}) \cdot x , h^{-1}) \; .$$
Write $\thetatil = \theta \circ \theta_1 \circ \theta_2$. Consider the measure preserving Borel action of $G \times H$ on $X \times H$ given by
$$(g,h) \cdot (x,h') = (\delta^{-1}(h) \cdot x \, , \, h \, h' \, \delta(g)^{-1}) \; .$$
Still using the action of $G \times H$ on $Y \times G$ defined in \eqref{eq.my-actions}, we get that $\thetatil$ is $G \times H$-equivariant. Define the Borel functions $\Delta : X \recht Y$ and $\gammatil : X \recht G$ such that $\thetatil(x,e) = (\Delta(x) , \gammatil(x))$ for all $x \in X$. Then,
\begin{equation}\label{eq.nice-thetatil}
\thetatil(x,h) = \thetatil\bigl((e,\delta^{-1}(h^{-1})) \cdot (x,e)\bigr) = (e,\delta^{-1}(h^{-1})) \cdot \thetatil(x,e) = (\Delta(x) , \gammatil(x) \, \delta^{-1}(h))
\end{equation}
for all $x \in X$ and $h \in H$. Since $\thetatil$ is bijective and $\delta$ is bijective, also $\Delta : X \recht Y$ must be bijective. Since $\thetatil$ is $H$-equivariant, we have $\Delta(g \cdot x) = \delta(g) \cdot \Delta(x)$ for all $x \in X$.

Since $\thetatil$ is measure preserving and $\delta$ is measure scaling, by \eqref{eq.nice-thetatil}, $\Delta$ must be measure scaling. Since both $\mu$ and $\eta$ are probability measures, it follows that $\Delta$ is measure preserving and thus also that $\delta$ is measure preserving. By construction, $\Delta(x) \in H \cdot \rho(x)$ for all $x \in X$ and thus $\Delta(x) \in H \cdot \pi(x)$ for all $x \in X_1$. This concludes the proof of the theorem.
\end{proof}

\begin{remark}\label{rem.precise-form-iso}
Assume that we are in the situation of Theorem \ref{thm.Wstar-strong-rigidity}. Given the more precise description in Theorem \ref{thm.OE-strong-rigidity} of the conjugacy between the actions $G \actson (X,\mu)$, $H \actson (Y,\eta)$ and its relation to the initial stable orbit equivalence, it follows that up to unitary conjugacy, any $*$-isomorphism $\pi : p(L^\infty(X) \rtimes G) p \recht q (L^\infty(Y) \rtimes H) q$ is the restriction of a $*$-isomorphism of the form
$$\pitil : L^\infty(X) \rtimes G \recht L^\infty(Y) \rtimes H : \pitil(u_g F) =  u_{\delta(g)} \, \Delta_*(\Om_g F) \quad\text{for all}\;\; F \in L^\infty(X) , g \in G \; ,$$
where $\Delta : X \recht Y$ is a pmp isomorphism, $\delta : G \recht H$ is a group isomorphism, $\Delta(g \cdot x) = \delta(g) \cdot \Delta(x)$ for all $g \in G$ and a.e.\ $x \in X$, and $\Om_g \in \cU(L^\infty(X))$ is a scalar cocycle, i.e.\ $\Om_g = \Om(g,\cdot)$ where $\Om : G \times X \recht \T$ is a Borel map satisfying $\Om(gh,x) = \Om(g,h\cdot x) \, \Om(h,x)$ for all $g,h \in G$ and a.e.\ $x \in X$.
\end{remark}

\section{Proof of Theorem \ref{thm.main-tech-stable-normalizer}}

Roughly speaking, Theorem \ref{thm.main-tech-stable-normalizer} follows by appropriately combining the proof of \cite[Proposition 3.6]{BHV15} with the setup and methods in the proof of Theorem \ref{thm.main-tech} in Section \ref{sec.proof-main-tech}.

We start by making a first simplification. We replace $M$ by $B(\ell^2(\N)) \ovt B(\ell^2(\N)) \ovt M$ and define the projection $e = 1 \ot 1 \ot p$ in $M$. We then replace $A$ by $\ell^\infty(\N) \ovt B(\ell^2(\N)) \ovt A$ and view it as a von Neumann subalgebra of $e M e$. We finally replace $p$ by the finite trace projection $e_{00} \ot e_{00} \ot p$ and the coaction $\Phi$ by $\id \ot \Phi$. We are now in the following situation: $M$ is a von Neumann algebra with a faithful normal semifinite trace $\Tr$, $e \in M$ is a projection, $A \subset eMe$ is a von Neumann subalgebra with $\Tr|_A$ being semifinite and $p \in A$ is a projection of finite trace. Moreover, by \cite[Lemma 3.5]{BHV15}, for every $x \in \sN_{pMp}(pAp)$, there exist $u \in \cN_{eMe}(A)$ and $a,b \in pAp$ such that $u a = x = b u$. In particular, $\sN_{pMp}(pAp)\dpr = p \cN_{eMe}(A)\dpr p$. Also, for every partial isometry $v \in \sN_{pMp}(pAp)$ with $v^* v = s$ and $vv^* = t$, there exists a $u \in \cN_{eMe}(A)$ such that $u s = v = t u$.

Assuming that $pAp$ is $\Phi$-amenable, we have to prove that $pAp$ can be $\Phi$-embedded or that $\sN_{pMp}(pAp)\dpr = p \cN_{eMe}(A)\dpr p$ is $\Phi$-amenable. Write $q = \Phi(p)$ and $f = \Phi(e)$.

{\bf Step 1.} If $u \in \cN_{eMe}(A)$, then $p (A \cup \{u\})\dpr p$ is still $\Phi$-amenable.

Since $pAp$ is $\Phi$-amenable, there exists a positive functional $\Om$ on $q(M \ovt B(L^2(G)))q$ that is $\Phi(pAp)$-central and that satisfies $\Om(\Phi(x)) = \Tr(x)$ for all $x \in pMp$. Denote by $E : eMe \recht A$ the unique $\Tr$-preserving normal conditional expectation. The functional $\Om$ gives a conditional expectation
$$P : q(M \ovt B(L^2(G)))q \recht \Phi(pAp)$$
satisfying $P(\Phi(x)) = \Phi(E(x))$ for all $x \in pMp$. We have a canonical isomorphism
$$f(M \ovt B(L^2(G)))f \cong B(\ell^2(\N)) \ovt B(\ell^2(\N)) \ovt q(M \ovt B(L^2(G)))q$$
sending $\Phi(A)$ onto $\ell^\infty(\N) \ovt B(\ell^2(\N)) \ovt \Phi(pAp)$. Denoting by $E_0 : B(\ell^2(\N)) \recht \ell^\infty(\N)$ the normal conditional expectation, taking $E_0 \ot \id \ot P$, we can extend $P$ to a conditional expectation
$$P : f(M \ovt B(L^2(G)))f \recht \Phi(A)$$
satisfying $P(\Phi(x)) = \Phi(E(x))$ for all $x \in eMe$.

For every $n \geq 1$, define
$$P_n : f(M \ovt B(L^2(G)))f \recht \Phi(A) : P_n(T) = \frac{1}{n} \sum_{k=1}^n \Phi(u^k) \, P(\Phi(u^{-k}) T \Phi(u^k)) \, \Phi(u^{-k}) \; .$$
Note that every $P_n$ is a conditional expectation satisfying $P_n(\Phi(x)) = \Phi(E(x))$ for all $x \in eMe$. Define
$$P_0 : f(M \ovt B(L^2(G)))f \recht \Phi(A)$$
as a point-weak$^*$ limit point of the sequence $(P_n)_{n \geq 1}$. Then $P_0$ is a conditional expectation satisfying $P_0(\Phi(x)) = \Phi(E(x))$ for all $x \in eMe$ and
$$P_0(\Phi(u^k) T \Phi(u^{-k})) = \Phi(u^k) P_0(T) \Phi(u^{-k}) \quad\text{for all}\;\; T \in f(M \ovt B(L^2(G)))f \; , \; k \in \Z \; .$$
Define the positive functional $\Om_0$ on $q (M \ovt B(L^2(G))) q$ given by $\Om_0(T) = \Tr(\Phi^{-1}(P_0(T)))$, which is well defined because $P_0(T) \in \Phi(pAp)$. By construction, $\Om_0$ is $\Phi(pAp)$-central and $\Om_0(\Phi(x))=\Tr(x)$ for all $x \in pMp$.

Let $k \in \Z$ and $a \in A$. Put $x_0 = p u^k a p$. Note that $x_0 = u^k b p$ where $b \in A$ is defined as $b = u^{-k} p u^k a$. Using the notation $P' = \Phi^{-1} \circ P_0$, we have for every $T \in q (M \ovt B(L^2(G))) q$ that
\begin{align*}
\Om_0(\Phi(x_0) T) &= \Om_0(\Phi(u^k bp) T) = \Tr(P'(\Phi(u^k) \Phi(bp) T)) \\
&= \Tr(u^k bp P'(T \Phi(u^k)) u^{-k}) = \Tr(P'(T \Phi(u^k)) bp) = \Tr(P'(T \Phi(u^k b p))) \\
&= \Om_0(T \Phi(x_0)) \; .
\end{align*}
Since $\Om_0(\Phi(x)) =\Tr(x)$ for all $x \in pMp$ and since the linear span of $\{p u^k a p \mid k \in \Z , a \in A\}$ is strongly dense in $p (A \cup \{u\})\dpr p$, the Cauchy-Schwarz inequality implies that $\Om_0$ is $p (A \cup \{u\})\dpr p$-central. This concludes the proof of step~1.

{\bf Notations.} Since $G$ has CMAP, we can fix a net $\eta_n \in A(G)$ such that the normal completely bounded maps $m_n : L(G) \recht L(G)$ given by $m_n = (\id \ot \eta_n) \circ \Delta$ satisfy the properties in \eqref{eq.cond-wa} with $\Lambda(G) = 1$. Define $\vphi_n : M \recht M$ given by $\vphi_n = (\id \ot \eta_n) \circ \Phi$. Since $\Phi \circ \vphi_n = (\id \ot m_n) \circ \Phi$, also $\|\vphi_n\|\cb \leq 1$ for all $n$ and $\vphi_n(x) \recht x$ strongly for all $x \in M$. Finally, we denote $\psi_n : pMp \recht pMp : \psi_n(x) = p \vphi_n(x) p$.

Whenever $Q \subset eMe$ is a von Neumann subalgebra, we denote by $\cN_Q$ the von Neumann subalgebra of $B(L^2(Me)) \ovt L(G)$ generated by $\Phi(M)$ and $\rho(Q)$, where $\rho(a)$ is given by right multiplication with $a \in Q$. We write $\cN = \cN_A$.

For every partial isometry $v \in \sN_{pMp}(pAp)$ with $s = v^* v$ and $t = v v^*$, denote by
$$\beta_v : \rho(s) \cN \rho(s) \recht \rho(t) \cN \rho(t)$$
the $*$-isomorphism implemented by right multiplication with $v^*$ on $L^2(Me) \ot L^2(G)$. Note that $\beta_v(\Phi(x)\rho(a)) = \Phi(x) \rho(v a v^*)$.

Define $q_1 = \Phi(p) \rho(p)$. We still denote by $\beta_v$ the normal, completely contractive map $q_1 \cN q_1 \recht q_1 \cN q_1$ given by $T \mapsto \beta_v(\rho(s) T \rho(s))$.

{\bf Step 2.} Let $Q \subset eMe$ be a von Neumann subalgebra such that $A \subset Q$ and such that $pQp$ is $\Phi$-amenable. Then there exists a net of functionals $\mu_n^Q \in (q_1 \cN_Q q_1)_*$ with the following properties.
\begin{enumlist}
\item $\mu_n^Q(\Phi(x)\rho(a)) = \Tr(\psi_n(x) a)$ for all $x \in pMp$ and $a \in p Q p$.
\item $\|\mu_n^Q\| \leq \Tr(p)$ for all $n$.
\end{enumlist}

To prove step~2, in the same way as in step~1 of the proof of Theorem \ref{thm.main-tech}, using the $\Phi$-amenability of $pQp$, we find normal completely contractive maps $\theta_n : q_1 \cN q_1 \recht B(L^2(pMp))$ satisfying $\theta_n(\Phi(x)\rho(a))= \psi_n(x) \rho(a)$ for all $x \in pMp$ and $a \in pQp$. Composing with the vector functional $T \mapsto \langle T p, p\rangle$, which has norm $\Tr(p)$, the proof of step~1 is complete.

{\bf Step 3.} The positive functionals $\om_n = |\mu_n^A|$ in $(q_1 \cN q_1)_*$ satisfy
\begin{enumlist}
\item $\lim_n \om_n(\Phi(x)) = \Tr(x)$ for all $x \in pMp$,
\item $\lim_n \om_n(\Phi(a) \rho(a^*)) = \Tr(p)$ for all $a \in \cU(pAp)$,
\item for every partial isometry $v \in \sN_{pMp}(pAp)$, we have that $\lim_n \| \om_n  \circ \beta_{v^*} - \om_n \circ \Ad \Phi(v)\| = 0$.
\end{enumlist}
Note that, as defined above, the functional $\om_n  \circ \beta_{v^*}$ on $q_1 \cN q_1$ is given by $(\om_n  \circ \beta_{v^*})(\Phi(x) \rho(a)) = \om_n(\Phi(x)\rho(v^* a v))$ for all $x \in pMp$ and $a \in pAp$, while the functional $\om_n \circ \Ad \Phi(v)$ on $q_1 \cN q_1$ is given by $(\om_n \circ \Ad \Phi(v))(\Phi(x) \rho(a)) = \om_n(\Phi(v x v^*) \rho(a))$.

To prove step~3, let $Q \subset eMe$ be a von Neumann subalgebra such that $A \subset Q$ and such that $pQp$ is $\Phi$-amenable. Define $\mu_n^Q$ as in step~2 and denote $\om_n^Q = |\mu_n^Q|$. Since $\|\mu_n^Q\| \leq \Tr(p)$ for all $n$ and $\lim_n \mu_n^Q(q_1) = \Tr(p)$, we find that $\lim_n \|\mu_n^Q - \om_n^Q\| = 0$. Whenever $a \in \cU(pQp)$, we get that $\lim_n \mu_n^Q(\Phi(a)\rho(a^*)) = \Tr(p)$ and thus also, $\lim_n \om_n^Q(\Phi(a)\rho(a^*)) = \Tr(p)$. So the first two properties in step~3 are already proven. It also follows that
$$\lim_n \| (\Phi(a) \rho(a^*))\cdot \om_n^Q - \om_n^Q\| = 0 = \lim_n \| \om_n^Q \cdot (\Phi(a) \rho(a^*)) - \om_n^Q \|$$
for all $a \in \cU(pQp)$ and thus,
\begin{equation}\label{eq.omnQ}
\lim_n \| \Phi(a) \cdot \om_n^Q - \rho(a) \cdot \om_n^Q\| = 0 = \lim_n \| \om_n^Q \cdot \Phi(a) - \om_n^Q \cdot \rho(a)\|
\end{equation}
for all $a \in pQp$.

To prove the third property in step~3, fix a partial isometry $v \in \sN_{pMp}(pAp)$ and write $s = v^* v$. Take $u \in \cN_{eMe}(A)$ such that $v = us$. Define $Q = (A \cup \{u\})\dpr$. By step~1 of the proof, $pQp$ is $\Phi$-amenable. By construction $v \in pQp$. By \eqref{eq.omnQ},
$$\lim_n \| \Phi(v^*) \cdot \om_n^Q \cdot \Phi(v) - \rho(v^*) \cdot \om_n^Q \cdot \rho(v) \| = 0 \; .$$
Restricting these positive functionals to $q_1 \cN q_1$, we find the third property in step~3.

{\bf Notations.} Choose a standard Hilbert space $\cH$ for the von Neumann algebra $\cN$, which comes with the normal $*$-homomorphism $\pi_l : \cN \recht B(\cH)$, the normal $*$-antihomomorphism $\pi_r : \cN \recht B(\cH)$ and the positive cone $\cH^+ \subset \cH$. For every $u \in \cN_{eMe}(A)$, define the automorphism $\beta_u$ of $\cN$ implemented by right multiplication with $u^*$ on $L^2(Me) \ot L^2(G)$ and denote by $W_u \in \cU(\cH)$ its canonical implementation.

Denote by $E_\cZ : pAp \recht \cZ(A)p$ the unique trace preserving conditional expectation (i.e.\ the center valued trace of $pAp$). For every projection $s \in pAp$, denote by $z_s \in \cZ(A)p$ its central support, which equals the support projection of $E_\cZ(s)$. Denote by $\cP_0 \subset \cP(pAp)$ the set of projections $s \in pAp$ for which there exists a $\delta > 0$ such that $E_\cZ(s) \geq \delta z_s$. We then denote $D_s = (E_\cZ(s))^{1/2}$ and we denote by $D_s^{-1}$ the (bounded) inverse of $D_s$ in $\cZ(A) z_s$. As in \cite[Section 3]{BHV15} and using \cite[Lemma 3.9]{BHV15}, we can choose a sequence $a_i \in pAp$ such that
$$\sum_{i=0}^\infty a_i a_i^* = D_s z_s \quad\text{and}\quad \sum_{i=0}^\infty a_i^* a_i = D_s^{-1} s \; .$$
We make once and for all a choice of $a_i$ for each $s \in \cP_0$. We also define
$$T(s) = \sum_{i=0}^\infty \Phi(a_i) \rho(a_i^*) \in q_1 \cN q_1 \; .$$
Note that the series defining $T(s)$ is strongly convergent, so that $T(s)$ is a well defined element of $q_1 \cN q_1$.

For every partial isometry $v \in \sN_{pMp}(pAp)$ with $s = v^* v$ and $t = v v^*$, we denote by $W_v : \pi_l(\rho(s)) \pi_r(\rho(s)) \cH \recht \pi_l(\rho(t)) \pi_r(\rho(t)) \cH$ the canonical unitary implementation of the $*$-iso\-mor\-phism $\beta_v : \rho(s) \cN \rho(s) \recht \rho(t) \cN \rho(t)$.

%

{\bf Step 4.} The canonical implementation $\xi_n \in \pi_l(q_1) \pi_r(q_1) \cH$ of $\om_n$ satisfies the following properties.
\begin{enumlist}
\item $\lim_n \langle \pi_l(\Phi(x)) \xi_n, \xi_n \rangle = \Tr(pxp) = \lim_n \langle \pi_r(\Phi(x)) \xi_n , \xi_n \rangle$ for all $x \in M$,
\item $\lim_n \|\pi_l(\Phi(a)) \xi_n - \pi_r(\rho(a))\xi_n \| = 0$ for all $a \in \cU(pAp)$,
\item Whenever $v \in \sN_{pMp}(pAp)$ is a partial isometry such that $s = v^* v$ and $t = v v^*$ belong to $\cP_0$, we have
\begin{equation}\label{eq.i-like}
\lim_n \| \pi_l(\Phi(v)) \xi_n - \pi_r(T(s)^*) \, \pi_r(\Phi(v)) \, W_v^* \, \pi_r(T(t)) \, \xi_n \| = 0 \; .
\end{equation}
\end{enumlist}

The first two properties follow immediately from the first two properties of $\om_n$ in step~3. To also deduce the third property from step~3, one can literally apply the proof of \cite[Proposition 3.6]{BHV15}.

{\bf Notations and formulation of the dichotomy.} As in the proof of Theorem \ref{thm.main-tech}, the coaction $\Psi : \cN \recht \cN \ovt L(G)$ given by $\Psi = \id \ot \Delta$ has a canonical implementation on the standard Hilbert space $\cH$ given by a nondegenerate $*$-homomorphism $\pi : C_0(G) \recht B(\cH)$. We again distinguish two cases.
\begin{itemlist}
\item {\bf Case 1.} For every $F \in C_0(G)$, we have that $\limsup_n \|\pi(F) \xi_n\| = 0$.
\item {\bf Case 2.} There exists an $F \in C_0(G)$ with $\limsup_n \|\pi(F) \xi_n\| > 0$.
\end{itemlist}

We prove that in case~1, the von Neumann subalgebra $\sN_{pMp}(pAp)\dpr \subset pMp$ is $\Phi$-amenable and that in case~2, the von Neumann subalgebra $pAp \subset pMp$ can be $\Phi$-embedded.

The proof in case~2 is identical to the proof of case~2 in Theorem \ref{thm.main-tech}, because that part of the proof only relies on the first two properties of the net $\xi_n$ in step~4. So from now on, assume that we are in case~1. Choose a positive functional $\Om$ on $B(\cH)$ as a weak$^*$ limit point of the net of vector functionals $T \mapsto \langle T \xi_n,\xi_n \rangle$.

Denote $\cG = \cN_{eMe}(A)$. The group $\cG$ acts on $\cN$ by the automorphisms $\beta_u$, $u \in \cG$. We also consider the diagonal action of $\cG$ on $\cN \otalg \cN\op$ and denote by $D$ the algebraic crossed product $D = (\cN \otalg \cN\op) \rtimesalg \cG$. As a vector space, $D = \cN \otalg \cN\op \otalg \C \cG$ and the product and $*$-operation on $D$ are given by
\begin{align*}
& (x_1 \ot y_1\op \ot u_1) \, (x_2 \ot y_2\op \ot u_2) = x_1 \beta_{u_1}(x_2) \ot (\beta_{u_1}(y_2) y_1)\op \ot u_1 u_2 \quad\text{and}\\
& (x \ot y\op \ot u)^* = \beta_{u^*}(x^*) \ot (\beta_{u^*}(y^*))\op \ot u^* \; .
\end{align*}
We define the $*$-representations
\begin{align*}
& \Theta : D \recht B(\cH) : \Theta(x \ot y\op \ot u) = \pi_l(x) \, \pi_r(y) \, W_u \;\; ,\\
& \Theta_1 : D \recht B(\cH) \ovt L(G) : \Theta_1(x \ot y\op \ot u) = (\pi_l \ot \id)\Psi(x) \, (\pi_r(y) \, W_u \ot 1) \;\; .
\end{align*}

Define the $*$-subalgebras $\cN_i$ of $\cN$ given by
\begin{align*}
\cN_1 &= \lspan \Phi(M) \rho(A) \; ,\\
\cN_2 &= [\cN_1] \; ,\\
\cN_3 &= \Bigl\{ x \in \cN \Bigm| \, \begin{aligned}[t] &\text{there exist sequences $x_i \in M$ and $a_i \in A$ such that} \\ & \text{all $\sum_i x_i^* x_i$, $\sum_i x_i x_i^*$, $\sum_i a_i^* a_i$ and $\sum_i a_i a_i^*$ are bounded} \\ & \text{and}\;\; x = \sum_i \Phi(x_i) \rho(a_i) \, \Bigr\} \; .\end{aligned}
\end{align*}
Each $\cN_i$ is globally invariant under the automorphisms $\beta_u$, $u \in \cG$, and so we have the $*$-subalgebras $D_i \subset D$ defined as $D_i = (\cN_i \otalg \cN_i\op) \rtimesalg \cG$. Note that $\cN_1 \subset \cN_3$, but that the inclusion $\cN_2 \subset \cN_3$ need not hold.

Denote $C = \Tr(p) = \|\Om\|$. We claim that

\begin{equation}\label{eq.major-claim}
|\Om(\Theta(x))| \leq C \, \|\Theta_1(x)\| \quad\text{for all}\;\; x \in D_3 \; .
\end{equation}

To prove \eqref{eq.major-claim}, first note that in exactly the same way as we proved \eqref{eq.even-better-est}, we get that \eqref{eq.major-claim} holds for all $x \in D_1$ and thus also for all $x \in D_2$ by norm continuity.

Whenever $x_i \in M$ and $a_i \in A$ are sequences as in the definition of $\cN_3$ and $x = \sum_i \Phi(x_i) \rho(a_i)$, we can choose a sequence of projections $p_n \in pMp$ such that $p_n \recht p$ strongly and such that for each fixed $n$, the series $p_n \sum_i x_i x_i^* p_n$ is norm convergent. This means that for each $n$, we have that $\Phi(p_n) x \in \cN_2$.

Fix $x \in D_3$. Since the automorphisms $\beta_u$ act as the identity on $\Phi(M)$, it follows that we can find a sequence of projections $p_n \in pMp$ such that $p_n \recht p$ strongly and such that
$$x_n = (\Phi(p_n) \ot 1 \ot 1) \, x \, (1 \ot \Phi(p_n)\op \ot 1) \in D_2$$
for all $n$.

Since $\Om(\pi_l(\Phi(x))) = \Tr(pxp)$ for all $x \in M$, we get that
$$\Om = \Om \cdot \pi_l(p) \quad\text{and}\quad \lim_n \|\Om - \Om \cdot \pi_l(\Phi(p_n)\| = 0 \; .$$
A similar result holds for $\pi_r$ and thus,
$$\lim_n \|\Om - \pi_r(\Phi(p_n)) \cdot \Om \cdot \pi_l(\Phi(p_n))\| = 0 \; .$$
This implies that $\Om(\Theta(x)) = \lim_n \Om(\Theta(x_n))$. Since $x_n \in D_2$, we know that
$$|\Om(\Theta(x_n))| \leq C \, \|\Theta_1(x_n)\| \leq C \, \|\Theta_1(x)\|$$
for all $n$. So, \eqref{eq.major-claim} follows.

By \eqref{eq.major-claim}, we can define the continuous functional $\Om_1$ on the C$^*$-algebra $[\Theta_1(D_3)]$ satisfying $\Om_1(\Theta_1(x)) = \Om(\Theta(x))$ for all $x \in D_3$. It follows that $\Om_1(\Theta_1(x)^* \Theta_1(x)) \geq 0$ for all $x \in D_3$ and thus, $\Om_1$ is positive. Extend $\Om_1$ to a bounded functional on $B(\cH \ot L^2(G))$ without increasing its norm. In particular, $\Om_1$ remains a positive functional.

Let $v \in \sN_{pMp}(pAp)$ be a partial isometry such that $s = v^* v$ and $t = vv^*$ belong to $\cP_0$. Using the same notation as in step~4, we define the operator $Y(v) \in B(\cH)$ given by
$$Y(v) = \pi_r(T(s)^*) \, \pi_r(\Phi(v)) \, W_v^* \, \pi_r(T(t)) \; .$$
Note that $Y(v)$ commutes with $\pi_l(\Phi(M))$. Also note that $Y(v)^* = Y(v^*)$. Take $u \in \cN_{eMe}(A)$ such that $v = u s$. Since $W_v = W_u \, \pi_l(\rho(s)) \, \pi_r(\rho(s))$, we define the element $y(v) \in D_3$ given by
$$y(v) = (\rho(s) \ot (\rho(s) \Phi(v) T(s)^*)\op \ot 1) \, (1 \ot 1 \ot u^*) \, (1 \ot T(t)\op \ot 1)$$
and note that $\Theta(y(v)) = Y(v)$ and $\Theta_1(y(v)) = Y(v) \ot 1$.

For every $T \in B(\cH)$, write $\|T\|_{\Om} = \sqrt{\Om(T^*T)}$. Similarly define $\|T\|_{\Om_1} = \sqrt{\Om_1(T^* T)}$ for all $T \in B(\cH \ot L^2(G))$. Applying \eqref{eq.i-like} for $v$ and $v^*$, and using that $Y(v^*) = Y(v)^*$, we find that
\begin{align*}
& \| \Theta( \Phi(v) \ot 1 \ot 1 - y(v))\|_{\Om} = \|\pi_l(\Phi(v)) - Y(v)\|_{\Om} = 0 \;\;\text{and}\\
& \| \Theta( \Phi(v^*) \ot 1 \ot 1 - y(v)^*)\|_{\Om} = \|\pi_l(\Phi(v^*)) - Y(v)^*\|_{\Om} = 0 \;\; .
\end{align*}
Then also
\begin{equation}\label{eq.joepie}
\begin{split}
& \| (\pi_l \circ \Phi \ot \id)(\Phi(v)) - Y(v) \ot 1\|_{\Om_1} \begin{aligned}[t] & = \|\Theta_1(\Phi(v) \ot 1 \ot 1 - y(v))\|_{\Om_1} \\ &= \|\Theta(\Phi(v) \ot 1 \ot 1 - y(v))\|_{\Om} = 0 \;\; , \end{aligned}\\
& \| (\pi_l \circ \Phi \ot \id)(\Phi(v^*)) - Y(v)^* \ot 1\|_{\Om_1} \begin{aligned}[t] & = \|\Theta_1(\Phi(v^*) \ot 1 \ot 1 - y(v)^*)\|_{\Om_1} \\ &= \|\Theta(\Phi(v^*) \ot 1 \ot 1 - y(v)^*)\|_{\Om} = 0 \;\; .\end{aligned}
\end{split}
\end{equation}
%
%
Define the positive functional $\Om_2$ on $q (M \ovt B(L^2(G)))q$ given by $\Om_2 = \Om_1 \circ (\pi_l \circ \Phi \ot \id)$. Since $Y(v) \ot 1$ commutes with $\pi_l(\Phi(M)) \ovt B(L^2(G))$, it follows from \eqref{eq.joepie} that $\Om_2(\Phi(v) T) = \Om_2(T \Phi(v))$ for every partial isometry $v \in \sN_{pMp}(pAp)$ with $v^* v$ and $v v^*$ belonging to $\cP_0$. We also have that $\Om_2(\Phi(x)) = \Tr(x)$ for all $x \in pMp$. Since the linear span of all such partial isometries $v$ is $\|\,\cdot\,\|_2$-dense in $P = \sN_{pMp}(pAp)\dpr$, it follows that $\Om_2$ is $\Phi(P)$-central. So we have proved that $P$ is $\Phi$-amenable. This concludes the proof of Theorem \ref{thm.main-tech-stable-normalizer}.

\section{Stable strong solidity; proof of Theorem \ref{thm.strongly-solid}}

The following is an immediate consequence of Ozawa's solidity theorem \cite{Oz03}.

\begin{proposition}\label{prop.solid-LG}
Let $G$ be a locally compact group with property~(S). Assume that $C^*_r(G)$ is an exact C$^*$-algebra. Then $M = L(G)$ is solid in the sense that for every diffuse von Neumann subalgebra $A \subset M$ that is the range of a normal conditional expectation, $A' \cap M$ is injective.
\end{proposition}

\begin{proof}
By the type III version of Ozawa's theorem \cite{Oz03}, as proved in \cite[Theorem 2.5]{VV05}, it suffices to prove the Akemann-Ostrand property, meaning that the $*$-homomorphism
$$\theta : C^*_r(G) \otalg C^*_r(G) \recht \frac{B(L^2(G))}{K(L^2(G))} : \theta(a \ot b) = \lambda(a) \rho(b) + K(L^2(G))$$
is continuous on the spatial tensor product $C^*_r(G) \otmin C^*_r(G)$.

As in the proof of \eqref{eq.crucial-limit}, property~(S) gives rise to an isometry $Z_0$ that is an adjointable operator from $C_0(G)$ to $C_0(G) \otmin L^2(G)$ with the property that
$$Z_0 \lambda(a) - (\lambda \ot \lambda)\Delta(a) Z_0 \quad\text{and}\quad Z_0 \rho(b) - (\rho(b) \ot 1) Z_0 \quad\text{belong to $K(L^2(G)) \otmin L^2(G)$}$$
for all $a,b \in C^*_r(G)$. The standard representation of $C^*_r(G) \otmin C^*_r(G)$ on $L^2(G) \ot L^2(G)$ is unitarily conjugate to the representation
$$\theta_1 : C^*_r(G) \otmin C^*_r(G) \recht B(L^2(G) \ot L^2(G)) : \theta_1(a \ot b) = (\lambda \ot \lambda)\Delta(a) \, (\rho(b) \ot 1) \; .$$
Since $\theta(a \ot b) = Z_0^* \theta_1(a \ot b) Z_0 + K(L^2(G))$ for all $a,b \in C^*_r(G)$, the Akemann-Ostrand property indeed holds.
\end{proof}

For the proof of Theorem \ref{thm.strongly-solid}, we need the following lemma.

\begin{lemma}\label{lem.reduction-corners}
Let $M$ be a diffuse $\sigma$-finite von Neumann algebra and $p_n \in M$ a sequence of projections such that $p_n \recht 1$ strongly. Then $M$ is stably strongly solid if and only if $p_n M p_n$ is stably strongly solid for every $n$.
\end{lemma}
\begin{proof}
Write $\cH = \ell^2(\N)$ and denote by $z_n \in \cZ(M)$ the central support of $p_n$. Since $M$ is $\sigma$-finite, we have $B(\cH) \ovt p_n M p_n \cong B(\cH) \ovt M z_n$. By \cite[Corollary 5.2]{BHV15}, we get that $p_n M p_n$ is stably strongly solid if and only if $M z_n$ is stably strongly solid. Since every diffuse von Neumann algebra admits a diffuse amenable (even abelian) von Neumann subalgebra with expectation, it is easy to check that $M$ is stably strongly solid if and only if $Mz_n$ is stably strongly solid for each $n$.
\end{proof}

\begin{proof}[{Proof of Theorem \ref{thm.strongly-solid}}]
First assume that $G$ is unimodular. Fix a Haar measure on $G$ and denote by $\Tr$ the associated faithful normal semifinite trace on $M = L(G)$. Fix a projection $p \in L(G)$ with $\Tr(p) < \infty$. Assume that $G$ is weakly amenable and has property~(S). We have to prove that $p M p$ is strongly solid. So fix a diffuse amenable von Neumann subalgebra $A \subset p M p$. We have to prove that $\cN_{pMp}(A)\dpr$ is amenable.

Denote by $\Delta : L(G) \recht L(G) \ovt L(G) : \Delta(\lambda_g) = \lambda_g \ot \lambda_g$ the comultiplication. View $\Delta$ as a coaction on $M$, so that we can apply Theorem \ref{thm.main-tech}. Since $A$ is amenable, we certainly have that $A$ is $\Delta$-amenable. We next prove that $A$ cannot be $\Delta$-embedded.

Fix a net $a_n \in \cU(A)$ such that $a_n \recht 0$ weakly. For every $\xi,\eta \in L^2(G)$, denote by $\om_{\xi,\eta} \in L(G)_*$ the vector functional given by $\om_{\xi,\eta}(\lambda_g) = \langle \lambda_g \xi,\eta\rangle$. Also denote by $m_{\xi,\eta} : L(G) \recht L(G) : m_{\xi,\eta} = (\id \ot \om_{\xi,\eta}) \circ \Delta$ the associated normal completely bounded map. We claim that
\begin{equation}\label{eq.claim-strong-zero}
m_{\xi,\eta}(a_n) \recht 0 \quad\text{strongly, for all $\xi,\eta \in L^2(G)$.}
\end{equation}
Fix $\xi,\eta \in L^2(G)$ and fix $\mu \in L^2(G)$. To prove \eqref{eq.claim-strong-zero}, we must prove that $\|m_{\xi,\eta}(a_n) \mu\| \recht 0$. Denote by $V \in L^\infty(G) \ovt L(G)$ the unitary given by $V(g) = \lambda_g$ for all $g \in G$. For every $a \in L(G)$, we have
\begin{align*}
m_{\xi,\eta}(a) \mu & = (\id \ot \om_{\xi,\eta})(\Delta(a)) \mu = (\om_{\xi,\eta} \ot \id)(\Delta(a)) \mu \\
& = (\om_{\xi,\eta} \ot \id)(V (a \ot 1) V^*) \mu = (\eta^* \ot 1) V (a \ot 1) V^* (\xi \ot \mu) \; .
\end{align*}
Approximating $V^*(\xi \ot \mu) \in L^2(G) \ot L^2(G)$ by linear combinations of vectors $\xi_0 \ot \mu_0$, it suffices to prove that
$$\lim_n \|(\eta^* \ot 1) V (a_n \ot 1) (\xi_0 \ot \mu_0) \| = 0 \quad\text{for all $\eta,\xi_0,\mu_0 \in L^2(G)$.}$$
Since the operator $(\eta^* \ot 1) V (1 \ot \mu_0)$ belongs to $K(L^2(G))$, this last statement indeed holds and the claim in \eqref{eq.claim-strong-zero} is proved.

For all $\mu_1,\mu_2 \in L^2(Mp)$ and for all $\xi,\eta \in L^2(G)$, we have
$$\langle \Delta(a_n) \, (\mu_1 \ot \xi) \, a_n^* , \mu_2 \ot \eta \rangle = \langle m_{\xi,\eta}(a_n) \mu_1 , \mu_2 a_n \rangle \; .$$
So, \eqref{eq.claim-strong-zero} implies that
$$\lim_n \langle \Delta(a_n) \cdot \xi \cdot a_n^* , \xi \rangle = 0 \quad\text{for all $\xi \in L^2(Mp) \ot L^2(G)$.}$$
So there is no nonzero vector $\xi \in L^2(Mp) \ot L^2(G)$ satisfying $\Delta(a) \xi = \xi a$ for all $a \in A$, meaning that $A$ cannot be $\Delta$-embedded.

Write $P = \cN_{pMp}(A)\dpr$. By Theorem \ref{thm.main-tech}, $P$ is $\Delta$-amenable. Since the $P$-$M$-bimodule
$$\bim{\Delta(P)}{\Delta(p) (L^2(M) \ot L^2(G))}{M}$$
is contained in a multiple of the coarse $P$-$M$-bimodule, it follows that $P$ is amenable. So we have proved that $pMp$ is strongly solid.

Next assume that $G$ is a locally compact second countable group with CMAP and property~(S) such that the kernel $G_0$ of the modular function $\delta : G \recht \R^+$ is an open subgroup of $G$. Fix a left Haar measure on $G$ and denote by $\vphi$ the associated faithful normal semifinite weight on $M = L(G)$. Denote by $\sigma^\vphi$ its modular automorphism group, given by
$$\sigma_t^\vphi(\lambda_g) = \delta(g)^{it} \, \lambda_g \quad\text{for all $g \in G$, $t \in \R$.}$$
So, $L(G_0)$ lies in the centralizer $L(G)^\vphi$ and since $G_0 \subset G$ is an open subgroup, the restriction of $\vphi$ to $L(G_0)$ is semifinite. By Lemma \ref{lem.reduction-corners}, it is sufficient to prove that $p L(G) p$ is stably strongly solid for each nonzero projection $p \in L(G_0)$ with $\vphi(p) < \infty$. Fix such a projection $p$ and let $A \subset p M p$ be a diffuse amenable von Neumann subalgebra with expectation. Write $P = \sN_{pMp}(A)\dpr$. We have to prove that $P$ is amenable.

Denote $\cH = \ell^2(\N)$ and define $M_1 = B(\cH) \ovt M$. Write $A_0 = B(\cH) \ovt A$ and $p_1 = 1 \ot p$. By \cite[Lemma 3.4]{BHV15}, we have to prove that $\cN_{p_1M_1 p_1}(A_0)\dpr$ is amenable. Since $G$ has CMAP, certainly $G$ is exact (see e.g.\ \cite[Corollary E]{BCL16}) and Proposition \ref{prop.solid-LG} implies that $A' \cap p M p$ is amenable. So, $A_1 := A_0 \vee (A_0' \cap p_1 M_1 p_1)$ is amenable. Since $\cN_{p_1M_1 p_1}(A_0)\dpr \subset \cN_{p_1M_1 p_1}(A_1)\dpr$ and since this is an inclusion with expectation, it suffices to show that $P_1 := \cN_{p_1M_1 p_1}(A_1)\dpr$ is amenable.

Let $e \in B(\cH)$ be a minimal projection and choose a faithful normal state $\eta$ on $B(\cH)$ such that $e$ belongs to the centralizer of $\eta$. Also choose a faithful normal state $\psi$ on $p M p$ such that $\sigma_t^\psi(A) = A$ for all $t \in \R$. Note that $\eta \ot \psi$ is a faithful normal state on $p_1 M_1 p_1$. Then $A_1$ and $P_1$ are globally invariant under $\sigma^{\eta \ot \psi}$ and we obtain the canonical inclusions of continuous cores
$$\core_{\eta \ot \psi}(A_1) \subset \core_{\eta \ot \psi}(P_1) \subset \core_{\eta \ot \psi}(p_1 M_1 p_1) \; .$$
Since $A_1' \cap p_1 M_1 p_1 = \cZ(A_1)$, it follows from \cite[Lemma 4.1]{BHV15} that $\core_{\eta \ot \psi}(P_1)$ is contained in the normalizer of $\core_{\eta \ot \psi}(A_1)$. By Takesaki's duality theorem \cite[Theorem X.2.3]{Ta03}, $P_1$ is amenable if and only if its continuous core is amenable. So, it suffices to prove that the normalizer of $\core_{\eta \ot \psi}(A_1)$ is amenable. Now we can cut down again with the projection $e \ot 1$ and conclude that it is sufficient to prove the following result: for any diffuse amenable $B \subset pMp$ with expectation and for every faithful normal state $\psi$ on $pMp$ with $\sigma_t^\psi(B) = B$ for all $t \in \R$, the canonical subalgebra $\core_\psi(B)$ of $\core_\psi(pMp)$ has an amenable stable normalizer.

Whenever $p' \in M^\vphi$ is a projection with $p' \geq p$ and $\vphi(p') < \infty$, we can realize the continuous core of $p'Mp'$ as $\pi_\vphi(p') \cM \pi_\vphi(p')$ where $\cM =  \core_\vphi(M)$. Let
$$\Pi : \core_\psi(pMp) \recht \pi_\vphi(p) \cM \pi_\vphi(p)$$
be the canonical trace preserving isomorphism. Let $p_n \in L_\psi(\R)$ be a sequence of projections having finite trace and converging to $1$ strongly. Since $B$ is diffuse and using Popa's intertwining-by-bimodules \cite[Section 2]{Po03}, it follows from \cite[Lemma 2.5]{HU15} that
\begin{equation}\label{eq.star-star}
\Pi(p_n \core_\psi(B) p_n) \; \underset{\pi_\vphi(p') \cM \pi_\vphi(p')}{\not\prec} \; L_\vphi(\R) \pi_\vphi(p')
\end{equation}
for all $n$ and all projections $p' \in M^\vphi$ with $p' \geq p$ and $\vphi(p') < \infty$. Denote by $\cP$ the set of these projections $p'$ and define the $*$-algebra
$$\cM_0 := \bigcup_{p' \in \cP} \pi_\vphi(p') \cM \pi_\vphi(p') \; .$$
There is a unique linear map $E : \cM_0 \recht L_\vphi(\R)$ such that for every $p' \in \cP$, the restriction of $E$ to $\pi_\vphi(p') \cM \pi_\vphi(p')$ is normal and given by $E(\pi_\vphi(x) \lambda_\vphi(t)) = \vphi(x) \lambda_\vphi(t)$ for all $x \in p' M p'$ and $t \in \R$. Note that this restriction of $E$ can be viewed as $\vphi(p')$ times the unique trace preserving conditional expectation of $\pi_\vphi(p') \cM \pi_\vphi(p')$ onto $L_\vphi(\R) p'$.

Combining \eqref{eq.star-star} and \cite[Theorem 4.3]{HI15}, in order to prove that $\core_\psi(B)$ has an amenable stable normalizer inside $\core_\psi(pMp)$, it is sufficient to prove the following statement: whenever $q \in \pi_\vphi(p) \cM \pi_\vphi(p)$ is a projection of finite trace and $A \subset q \cM q$ is a von Neumann subalgebra that admits a net of unitaries $a_n \in \cU(A)$ satisfying
\begin{equation}\label{eq.such-an-we-have}
E(x^* a_n y) \recht 0 \quad\text{strongly, for all $x,y \in \cM_0$,}
\end{equation}
then the stable normalizer of $A$ inside $q \cM q$ is amenable. Fix such a von Neumann subalgebra $A \subset q \cM q$ and fix a net of unitaries $a_n \in \cU(A)$ satisfying \eqref{eq.such-an-we-have}.

Since $\Delta \circ \sigma_t^\vphi = (\sigma_t^\vphi \ot \id)\circ \Delta$ for all $t \in \R$, there is a well defined coaction given by
$$\Phi : \cM \recht \cM \ovt L(G) : \Phi(\pi_\vphi(x) \lambda_\vphi(t)) = (\pi_\vphi \ot \id)(\Delta(x)) \, (\lambda_\vphi(t) \ot 1)$$
for all $x \in M$, $t \in \R$.

The $\cM$-bimodule $\bim{\Phi(\cM)}{(L^2(\cM) \ot L^2(G))}{\cM \ot 1}$ is isomorphic with $L^2(\cM) \ot_{L_\vphi(\R)} L^2(\cM)$ and thus weakly contained in the coarse $\cM$-bimodule. Using Theorem \ref{thm.main-tech-stable-normalizer}, it only remains to prove that \eqref{eq.such-an-we-have} implies that $A$ cannot be $\Phi$-embedded.

We deduce that $A$ cannot be $\Phi$-embedded from the following approximation result: for all $a \in C^*_r(G_0)$, $\om \in L(G)_*^+$ and $\eps > 0$, there exist $n \in \N$, elements $a_j,x_j \in L(G)$ and scalars $\delta_j > 0$ for $j \in \{1,\ldots,n\}$ such that
\begin{equation}\label{eq.aj-eigenvector}
\sigma_t^\vphi(a_j) = \delta_j^{it} a_j \quad\text{and}\quad \sigma_t^\vphi(x_j) = \delta_j^{-it} x_j
\end{equation}
for all $j \in \{1,\ldots,n\}$ and all $t \in \R$, and such that the map
\begin{equation}\label{eq.map-Psi}
\Psi : \cM \recht \cM : \Psi(x) = \sum_{j,k=1}^n \pi_\vphi(x_j^*) \, E\bigl( \pi_\vphi(pa_j)^* x \pi_\vphi(pa_k)\bigr) \, \pi_\vphi(x_k)
\end{equation}
is normal and completely bounded, and satisfies
\begin{equation}\label{eq.approx}
\bigl\| (\id \ot \om) \Phi\bigl(\pi_\vphi(pa)^* \, \cdot \, \pi_\vphi(pa)\bigr) - \Psi \bigr\|\cb < \eps \; .
\end{equation}
Already note that \eqref{eq.aj-eigenvector} implies that the right support of $p a_j$ belongs to $\cP$, so that $\pi_\vphi(p a_j) \in \cM_0$ and the map $\Psi$ is well defined and normal.

Assuming that such an approximation exists, we already deduce that $A$ cannot be $\Phi$-embedded. It suffices to prove that
$$\lim_n \langle \Phi(a_n) (\mu \ot \xi), \mu a_n \ot \xi \rangle = 0 \quad\text{for all $\mu \in L^2(\cM)$ and $\xi \in L^2(G)$,}$$
because then also $\lim_n \langle \Phi(a_n) \eta (a_n^* \ot 1), \eta \rangle=0$ for all $\eta \in L^2(\cM q) \ot L^2(G)$, excluding the existence of a nonzero vector $\eta \in L^2(\cM q) \ot L^2(G)$ satisfying $\Phi(a) \eta = \eta (a \ot 1)$ for all $a \in A$.

Since $\mu \ot \xi$ can be approximated by vectors of the form $\Phi(\pi_\vphi(a))(\mu \ot \xi)$ with $a \in C^*_r(G_0)$, it suffices to prove that
$$\lim_n \langle \Phi(\pi_\vphi(a)^* a_n \pi_\vphi(a)) (\mu \ot \xi), \mu a_n \ot \xi \rangle = 0 \quad\text{for all $a \in C^*_r(G_0)$, $\mu \in L^2(\cM)$ and $\xi \in L^2(G)$.}$$
Denoting by $\om \in L(G)_*^+$ the vector functional implemented by $\xi$, it is sufficient to prove that
$$(\id \ot \om)(\Phi(\pi_\vphi(a)^* a_n \pi_\vphi(a))) \recht 0 \quad\text{strongly, for all $a \in C^*_r(G_0), \om \in L(G)_*^+$.}$$
But this follows by the approximation in \eqref{eq.approx} and because \eqref{eq.such-an-we-have} implies that $\Psi(a_n) \recht 0$ strongly.

Fixing $a \in C^*_r(G_0)$, $\om \in L(G)_*^+$ and $\eps > 0$, it remains to find the approximation \eqref{eq.approx}.

First take $\xi \in C_c(G)$ such that the vector functional $\om_\xi$ satisfies $\|\om - \om_\xi\| < (1/3) \, \eps \, \|a\|^{-2}$. It follows that, as maps on $M = L(G)$,
\begin{equation}\label{eq.first-approx}
\bigl\| (\id \ot \om) \Delta \bigl(a^* p \, \cdot \, p a\bigr) - (\id \ot \om_\xi) \Delta\bigl(a^* p \, \cdot \, p a\bigr)\bigr) \bigr\|\cb < \frac{\eps}{3} \; .
\end{equation}
Fix $F \in C_c(G)$ with $0 \leq F \leq 1$ and $\xi = F \xi$. For every $x \in M$ and using the unitary $V \in L^\infty(G) \ovt L(G)$ as in the first part of the proof, we have
\begin{align*}
(\id \ot \om_\xi) \Delta\bigl(a^* p \, x \, p a\bigr)\bigr) &= (\om_\xi \ot \id) \Delta\bigl(a^* p \, x \, p a\bigr)\bigr) \\
&= (\xi^* \ot 1) V (a^* p xp a \ot 1) V^* (\xi \ot 1) \\
&= (\xi^* \ot 1) V (F a^* \, pxp \, a F \ot 1)V^* (\xi \ot 1) \; .
\end{align*}
Since $a \in C^*_r(G_0) \subset C^*_r(G)$ and $F \in C_c(G)$, we get that $a F$ is a compact operator on $L^2(G)$ that commutes with the modular function (viewed as a multiplication operator).

So we can approximate $a F$ by a finite rank operator $T$ of the form
$$T = \sum_{j=1}^n \mu_j \xi_j^*$$
where $\xi_j , \mu_j \in C_c(G)$ and $\delta \xi_j = \delta_j \xi_j$, $\delta \mu_j = \delta_j \mu_j$, and such that $\|T\| \leq \|a F \| \leq \|a\|$ and
$$\|a F - T\| \leq \frac{\eps}{3 \, \|a\| \, \|\xi\|^2} \; .$$

Defining
\begin{align*}
& m : M \recht M : m(x) = (\id \ot \om)(\Delta(a^* \, pxp \, a)) \;\; ,\\
& m_1 : M \recht M : m_1(x) = (\xi^* \ot 1) V (T^* \, pxp \, T \ot 1)V^* (\xi \ot 1) \; ,
\end{align*}
we get that $\|m - m_1\|\cb < \eps$.

Defining
$$a_j = \int_G \mu_j(g) \, \lambda_g \, dg \quad\text{and}\quad x_j = (\om_{\xi,\xi_j} \ot \id)(V^*) \; ,$$
we get that
$$m_1(x) = \sum_{j,k=1}^n \vphi(a_j^* \, pxp \, a_k) x_j^* x_k \; .$$
Both $m$ and $m_1$ commute with the modular automorphism group $\sigma^\vphi$ and thus canonically extend to $\cM = \core_\vphi(M)$ by acting as the identity on $L_\vphi(\R)$. The canonical extension of $m$ equals
$$(\id \ot \om) \Phi\bigl(\pi_\vphi(pa)^* \, \cdot \, \pi_\vphi(pa)\bigr) \; ,$$
while the canonical extension of $m_1$ equals the map $\Psi$ given by \eqref{eq.map-Psi}. Since $\|m - m_1\|\cb < \eps$, also \eqref{eq.approx} holds and the theorem is proved.

\end{proof}

\section{Locally compact groups with property (S)}\label{sec.groups-prop-S}

Recall that a compactly generated locally compact group $G$ is said to be \emph{hyperbolic} if the Cayley graph of $G$ with respect to a compact generating set $K \subset G$ satisfying $K = K^{-1}$ is Gromov hyperbolic, in the sense that the metric $d$ on $G$ defined by
$$d(g,h) = \begin{cases} 0 &\quad\text{if $g=h$,}\\
\min \{n \in \N \mid h^{-1} g \in K^n\} &\quad\text{if $g \neq h$,}\end{cases}$$
turns $G$ into a Gromov hyperbolic metric space. When $G$ is non discrete, this Cayley graph is not locally finite and often the action of $G$ on its Cayley graph is not continuous.

By \cite[Corollary 2.6]{CCMT12}, a locally compact group $G$ is hyperbolic if and only if $G$ admits a proper, continuous, cocompact, isometric action on a proper geodesic hyperbolic metric space.

Combining several results from the literature, we have the following list of locally compact groups that are weakly amenable and have property~(S). In the formulation of the proposition, graphs are assumed to be simple, non oriented and connected. We always equip their vertex set with the graph metric. A \emph{hyperbolic graph} is a simple, non oriented, connected graph such that the underlying metric space is Gromov hyperbolic.

\begin{proposition}\label{prop.groups-prop-S}
Let $G$ be a locally compact group. If one of the following conditions holds, then $G$ has the complete metric approximation property and property~(S).
\begin{enumerate}
\item $G$ is $\sigma$-compact and amenable.
\item ({\cite{Ha78,Sz91,Oz03}}) $G$ admits a continuous action on a (not necessarily locally finite) tree that is metrically proper in the sense that for every vertex $x$, we have that $d(x,g \cdot x) \recht \infty$ when $g$ tends to infinity in $G$.
\end{enumerate}
If one of the following conditions holds, then $G$ is weakly amenable and has property~(S).
\begin{enumerate}[resume]
\item $G$ is compactly generated and hyperbolic.
\item ({\cite{Oz03,Oz07}}) $G$ admits a continuous proper action on a hyperbolic graph with uniformly bounded degree.
\item ({\cite{CH88,Sk88}}) $G$ is a real rank one, connected, simple Lie group with finite center.
\end{enumerate}
\end{proposition}

\begin{proof}
1.\ Since $G$ is amenable, a fortiori $G$ has CMAP. Since $G$ is $\sigma$-compact, we can fix an increasing sequence of compact subsets $K_n \subset G$ such that the interiors $\inter(K_n)$ cover $G$. We make this choice such that $K_0 = \emptyset$, $K_n = K_n^{-1}$, $K_n \subset \inter(K_{n+1})$ and $K_n K_n K_n \subset K_{n+1}$ for all $n$. Since $G$ is amenable, we can choose $\eta_n \in \cS(G)$ such that $\|g \cdot \eta_n - \eta_n \|_1 \leq 2^{-n}$ for all $n \geq 0$ and all $g \in K_n$. Choose continuous functions $F_n : G \recht [0,1]$ such that $F_n(g) = 0$ for all $g \in K_{n-1}$ and $F_n(g) = 1$ for all $g \in G \setminus K_n$. By convention, $F_0(g) = 1$ for all $g \in G$.

Define the continuous function
$$\mu : G \recht L^1(G)^+ : \mu(g) = \sum_{n=0}^\infty F_n(g) \eta_n \; .$$
Note that $n \leq \|\mu(g)\|_1 \leq n+1$ whenever $n \geq 1$ and $g \in K_n \setminus K_{n-1}$. Define $\eta : G \recht \cS(G) : \eta(g) = \|\mu(g)\|_1^{-1} \, \mu(g)$.

Choose $\eps > 0$ and $K \subset G$ compact. Take $n_0$ such that $K \subset K_{n_0}$. Then take $n_1>n_0$ such that $2(2n_0 + 6)/n_1 < \eps$. Fix $g,k \in K$ and $h \in G \setminus K_{n_1}$. We prove that $\|\eta(ghk) - g \cdot \eta(h)\|_1 < \eps$. Once this is proved, it follows that $G$ has property~(S). Take $n \geq n_1$ such that $h \in K_{n+1} \setminus K_n$. Since $K^{-1} K_{n-1} K^{-1} \subset K_n$ and $K K_{n+1} K \subset K_{n+2}$, we have $ghk \in K_{n+2} \setminus K_{n-1}$. Define $\gamma \in L^1(G)$ given by
$$\gamma = \sum_{k=0}^{n} \eta_k \; .$$
By construction, $\|\mu(h) - \gamma\|_1 \leq 1$ and $\|\mu(ghk) - \gamma\|_1 \leq 3$. Also,
$$\|g \cdot \gamma - \gamma\|_1 \leq 2 n_0 + \sum_{k=n_0}^n \|g \cdot \eta_k - \eta_k\|_1 \leq 2 n_0 + \sum_{k=n_0}^n 2^{-k} \leq 2(n_0+1) \; .$$
Altogether, it follows that $\|g \cdot \mu(h) - \mu(ghk)\|_1 \leq 2 n_0 + 6$. Since $\|\mu(h)\|_1 \geq n_1$, it follows that $\|g \cdot \eta(h) - \eta(ghk)\|_1 < \eps$.

2.\ Let $\cG = (V,E)$ be a tree and $G \actson \cG$ a continuous metrically proper action. By \cite[Corollary 12.3.4]{BO08}, the group $G$ has CMAP. For all $x,y \in V$, denote by $A(x,y) \subset V$ the (unique) geodesic between $x$ and $y$. Fix a base point $x_0 \in V$. Define the continuous map $\eta : G \recht \Prob(V)$ by defining $\eta(g)$ as the uniform probability measure on $A(x_0,g \cdot x_0)$. For all $g,h,k \in G$, the symmetric difference between $A(x_0, ghk \cdot x_0)$ and $g \cdot A(x_0,h \cdot x_0)$ contains at most $d(x_0,g \cdot x_0) + d(x_0,k \cdot x_0)$ elements. Since the action $G \actson \cG$ is metrically proper, we have $d(x_0,h \cdot x_0) \recht \infty$ when $h$ tends to infinity in $G$. It then follows that
$$\lim_{h \recht \infty} \|\eta(ghk) - g \cdot \eta(h) \|_1 = 0 \quad\text{uniformly on compact sets of $g,k \in G$.}$$
Since the action $G \actson V$ has compact open stabilizers, there exists a $G$-equivariant isometric map $\Prob(V) \recht \cS(G)$. Composing $\eta$ with this map, it follows that $G$ has property~(S).

3.\ By \cite[Corollary 2.6]{CCMT12}, $G$ admits a proper, continuous, cocompact, isometric action on a proper geodesic hyperbolic metric space. By \cite[Theorem 21 and Proposition 8]{MMS03}, $G$ satisfies at least one of the following three structural properties: $G$ is amenable, or $G$ admits a proper action on a hyperbolic graph with uniformly bounded degree, or $G$ admits closed subgroups $K < G_0 < G$ such that $G_0$ is of finite index and open in $G$, $K$ is a compact normal subgroup of $G_0$ and $G_0/K$ is a real rank one, connected, simple Lie group with finite center. Since we already proved 1, to complete the proof of 3, it suffices to prove 4 and 5 and apply Lemma \ref{lem.stability} below.

4.\ Let $\cG = (V,E)$ be a hyperbolic graph with uniformly bounded degree. Let $G \actson \cG$ be a continuous proper action. By \cite[Theorem 1]{Oz07}, the group $G$ is weakly amenable. The proof of property~(S) is almost identical to the proof of \cite[Theorem 1.33]{Ka02} and especially the version in \cite[Theorem 5.3.15]{BO08}. For completeness, we provide the details here.

We use the following ad hoc terminology. Assume that $[x',y'] \subset V$ is a geodesic. If $d(x',y')$ is even, we call ``mid point of $[x',y']$'' the unique point $z \in [x',y']$ with $d(x',z) = d(z,y') = d(x',y')/2$. If $d(x',y')$ is odd, we declare two points of $[x',y']$ to be the ``mid points of $[x',y']$'', namely the two points $z \in [x',y']$ with $d(x',z) = (d(x',y') \pm 1)/2$ and thus $d(z,y') = (d(x',y') \mp 1)/2$. For all $x,y \in V$ and $k \in \N$, define the nonempty subset $A(x,y,k) \subset V$ given by
\begin{align*}
A(x,y,k) = \bigl\{ z \in V \; \bigm| \; &\text{there exists a geodesic $[x',y'] \subset V$ with $d(x,x') \leq k$ and $d(y,y') \leq k$}\\
&\text{such that $z$ is one of the mid points of $[x',y']$}\;\bigr\} \; .
\end{align*}
Note that $A(x,y,k) = A(y,x,k)$ and $A(g \cdot x, g \cdot y, k) = g \cdot A(x,y,k)$ for all $x,y \in V$, $k \in \N$ and $g \in G$.

Take $\delta > 0$ such that every geodesic triangle in $\cG$ is $\delta$-thin (see \cite[Definition 5.3.3]{BO08}). Define
$$B = \sup_{x \in V} |\{y \in V \mid d(y,x) \leq 2\delta\}| \; .$$
Since $\cG$ has uniformly bounded degree, we have that $B < \infty$. We claim that for all $x,y \in V$ with $d(x,y) \geq 4k$, we have
\begin{equation}\label{eq.est-number-el-A}
|A(x,y,k)| \leq 2 (k+1) B \;  .
\end{equation}
To prove this claim, fix a geodesic $[x,y]$ between $x$ and $y$ and denote by $[a,b] \subset [x,y]$ the unique segment determined by
$$d(x,a) = \lfloor d(x,y)/2 \rfloor - k \quad\text{and}\quad d(x,b) = \lceil d(x,y)/2 \rceil + k \; .$$
Note that $[a,b]$ contains at most $2(k+1)$ vertices. To prove the claim, it thus suffices to show that every $z \in A(x,y,k)$ lies at distance at most $2 \delta$ from a vertex on $[a,b]$.

Choose a geodesic $[x',y'] \subset V$ with $d(x,x') \leq k$ and $d(y,y') \leq k$. Let $z$ be one of the mid points of $[x',y']$. Since $d(x,y) \geq 4k$, the geodesic picture of the five points $x,x',y,y',z$ in a \emph{tree} would look as the following picture on the left.
$$
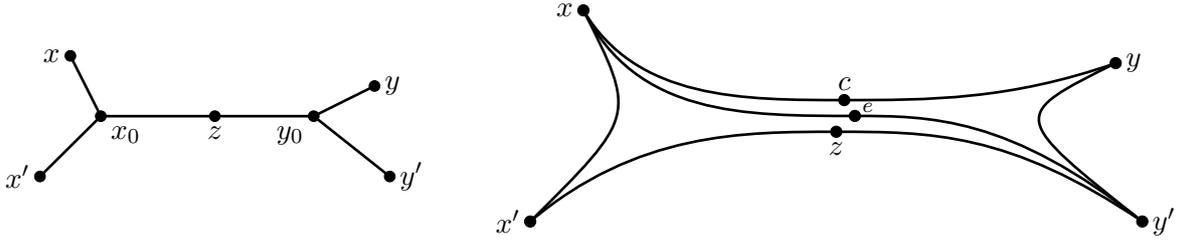
\begin{tikzpicture}[scale=.4,baseline=(current bounding box.west)]
\draw[line width=1pt] (0,0) -- (2,2) -- (9,2) -- (11.5,0);
\draw[line width=1pt] (1,4) -- (2,2);
\draw[line width=1pt] (9,2) -- (11,3);
\draw[fill] (0,0) circle (5pt) node[anchor=east] {$x'$};
\draw[fill] (1,4) circle (5pt) node[anchor=east] {$x$};
\draw[fill] (2,2) circle (5pt) node[anchor=north west] {$x_0$};
\draw[fill] (9,2) circle (5pt) node[anchor=north east] {$y_0$};
\draw[fill] (11,3) circle (5pt) node[anchor=west] {$y$};
\draw[fill] (11.5,0) circle (5pt) node[anchor=west] {$y'$};
\draw[fill] (5.75,2) circle (5pt) node[anchor=north] {$z$};
\end{tikzpicture}\hspace{7mm}
\begin{tikzpicture}[scale=.7,baseline=(current bounding box.west)]
\draw[line width=1pt] (0,0) .. controls (2,1.7) and (4,1.7) .. (5.75,1.7);
\draw[line width=1pt] (11.5,0) .. controls (9,1.7) and (7.5,1.7) .. (5.75,1.7);
\draw[line width=1pt] (0,0) .. controls (2,2) .. (1,4);
\draw[line width=1pt] (1,4) .. controls (2,2.3) and (4,2.3) .. (5.9,2.3);
\draw[line width=1pt] (11,3) .. controls (9,2.3) and (7.5,2.3) .. (5.9,2.3);
\draw[line width=1pt] (11.5,0) .. controls (9,2) .. (11,3);
\draw[line width=1pt] (1,4) .. controls (2,2) and (4,2) .. (6.1,2);
\draw[line width=1pt] (11.5,0) .. controls (9,2) and (7.5,2) .. (6.1,2);
\draw[fill] (0,0) circle (3pt) node[anchor=east] {$x'$};
\draw[fill] (1,4) circle (3pt) node[anchor=east] {$x$};
\draw[fill] (11,3) circle (3pt) node[anchor=west] {$y$};
\draw[fill] (11.5,0) circle (3pt) node[anchor=west] {$y'$};
\draw[fill] (6.1,2) circle (3pt) node[anchor=south west,xshift=-1,yshift=-2.3] {{\scriptsize $e$}};
\draw[fill] (5.75,1.7) circle (3pt) node[anchor=north] {$z$};
\draw[fill] (5.9,2.3) circle (3pt) node[anchor=south] {$c$};
\end{tikzpicture}
$$
In our comparison tree, some of the ``small'' segments $[x,x_0]$, $[x',x_0]$, $[y,y_0]$, $[y',y_0]$ could be reduced to a single point, but the ``large'' segment $[x_0,y_0]$ has length at least $2k$. Therefore, in the comparison tree, the mid point $z$ of $[x',y']$ lies on the segment $[x_0,y_0]$. We now turn back to segments in the hyperbolic graph $\cG$, as in the picture on the right. Denote by $c \in [x,y]$ the unique point with $d(x,c) = d(x',z)$. By construction, $c \in [a,b]$. To conclude the proof of \eqref{eq.est-number-el-A}, we show that $d(c,z) \leq 2\delta$. Choose a geodesic $[x,y']$ and denote by $e \in [x,y']$ the unique point with $d(x,e) = d(x',z)$. Applying $\delta$-thinness to the geodesic triangle $x,x',y'$, we find that $d(z,e) \leq \delta$. Then applying $\delta$-thinness to the geodesic triangle $x,y,y'$, we get that $d(e,c) \leq \delta$. So, $d(z,c) \leq 2\delta$ and the claim in \eqref{eq.est-number-el-A} is proven.


Given a finite subset $A \subset V$, denote by $p(A)$ the uniform probability measure on $A$. Exactly as in the proof of \cite[Theorem 5.3.15]{BO08}, define the sequence of maps
$$\eta_n : V \times V \recht \Prob(V) : \eta_n(x,y) = \frac{1}{n} \sum_{k=n+1}^{2n} p(A(x,y,k)) \; .$$
For finite sets $A,B \subset V$, we have
$$\frac{1}{2} \, \|p(A) - p(B)\|_1 = 1 - \frac{|A \cap B|}{\max\{|A|,|B|\}} \; .$$
When $d(x,x') \leq d \leq k$, we have
$$A(x,y,k-d) \subset A(x',y,k) \subset A(x,y,k+d) \; .$$
Therefore, whenever $d(x,x') \leq d \leq k$, we have
$$\frac{1}{2} \, \|p(A(x,y,k)) - p(A(x',y,k))\|_1 \leq 1 - \frac{|A(x,y,k-d)|}{|A(x,y,k+d)|} \; .$$
So if $d(x,x') \leq d \leq n$, we use the inequality between arithmetic and geometric mean and get that
\begin{align*}
\frac{1}{2} \, \| \eta_n(x,y) - \eta_n(x',y)\|_1 & \leq 1 - \frac{1}{n} \sum_{k=n+1}^{2n} \frac{|A(x,y,k-d)|}{|A(x,y,k+d)|} \\
& \leq 1 - \left(\prod_{k=n+1}^{2n} \frac{|A(x,y,k-d)|}{|A(x,y,k+d)|}\right)^{1/n} \\
& = 1 - \left(\frac{\prod_{k=n-d+1}^{n+d} |A(x,y,k)|}{\prod_{k=2n-d+1}^{2n+d}|A(x,y,k)|}\right)^{1/n} \; .
\end{align*}
Using \eqref{eq.est-number-el-A}, it follows that whenever $d(x,x') \leq d \leq n$ and $d(x,y) \geq 4(2n+d)$, we have
$$\frac{1}{2} \, \| \eta_n(x,y) - \eta_n(x',y)\|_1 \leq 1 - \bigl( 2 (2n+d+1) B \bigr)^{-2d / n} \; .$$
So, for every $\eps > 0$ and every $d \in \N$, there exists an $n$ such that $$\|\eta_n(x,y) - \eta_n(x',y)\|_1 < \eps$$
for all $x,x',y \in V$ with $d(x,x') \leq d$ and $d(x,y) \geq 4(2n+d)$.

The maps $\eta_n$ are $G$-equivariant in the sense that
$$\eta_n(g \cdot x, g \cdot y) = g \cdot \eta_n(x,y) \quad\text{for all $g \in G$, $x,y \in V$,}$$
and the maps $\eta_n$ are symmetric in the sense that $\eta_n(x,y) = \eta_n(y,x)$ for all $x,y \in V$.

Passing to a subsequence, we find a sequence of $G$-equivariant symmetric maps $\eta_n : V \times V \recht \Prob(V)$ and a strictly increasing sequence of integers $d_n \in \N$ such that
$$\|\eta_n(x,y) - \eta_n(x',y)\|_1 \leq 2^{-n}$$
whenever $x,x',y \in V$, $d(x,x') \leq n$ and $d(x,y) \geq d_n$. By convention, we take $d_0 = 0$ and $\eta_0(x,y) = \frac{1}{2}(\delta_x + \delta_y)$. Define the function $\rho : \N \recht \N$ given by
$$\rho(n) = \max \{ k \in \N \mid d_k \leq n\} \; .$$
Then $\rho(0) = 0$, $\rho$ is increasing, $\rho(n) \recht \infty$ when $n \recht \infty$ and $|\rho(n) - \rho(m)| \leq |n-m|$ for all $n,m \in \N$. Using the trick in \cite[Exercise 15.1.1]{BO08}, define the $G$-equivariant symmetric map
$$\mu : V \times V \recht \ell^1(V)^+ : \mu(x,y) = \sum_{n=0}^{\rho(d(x,y))} \eta_n(x,y) \; .$$
Note that $\|\mu(x,y)\|_1 = 1 + \rho(d(x,y))$. For all $x,x',y \in V$ with $\rho(d(x,y)) \geq d(x,x')$, we have
\begin{align}
\|\mu(x,y) - \mu(x',y)\|_1 & \leq 2 d(x,x') + |\rho(d(x,y)) - \rho(d(x',y))| + \sum_{n=d(x,x')}^{\rho(d(x,y))} \|\eta_n(x,y) - \eta_n(x',y)\|_1  \notag\\
& \leq 3 d(x,x') + \sum_{n=d(x,x')}^{\rho(d(x,y))} 2^{-n} \leq 3 d(x,x') + 2 \; .\label{eq.est-mu-mu-prime}
\end{align}
Define the $G$-equivariant symmetric map
$$\eta : V \times V \recht \Prob(V) : \eta(x,y) = \|\mu(x,y)\|_1^{-1} \, \mu(x,y) \; .$$
Since $\|\mu(x,y)\|_1 = 1 + \rho(d(x,y))$, it follows from \eqref{eq.est-mu-mu-prime} that
$$\|\eta(x,y) - \eta(x',y)\|_1 \leq \frac{2(3 d(x,x') + 2)}{1 + \rho(d(x,y))}$$
for all $x,x',y \in V$ with $\rho(d(x,y)) \geq d(x,x')$. This implies that for every $n \in \N$, there exists a $\kappa_n \in \N$ such that
$$\|\eta(x,y) - \eta(x',y)\|_1 < \frac{1}{n}$$
for all $x,x',y \in V$ with $d(x,x') \leq n$ and $d(x,y) \geq \kappa_n$.

Fix a base point $x_0 \in V$ and define the continuous map $\gamma : G \recht \Prob(V) = \gamma(g) = \eta(x_0 , g \cdot x_0)$. Since
$$\gamma(gh) = \eta(x_0, gh \cdot x_0) = g \cdot \eta(g^{-1} \cdot x_0, h \cdot x_0) \quad\text{and}\quad g \cdot \gamma(h) = g \cdot \eta(x_0, h \cdot x_0) \; ,$$
we find that $\lim_{h \recht \infty} \|\gamma(gh) - g \cdot \gamma(h)\|_1 = 0$ uniformly on compact sets of $g \in G$. Since
$$\gamma(hk) = \eta(x_0,hk \cdot x_0) = \eta(hk \cdot x_0, x_0) = h \cdot \eta(k \cdot x_0, h^{-1} \cdot x_0) \quad\text{and}\quad \gamma(h) = h \cdot \eta(x_0, h^{-1} \cdot x_0) \; ,$$
we find that
$$\|\gamma(hk) - \gamma(h)\|_1 = \|\eta(k \cdot x_0, h^{-1} \cdot x_0) - \eta(x_0, h^{-1} \cdot x_0)\|_1 \; ,$$
so that also $\lim_{h \recht \infty} \|\gamma(hk) - \gamma(h)\|_1 = 0$ uniformly on compact sets of $k \in G$.

As in the proof of 2, there exists a $G$-equivariant isometric map $\Prob(V) \recht \cS(G)$, so that $G$ has property~(S).

5.\ By \cite{CH88}, $G$ is weakly amenable. By \cite[Proof of Th\'{e}or\`{e}me 4.4]{Sk88}, $G$ has property~(S).
\end{proof}

In the proof of Proposition \ref{prop.groups-prop-S}, we used the following stability result for property~(S). One can actually prove that property~(S) is stable under measure equivalence of locally compact groups, but for our purposes, the following elementary lemma is sufficient.

\begin{lemma}\label{lem.stability}
Let $G$ be a locally compact group and $K < G_0 < G$ closed subgroups such that $K$ is compact and normal in $G_0$, and $G_0$ is open and of finite index in $G$. If $G_0/K$ has property~(S), then also $G$ has property~(S).
\end{lemma}
\begin{proof}
Since $L^1(G_0/K) \subset L^1(G_0)$, we have a $G_0$-equivariant map $\cS(G_0/K) \recht \cS(G_0)$. So, property~(S) for $G_0/K$ implies property~(S) for $G_0$. Write $G$ as the disjoint union of $G_0 g_i$, $i=1,\ldots,n$. Define the continuous map $\pi : G \recht G_0$ given by $\pi(g g_i) = g$ for all $g \in G_0$ and $i \in \{1,\ldots,n\}$. View $\cS(G_0) \subset \cS(G)$. Given $\eta_0 : G_0 \recht \cS(G_0)$ as in the definition of property~(S), define the continuous map
$$\eta : G \recht \cS(G) : \eta(g) = \frac{1}{n} \sum_{i=1}^n g_i^{-1} \cdot \eta_0(\pi(g_i g)) \; .$$
One checks that $\lim_{h \recht \infty} \|\eta(ghk) - g \cdot \eta(h)\|_1 = 0$ uniformly on compact sets of $g,k \in G$. So, $G$ again has property~(S).
\end{proof}

\end{document}